\documentclass[11pt,a4]{article} 
\usepackage[hmargin=3cm,vmargin=2cm]{geometry}

\usepackage{mathtools}						
\usepackage{enumerate}						
\usepackage{amsthm,amsmath,amsfonts,amsthm,amssymb}
\usepackage[all]{xy}
\usepackage{hyperref}													
\makeatletter														
\newcommand{\labitem}[2]{%
\def\@itemlabel{(\text{#1})}												
\item 															
\def\@currentlabel{#1}\label{#2}} 											
\makeatother 														

\usepackage{booktabs} 
\usepackage{array} 
\usepackage{paralist} 
\usepackage{verbatim} 
\usepackage{subfig} 

\usepackage[mathscr]{euscript}						 	%
\usepackage{dsfont}
\usepackage{upgreek}

\def\mbb{\mathbb}

\def\mscr{\mathscr}
\def\mcal{\mathcal}
\def\mfrak{\mathfrak}

\title{On the transfer congruence between~$p$-adic Hecke~$L$-functions}
\author{Dohyeong Kim}

\def\Q{\mathds{Q}}
\def\R{\mathds{R}}
\def\Z{\mathds{Z}}

\def\C{\mathds{C}}

\def\Gal{\mathrm{Gal}}

\def\GL{\mathrm{GL}}
\def\SO{\mathrm{SO}}
\def\det{\mathrm{det}}

\def\adele{\mathds A}

\def\bs{\boldsymbol}
\def\vol{\mathrm{Vol}}
\def\val{\mathrm{val}}
\def\rec{\mathrm{rec}}
\def\ver{\mathrm{ver}}

\def\Eul{\mathrm{Eul}}

\numberwithin{equation}{subsection}

\theoremstyle{plain}
\newtheorem{theorem}[equation]{Theorem}
\newtheorem{lemma}[equation]{Lemma}
\newtheorem{proposition}[equation]{Proposition}
\newtheorem{corollary}[equation]{Corollary}

\theoremstyle{remark}
\newtheorem{remark}[equation]{Remark}

\theoremstyle{definition}
\newtheorem{definition}[equation]{Definition}

\begin{document}

\maketitle

\begin{abstract}
We prove the transfer congruence between $p$-adic Hecke $L$-functions for CM fields over cyclotomic extensions, which is a non-abelian generalization of the Kummer's congruence. The ingredients of the proof include the comparison between Hilbert modular varieties, the $q$-expansion principle, and some modification of Hsieh's Whittaker model for Katz' Eisenstein series. As a first application, we prove explicit congruence between special values of Hasse-Weil $L$-function of a CM elliptic curve twisted by Artin representations. As a second application, we prove the existence of a non-commutative $p$-adic $L$-function in the algebraic $K_1$-group of the completed localized Iwasawa algebra.
\end{abstract}

\tableofcontents

\section{Introduction}
One of the major themes of number theory is the relationship between the critical values of~$L$-functions of motives and the arithmetic of these motives. Iwasawa theory asserts, among many other things, that there is an equality between the two, provided we vary both in suitable~$p$-adic families. A precise formulation of such an equality is called the main conjecture of Iwasawa theory.
\par
In traditional Iwasawa theory one works with a prime~$p$ which is ordinary for the given motive, and an abelian~$p$-adic Lie extension~$H_\infty$ of a number field~$H$. Let~$\mcal G$ be the Galois group~$\Gal(H_\infty / H)$. Given a motivic~$L$-function, the suitably normalized critical values of its twist by characters of~$\mcal G$ are expected to admit a~$p$-adic analytic interpolation, and a~$p$-adic~$L$-function is a~$p$-adic analytic function defined on a certain domain that achieves this desired interpolation. The Iwasawa main conjecture in this situation is then the assertion that such~$p$-adic~$L$-function is equal to a suitably defined characteristic element of the Selmer group of the motive over~$H_\infty$.
\par
Non-commutative Iwasawa theory proposed by \cite{CFKSV, Fukaya-Kato} generalizes traditional Iwasawa theory in order to enable one to work with a possibly non-commutative~$p$-adic Galois group~$\mcal G$. In this case, a characteristic element of the Selmer group of the motive over~$H_\infty$ lies in the algebraic~$K_1$ group of a suitably localized version of the Iwasawa algebra of~$\mcal G$. Hence we expect the~$p$-adic~$L$-function to be inside the same~$K_1$-group, and, at the same time, to interpolate suitably normalized~$L$-values. We call such~$p$-adic~$L$-function commutative (resp. non-commutative) if~$\mcal G$ is commutative (resp. non-commutative). Although we have a plenty of examples of commutative~$p$-adic~$L$-functions, no non-commutative examples have yet been proven to exist for motives other than the Tate motive. For the Tate motive and a totally real~$H_\infty$, Kakde and Ritter-Weiss independently established the existence of the non-commutative~$p$-adic~$L$-function and proved the associated main conjecture, assuming the vanishing of certain~$\mu$-invariants. Their result is unconditional when the base field~$H$ is abelian over~$\Q$, since the theorem of Ferrero-Wanshington proves the vanishing of~$\mu$-invariants for such fields.
\par
One of the key aspects of the work of Kakde and Ritter-Weiss is to formulate and prove non-commutative generalizations of the Kummer congruence, following an idea of Kato \cite{Kato}. Specifically, they list a set of congruences between commutative~$p$-adic~$L$-functions implied by the existence of a non-commutative~$p$-adic~$L$-function, and show that in certain cases those congruence in turn imply the existence of a non-commutative~$p$-adic~$L$-function. These predicted congruences between commutative~$p$-adic~$L$-functions are what we shall call non-commutative congruences.
\par
The principal result of the present paper is Theorem~\ref{thm:main}, where we prove the transfer congruence between commutative~$p$-adic Hecke~$L$-functions, which is one of the non-commutative congruences. By~$p$-adic Hecke~$L$-function, we mean the~$p$-adic~$L$-function interpolating Hecke~$L$-values of~$p$-ordinary CM fields which was first constructed by Katz, and later studied by Hida-Tilouine and Hsieh. Our method heavily depends on the work of Hsieh, in which the relevant Eisenstein series are constructed using Whittaker models. This is a key difference between Hsieh's work and the earlier ones. The theory of Whittaker models reduces the computation of Fourier expansion to purely local problems, and our proof crucially depends on this aspect of Hsieh's work. Furthermore, this is not only a reconstruction of the same Eisenstein series, but the resulting~$p$-adic~$L$-functions in fact differ by certain factor which is not always a~$p$-adic unit. It seems to us that Hsieh's method leads to the right measure for the main conjecture. This difference is crucial for us as this factor has a large~$p$-adic valuation for the cases that are studied in the current paper, while it is a~$p$-adic unit in the earlier work \cite{Hsieh nonvanishing} of Hsieh. We explain more about this in Subsection~\ref{ss:0308}.
\par
As a consequence of the transfer congruence, we will prove some results on the existence of non-commutative~$p$-adic~$L$-functions for some~$p$-adic Galois groups~$\mcal G$, assuming certain~$\mu$-invariants vanish. The argument is almost identical to \cite{Kakde 11}, but our result is conditional since we do not know the validity of the assumption on the vanishing of~$\mu$-invariants. Such a group~$\mcal G$ shall be of the form
\begin{align}\label{eq:FT}
\mcal G = \lim_{\leftarrow r} \Gal(\mscr M(\mu_{p^r}, m^{1/p^r})/ \mscr M)
\end{align}
where~$\mscr M$ is an imaginary quadratic field in which~$p$ splits,~$\mu_{p^r}$ is the group of~$p^r$-th roots of unity, and~$m$ is an integer. To avoid the trivial cases in which~$\mcal G$ becomes abelian, we will always assume that~$m$ is a~$p$-th power free integer, and~$m$ is neither~$0,1,$ or~$-1$. This type of extension is often called a false Tate curve extension. Unfortunately, even with the hypothesis on the vanishing of~$\mu$-invariants, we need to consider the completion of the localized Iwasawa algebra and take its~$K_1$-group, in order to produce a non-commutative~$p$-adic~$L$-function in it. If we assume the commutative main conjectures in addition, then our method will produce the non-commutative~$p$-adic~$L$-function which satisfy the main conjecture. For the details, see Section~\ref{s5}.
\par
We briefly outline the contents of the paper. In Section~\ref{s2}, we review the theory of Hilbert modular Shimura varieties and Hilbert modular forms. It is important to choose a good diagonal embedding between Hilbert modular Shimura varieties, so that we have controls over the CM points and the cusps. In Section~\ref{s3}, we review the work of Hsieh. Although it is clear from Hsieh's computation, it was not explicitly pointed out in his papers that in certain cases it is possible to improve the~$p$-integrality of~$p$-adic~$L$-function constructed by Katz and Hida-Tilouine. We explain how to modify Hsieh's construction to obtain a~$p$-integrally optimal measure. Also, we reformulate the work of Hsieh in the language of~$\Lambda$-adic forms. Although we do not use any deep result about~$\Lambda$-adic forms, the approach using~$\Lambda$-adic forms is more convenient when we prove the transfer congruence. Section~\ref{s4} contains the proof of the transfer congruences under the hypotheses \eqref{P} and \eqref{C}, and proof of these hypotheses for the~$p$-ordinary CM fields of the form~$\Q(\sqrt{-D},\mu_{p^r})$ with~$r\ge 1$,~$D<0$. In Section~\ref{s5}, we explain the consequences on the existence of non-commutative~$p$-adic~$L$-functions and the associated non-commutative Iwasawa main conjecture, as well as the congruence between special values of $L$-functions in concrete terms.


\section{Hilbert modular forms and~$q$-expansion principle}\label{s2}
In this section, we review Hilbert modular forms,~$\Lambda$-adic forms, and the~$q$-expansion principle. Also, we specify CM points on the Hilbert Shimura variety following \cite{Hsieh mu}, and study the behavior of those CM points under the diagonal embedding.

\subsection{Hilbert modular forms}
Let~$\Q$ be the field of rational numbers with a fixed algebraic closure~$\overline \Q$. By a number field, we mean a subfield of~$\overline \Q$ which is a finite extension of~$\Q$. We write the ring of adeles associated to a number field~$H$ by~$\adele_{H}$, and use the convention that if~$\bullet$ is a set of places of~$H$, a place of~$H$, or a product of these two, then~$\adele_H^\bullet$ denotes the ring of adeles whose component at the place dividing~$\bullet$ equals~$1$. For example,~$\adele_H^\infty$ denotes the ring of finite adeles. Also, if~$g$ is an element of~$\GL_2(\adele_H)$ then we use similar convention for~$g^\bullet$, and~$g_\bullet$ denotes the product of all~$v$-component of~$g$ for~$v$ dividing~$\bullet$.
\par
Let~$\mscr F$ be a totally real field, and we consider the algebraic group~$\GL_2/\mscr F$. Let~$I$ be the set of embedding of~$\mscr F$ into the field of real numbers~$\R$. Let~$\mfrak H$ be the complex upper half plane, and let
\begin{align}\label{eq:2}
X_+ =\mfrak H^I= \{(\tau_\sigma)_{\sigma\in I}| \tau_\sigma \in \mfrak H \}.
\end{align}
Let~$\GL_2(\R)_+$ be the subgroup of~$\GL_2(\R)$ whose elements have positive determinant. Then~$\GL_2(\R)_+$ acts on~$\mfrak H$ by the linear fractional transformation
\begin{align}
\begin{bmatrix}
a 	& 	b
\\
c 	& 	d
\end{bmatrix}
\cdot \tau = \frac{a\tau+b}{c\tau+d},\hspace{5mm}
\begin{bmatrix}
a 	& 	b
\\
c 	& 	d
\end{bmatrix}
\in \GL_2(\R)_+,\,\,\tau\in\mfrak H.
\end{align}
Using the identification~$\mscr F \otimes _\Q \R = \R^I$, we identify~$\GL_2(\mscr F \otimes _\Q \R)=\GL_2(\R)^I$, and define~$\GL_2(\mscr F \otimes _\Q \R)_+=\GL_2(\R)_+^I$. In particular,~$\GL_2(\R)_+^I$ acts on~$X_+$ by componentwise linear fractional transformation. Let~$\Z[I]$ be the free abelian group generated by~$I$. A typical element~$m\in \Z[I]$ is written as~$\sum_{\sigma \in I} m_\sigma \sigma$. We often write~$I = \sum_{\sigma \in I} \sigma$. For~$g \in \GL_2(\R)$ and~$\tau \in \mfrak H$, we put
\begin{align}
J(g, \tau) = c\tau+d.
\hspace{5mm}
g=\begin{bmatrix}
a 	& 	b
\\
c 	& 	d
\end{bmatrix}
\in \GL_2(\R),\,\,\tau \in\mfrak H,
\end{align}
and for~$g_\infty \in \GL_2(\R)^I$ and~$\tau \in X_+$, then we put
\begin{align}
J(g_\infty ,\tau)^{m} = \prod_{\sigma \in I} J(g_\sigma, \tau_\sigma)^{m_\sigma}.
\end{align}
\par
Now we define the space of Hilbert modular forms. Let~$k$ be an integer and~$K$ be an open compact subgroup of~$\GL_2(\adele_\mscr F^\infty)$. Let~$\C$ denote the field of complex numbers. Let~$\GL_2(\mscr F)_+$ be the subgroup of~$\GL_2(\mscr F)$ which consists of the elements with totally positive determinant. Then~$\GL_2(\mscr F)_+$ acts on~$X_+ \times \GL_2(\adele^\infty_\mscr F)$ from the left,~$\GL_2(\adele^\infty_\mscr F)$ acts on it from the right, and the two actions commute. Precisely speaking, for~$\alpha \in \GL_2(\mscr F)_+$,~$u \in \GL_2(\adele_\mscr F^\infty)$, and~$(z,g^\infty) \in X_+ \times \GL_2(\adele^\infty_\mscr F)$, we define the action to be~$\alpha(z, g^\infty ) u = (\alpha z, \alpha^\infty g^\infty u)$. 
\begin{definition}
Define~$M_k(K,\C)$ to be the space of functions~$f\colon X_+ \times \GL_2(\adele^\infty_\mscr F) \to \C$ such that
\begin{enumerate}
\item for each~$g^\infty \in \GL_2(\adele_\mscr F^\infty)$,~$f(\,\cdot\,, g^\infty) \colon X_+ \to \C$ is holomorphic,
\item for each~$\alpha \in \GL_2(\mscr F)_+$ and~$u \in K$,~$f$ satisfies
\begin{align}
f(\alpha(z, g^\infty ) u ) = J(\alpha, z)^{kI} f(z,g^\infty).
\end{align}
\end{enumerate}
When~$\mscr F=\Q$, we add a growth condition at cusps in the definition of~$M_k(K,\C)$. The growth condition is automatically satisfied if~$\mscr F \not = \Q$ by the Koecher principle, which is our main interest in the current article.
\end{definition}
\subsection{$q$-expansion of Hilbert modular forms}\label{ss:22}
We first introduce some notation. For a number field~$H$, let~$\mscr O_H$ be the ring of integers of~$H$, and~$\mfrak d_H$ the absolute different of~$H$. For our fixed totally real field~$\mscr F$, we simply write~$\mscr O = \mscr O_\mscr F$ and~$\mfrak d = \mfrak d_\mscr F$. For a nonzero fractional ideal~$\mfrak a$ of~$\mscr F$, we write~$\mfrak a^* = \mfrak a^{-1} \mfrak d^{-1}$. We consider the algebraic group~$\GL_2/\mscr F$, and for any~$\mscr F$-algebra~$R$ we denote by~$\GL_2(R)$ the group of invertible~$2\times 2$ matrices over~$R$. Let~$V=\mscr F \oplus \mscr F$ be the two dimensional vector space over~$\mscr F$, which we regard as the space of row vectors of length two. For an~$\mscr F$-algebra~$R$, we let~$V_R=V\otimes_\mscr F R$, on which~$\GL_2(R)$ acts by multiplication from the right. We also define an action of~$\GL_2(R)$ on~$V_R$ from the left by
\begin{align}\label{eq:07}
g*x = x \cdot g^{-1}\det (g)
\end{align}
for~$g \in \GL_2(R)$ and~$x \in V_R$. The reason for defining the above left action in this way originates from Subsection~2.4. of \cite{Hsieh mu}, where the moduli problem for the Shimura variety is described using this left action.
\par
Let~$e_1=(1,0)$ and~$e_2=(0,1)$ be vectors in~$V$ so that~$V = \mscr F e_1 \oplus \mscr F e_2$. We denote by~$L$ the canonical lattice~$\mscr O e_1 \oplus \mscr O^* e_2$ of~$V$. For a positive integer~$N$, we define
\begin{align}
U(N) :=\{g^\infty \in \GL_2(\adele_\mscr F^\infty ) | g^\infty L \subset L,\,\, g \equiv 1 \hspace {3mm} \text{(mod~$N\cdot L$)}\}.
\end{align}
Let~$K$ be an open compact subgroup of~$\GL_2(\adele_\mscr F^\infty)$ which is of the form
\begin{align}
K = \prod_{v<\infty} K_v
\end{align}
where~$v$ runs over the set of finite places of~$\mscr F$ and~$K_v$ is an open compact subgroup of~$\GL_2(\mscr F_v)$. Define
\begin{align}
K_p = \prod_{v| p } K_v
\end{align}
to be the product of~$K_v$ with~$v|p$.
We will always assume that~$K \supset U(N)$ for some~$N$, and that~$K_p$ is equal to~$K_p^0$ which we define as
\begin{align}
K_p^0 := \{g_p \in \GL_2(\mscr F_p)| g_p * (L \otimes _\mscr O \mscr O_p )=L \otimes _\mscr O \mscr O_p  \}.
\end{align}
For such an open compact subgroup~$K \subset \GL_2(\adele_\mscr F^\infty)$ we define
\begin{align}\label{eq:10}
K_1^n = \{g^\infty \in K | g_p \equiv 
\begin{bmatrix}
1	&	*
\\
0	&	1
\end{bmatrix}
\hspace{5mm}\text{(mod~$p^n$)} \}.
\end{align}

Let~$\mfrak c$ be a non-zero fractional ideal of~$\mscr F$ prime to~$pN$. For each place~$v$ of~$\mscr F$ prime to~$pN$, choose~$\bs c_v \in \mscr F_v$ such that~$\bs c_v \mscr O _v= \mfrak c _v$. For~$v\mid pN$, let~$\bs c_v = 1$. Write~$\bs c = \prod_v \bs c_v \in \adele_\mscr F^\infty$. For a fractional ideal~$\mfrak a$ of~$\mscr F$, denote by~$\mfrak a_+$ the subset of~$\mfrak a$ consisting of totally positive elements. For any ring~$R$ and a finitely generated semigroup~$L\subset \mscr F$, we write~$R[[q^{L}]]$ for the set of formal power series
\begin{align}\label{eq:225}
\sum_{\beta \in L_{\ge 0 }} a_\beta q^\beta
\end{align}
where~$L_{\ge 0 }=\{0\}\cup (L \cap \mscr F_+)$. We view~$R[[q^{L}]]$ as a ring of formal power series, where the rule of associative and distributive multiplication is characterized  by~$q^{\beta_1} \cdot q^{\beta_2}= q^{\beta_1+\beta_2}$ for~$\beta_1,\beta_2\in \mscr F_+$. However, there is an exception to this, where we interpret~$q^\beta$ as a function. Given a holomorphic Hilbert modular form~$f \in M_k(K,\C)$, we consider the restriction of~$f$ to the component~$X_+ \times
\begin{bsmallmatrix}
1 	&	0
\\
0		&	\bs c^{-1}
\end{bsmallmatrix}
$, written as~$f|_{(\mscr O, \mfrak c^{-1})}$. Then~$f|_{(\mscr O, \mfrak c^{-1})}$ has Fourier expansion
\begin{align}\label{eq11}
f|_{(\mscr O, \mfrak c^{-1})}(q) =
\sum_{\beta\in (N\mfrak c)_{\ge 0}^{-1}} a_\beta(f, \mfrak c) q^\beta
\end{align}
where we interpret~$q^\beta$ as a function on~$X_+$ defined by
\begin{align}
q^\beta\colon \tau \mapsto \mathrm{exp}(\sum_{\sigma \in I} \sigma (\beta) \tau_\sigma 2\pi \sqrt{-1}).
\end{align}
The coefficients~$a_\beta(f, \mfrak c)$ is independent of the choices for~$\bs c_v$, and \eqref{eq11} is called~$q$-expansion of~$f$ at the cusp~$(\mscr O, \mfrak c^{-1})$. 

\par
By the~$q$-expansion principle, for a subring~$R$ of~$\C$, we may define
\begin{align}\label{eq:228}
M_k(\mfrak c, K, R) := \{f \in M_k(K ,\C)| f|_{(\mscr O, \mfrak c^{-1})} \in R[[q^{(N\mfrak c)^{-1}}]]\}.
\end{align}

\par
Although this interpretation of~$q^\beta$'s as variables in the Fourier expansion is important for our later computation, we keep viewing~$R[[q^L]]$ as a ring of formal power series because such an interpretation is not available when we consider~$R[[q^L]]$ for a ring~$R$ that is not contained in~$\C$.

\subsection{$\Lambda$-adic Hilbert modular forms}\label{ss:pmf}
We fix an algebraic closure~$\overline \Q_p$ of~$\Q_p$ and let~$\C_p$ be the completion of~$\overline \Q_p$. Let~$\mcal I$ be the ring of Witt vectors of an algebraic closure of the finite field of~$p$ element, and we fix an embedding of~$\mcal I$ into~$\C_p$. We will always regard~$\mcal I$ as a subring of~$\C_p$. Also, we fix embeddings~$i_\infty \colon \overline \Q \hookrightarrow \C$ and~$i_p \colon \overline \Q \hookrightarrow \C$. 
\par
We say that a ring~$R$ is a~$p$-adic ring if the natural map~$R \to \lim R/p^nR$ is an isomorphism, where the inverse limit is taken over the set of all positive integers~$n$. Of course, a~$p$-adic ring~$R$ is a~$\Z_p$-algebra. For a fractional ideal~$\mfrak c$ of~$\mscr F$ prime to~$pN$, we would like to define the space of~$R$-adic Hilbert modular forms which we will denote by~$V(\mfrak c, K ; R)$. Consider a formal power series
\begin{align}
f = \sum_{\beta \in (N\mfrak c)^{-1}_{\ge 0}}  a_\beta (f) q^\beta \in R[[q^{(N\mfrak c)^{-1}}]].
\end{align}
For any~$\Z_p$-algebra homomorphism
\begin{align}
\alpha \colon R \to \C_p
\end{align}
define
\begin{align}
\alpha (f) = \sum_{\beta \in (N\mfrak c)^{-1}_{\ge 0}} \alpha ( a_\beta(f))q^\beta \in \C_p[[q^{(N\mfrak c)^{-1}}]].
\end{align}
We say a family~$Y$ consisting of~$\Z_p$-algebra homomorphisms~$\alpha\colon R \to \C_p$ is Zariski dense if an element~$\mcal M \in R$ is determined by~$\alpha(R)$ for~$\alpha \in Y$.
\begin{definition}\label{def:pmf}
We say~$f$ is an~$R$-adic Hilbert modular form, or simply an~$R$-adic modular form, if there is a Zariski dense set~$Y$ for~$R$ such that for every~$\alpha \in Y$ we have
\begin{align}
\alpha (f) = \sum_{\beta \in (N\mfrak c)^{-1}_{\ge 0}} i_p(a_\beta(g_\alpha))q^\beta
\end{align}
for some element~$g_\alpha$ of~$M_k(\mfrak c, K_1^n, \overline \Q)$, where~$K_1^n$ is an open compact subgroup of~$\GL_2(\adele^\infty_{\mscr F})$ defined in \eqref{eq:10}. Here~$k$ and~$n$ may vary with~$\alpha$. The space of~$R$-adic modular form is denoted by~$V_0(\mfrak c, K; R)$.
\end{definition}
\begin{remark}\label{rmk:235}
When we refer to a~$p$-adic modular form, we mean an~$R$-adic modular form with an unspecified~$p$-adic ring~$R$. Since we will not use the symbol~$p$ to denote a ring, it does not cause any confusion. The definition for a~$p$-adic modular form in the style of Definition~\ref{def:pmf} is due to Hida. This definition is sufficient for our purpose for a while, but we will need a stronger notion of a~$p$-adic modular form in Section~\ref{ss:44}, which is due to Katz. See Subsection~\ref{ss:42} for the comparison of two notions.
\end{remark}
\par
The main example of such an~$R$ for us is certain Iwasawa algebra which we now explain. Let~$\mscr M$ be a quadratic totally imaginary field extension of~$\mscr F$. We will always assume that every place of~$\mscr F$ lying over~$p$ splits in~$\mscr M$. Such number field~$\mscr M$ is called~$p$-ordinary. We choose an integral ideal~$\mfrak C$ of~$\mscr M$ which is prime to~$p$, and consider the ray class field~$\mscr M(p^n\mfrak C)$ of~$\mscr M$ modulo~$p^n\mfrak C$ for all~$n\ge 0$. We define
\begin{align}\label{eq:2.3.7}
Z = Z(\mfrak C) = \lim_{\infty\leftarrow n} \Gal(\mscr M(p^n \mfrak C) / \mscr M),
\end{align}
and consider the Iwasawa algebra
\begin{align}\label{eq:20}
\Lambda = \Lambda(Z) = \lim_{\leftarrow} \mcal I  [Z / Y]
\end{align}
where the inverse limit is taken over the set of all open normal subgroups~$Y$ of~$Z$. We will be considering~$\Lambda$-adic Eisenstein series lying in~$V_0(\mfrak c, K; \Lambda)$, where~$\Lambda$ is given by \eqref{eq:20}.


\subsection{CM points}\label{ss:cmp}
In this subsection, we will introduce the notion of CM points in~$\GL_2(\adele_\mscr F)$ and the level structure on them, where~$\mscr F$ is a totally real field and~$\adele_\mscr F$ is the ring of adeles of~$\mscr F$. Instead of delving into the most general definition of a CM point, we will specify the CM points which will be used later. Roughly speaking, choosing a CM point amounts to choosing a suitable totally imaginary quadratic extension~$\mscr M$ of~$\mscr F$ and an element~$\delta \in \mscr M$ satisfying certain conditions which is listed below, and choosing a level structure amounts to making some explicit choice of a base for each place of~$\mscr F$. When we specify the CM points and their level structure, we will follow the convention of \cite{Hsieh mu} closely. However, we note, in contrast to the previous works on~$p$-adic Hecke~$L$-functions, including \cite{Hsieh mu}, that the proof of our main result will rely on some particular choices of~$\delta$.
\par
We begin specifying the CM points. Fix an odd prime~$p$. Let~$\mscr M$ be a totally imaginary quadratic extension of~$\mscr F$. We say that~$\mscr M$ is~$p$-ordinary if every place of~$\mscr F$ dividing~$p$ splits in~$\mscr M$. Let~$c$ be the complex conjugation, which also defines a canonical nontrivial involution~$\Gal(\mscr M/\mscr F)$. A CM type of~$\mscr M$ is a set of embeddings~$\sigma\colon \mscr M \to \C$ such that the restricting~$\sigma$ to~$\mscr F$ induces a bijection~$\Sigma \cong I$, where we recall that~$I$ was defined to be the set of embedding of~$\mscr F$ into~$\R$. We define the~$p$-adic CM type~$\Sigma_p$ associated to~$\Sigma$ to be the set of places induced by~$i_p \circ i_\infty^{-1}\circ \sigma$ for~$\sigma \in \Sigma$. From the definition of~$p$-ordinary CM fields, it follows that the restriction induces a bijection between~$\Sigma_p$ and the set of places of~$\mscr F$ lying above~$p$. Let~$\mscr R$ be the ring of integers of~$\mscr M$, and let~$\mfrak d$ be the absolute different of~$\mscr F$. We choose an element~$\delta_\mscr M \in \mscr M$ which is often written as~$\delta$ when a reference to~$\mscr M$ is unimportant, satisfying the following three conditions:-
\begin{enumerate}
\labitem{pol-1}{pol-1} For each~$\sigma \in \Sigma$, we have~$\mathrm{Im}(\sigma(\delta))>0$.
\labitem{pol-2}{pol-2} For each~$\sigma \in \Sigma$, we have~$\delta^c=-\delta$.
\labitem{pol-3}{pol-3} Let~$\mathrm{Tr}_{\mscr M/\mscr F}\colon \mscr M \to \mscr F$ be the trace map. Then the alternating pairing~$$(x,y)\mapsto 
\mathrm{Tr}_{\mscr M /\mscr F}
\left(
\frac
{xy^c}{2\delta}
\right)$$
defines an isomorphism~$\mscr R\wedge_\mscr O \mscr R \tilde\to (\mfrak c \mfrak d)^{-1}$ for an integral ideal~$\mfrak c\subset \mscr O$ that is prime to~$p$.
\end{enumerate}
The above three conditions are called polarization conditions, as suggested by the labels. For such~$\delta$, we let~$D=-\delta^2$ which is a totally positive element in~$\mscr F$. Of course,~$D$ depends on the choice of~$\delta$, but we often suppress it from the notation.
\par
We remark that in the literature the choice of~$\delta$ is often referred to as the choice of polarization and the resulting~$\mfrak c$ is called the polarization ideal. We note that having a polarization ideal which is prime to~$p$ is crucial for establishing the~$q$-expansion principle for~$p$-adic modular forms on which the proof of our main theorem is based on. On the other hand, it will be clear from the computations that we will carry out why we need to impose the first two conditions on~$\delta$. See the paragraph after \eqref{P} in Subsection~\ref{ss:rad} for the choice of~$\delta$ we make relative to a smaller CM field~$\mscr M'$ of~$\mscr M$.
\par
From now, we explain what we call the level structure. Let~$V$ be the two dimensional vector space~$\mscr F\oplus \mscr F$ over~$\mscr F$, viewed as the set of row vectors of length two with entries in~$\mscr F$. In particular, an element~$x$ in~$V$ is written as a row vector~$(a,b)$ with~$a,b\in \mscr F$. Given~$\delta$, an element~$y\in \mscr M$ can be written uniquely in the form~$a\delta + b$, and we introduce an isomorphism~$q_\delta\colon \mscr M \tilde\longrightarrow V$ defined by
\begin{align}
q_\delta\colon a\delta + b\mapsto (a,b).
\end{align}
Of course, it would suffice to choose any~$\delta$ in~$\mscr M - \mscr F$ in order to conclude that~$q_\delta$ is an isomorphism, and the above  sophisticated conditions on~$\delta$ are not used here. In any case, for our later purpose, we record the dependence of the isomorphism~$q_\delta$ on the choice of~$\delta$ in the following lemma.
\begin{lemma}\label{lem:lem1}
Let~$u$ be an element in~$\mscr F$ and~$\tilde \delta$ be an element in~$\mscr M-\mscr F$. We have 
\begin{align}
q_{u\tilde \delta}(a\tilde\delta+b)
\cdot
\begin{bmatrix}
u	&	0\\
0		&	1
\end{bmatrix}
 =q_{\tilde\delta}(a\tilde\delta +b).
\end{align}
\end{lemma}
\begin{proof}
The proof is a straightforward calculation. We have
\begin{align}
q_{u\tilde\delta}(a\tilde\delta+b)
&= q_{u\tilde\delta}(au^{-1}u\tilde\delta+b)
\\
&= au^{-1}(1,0) + b(0,1).
\end{align}
Multiplying~$\begin{bmatrix}
u	&	0\\
0		&	1
\end{bmatrix}$ on the both sides from the right yields
\begin{align}
q_{u\tilde\delta}(a\tilde\delta+b) \cdot 
\begin{bmatrix}
u	&	0\\
0		&	1
\end{bmatrix}
=a(1,0) + b(0,1),
\end{align}
whose right hand side equals~$q_{\tilde\delta}(a\tilde\delta +b)$ by the definition of~$q_{\tilde\delta}$. The proof of the lemma is complete.
\end{proof}
Let~$M(2,\mscr F)$ be the algebra of~$2 \times 2$ matrices and denote by~$\varrho$ the embedding
\begin{align}\label{eq:26}
\varrho \colon \mscr M \hookrightarrow M(2,\mscr F),\,\,\,\varrho\colon a\delta+b\mapsto \begin{bmatrix}b&-Da\\a&b\end{bmatrix}.
\end{align}
When we need to emphsize the dependence of~$\varrho$ on the choice of~$\delta$, we will write~$\varrho_\delta$ instead of~$\varrho$. Clearly, the first two conditions imposed on~$\delta$ are indispendable in order to have~$\varrho$, while the third condition is irrelevant here.
\par
Let~$\mfrak C$ be an integral ideal of~$\mscr R$ which is prime to~$p$. For an ideal~$\mfrak A$ of~$\mscr M$, denote by~$\mfrak A^c$ the complex conjugate of~$\mfrak A$. We decompose
\begin{align}\label{eq:b1}
\mfrak C= \mfrak F \mfrak F_c \mfrak I
\end{align}
for ideals~$\mfrak F$,~$\mfrak F_c$, and~$\mfrak I$ of~$\mscr M$ so that~$\mfrak F$ is a product of prime ideals split over~$\mscr F$,~$\mfrak F^c \subset \mfrak F_c$, and~$\mfrak I$ is a product of prime ideals ramified or inert over~$\mscr F$. Such a decomposition is unique. Also, we write
\begin{align}
\mfrak C^+ := \mfrak F \mfrak F_c
\end{align}
and
\begin{align}
\mfrak D : = p \mfrak C D_{\mscr M / \mscr F}
\end{align}
where~$D_{\mscr M / \mscr F}$ is the relative discriminant for~$\mscr M/ \mscr F$. For a finite place~$v$ of~$\mscr F$, we define
\begin{align}
\mscr R_v = \mscr R \otimes _\mscr O \mscr O_v.
\end{align}
We are going to choose certain~$\mscr O_v$-basis for~$\mscr R_v$, which will be used to define the level structure. We divide it into several cases. 
\par
First, let~$v$ be a place of~$\mscr F$ which divides~$p\mfrak F \mfrak F_c$. Then,~$v$ decomposes in~$\mscr M$ into two places in~$\mscr M$. We define~$w$ to be the one of the two such that~$w\in \Sigma_p$ or~$w|\mfrak F$, recalling that~$\Sigma_p$ was defined in the beginning of this subsection to be the set of places above~$p$ corresponding to~$\Sigma$. We define~$\overline w$ to be the unique place of~$\mscr M$ which is different from~$w$ and lies over~$v$. In particular,~$\mscr R_v$ is canonically isomorphic to~$\mscr R_w \oplus \mscr R_{\overline w}$, which we view as an identification without further indication. Let~$e_w$ (resp.~$e_{\overline w})$ be the idempotent in~$\mscr R_w \oplus \mscr R_{\overline w}$ corresponding to~$w$ (resp.~$\overline w$). Then~$\{e_w,e_{\overline w}\}$ is an~$\mscr O_v$-basis for~$\mscr R_v$.
\par
Let~$v$ be a place of~$\mscr F$ which does not divide~$p\mfrak C$ and split in~$\mscr M$. As in the previous case,~$v$ splits into two places~$w$ and~$\overline w$ in~$\mscr M$. The difference is that we do not have a preferred choice for~$w$. We identify~$\mscr R_v$ with~$\mscr R_w \oplus \mscr R_{\overline w}$, and define~$e_w$ and~$e_{\overline w}$ as idempotents, as in the previous case.
\par
Next we consider a place~$v$ of~$\mscr F$ which is either inert or ramified in~$\mscr M$. In any case, there is a unique place~$w$ of~$\mscr M$ which lies above~$v$, and~$\mscr R_v$ is canonically identified with~$\mscr R_w$. We explain how we choose a basis~$\{1,\bs \theta_v \}$ for~$\mscr R_v$ as an~$\mscr O_v$-module. When~$v$ is ramified in~$\mscr M$, we take~$\{1, \bs \theta_v \}$ to be an~$\mscr O_v$-basis of~$\mscr R_w$ with~$\bs \theta_v$ being a uniformizer in~$\mscr R_w$. We always require~$\bs \theta_v$ to satisfy~$\bs \theta_v^c=-\bs \theta_v$ if~$v\nmid 2$. If we put~$t_v = \bs \theta_v+ \bs \theta_v^c$, then by our choice~$t_v=0$ for~$v\nmid 2$ but possibly non-zero if~$v|2$. If~$v$ is inert in~$\mscr R$ and does not divide~$2$, then we take~$\{1,\bs \theta_v\}$ to be an~$\mscr O_v$-basis of~$\mscr R_v$ such that~$\bs \theta_v^c=-\bs \theta_v$. If~$v$ is inert in~$\mscr R$ and~$v$ divides~$2$, then we relax the condition~$\bs \theta_v^c=-\bs \theta_v$.
\par
We need to make one more choice before we choose the level structure. Write~$\widehat{\Z}$ for the profinite completion of~$\Z$ and write~$\widehat{\mscr O}=\mscr O \otimes_\Z\widehat{\Z}$. We fix a finite idele~$d=d_\mscr F=(d_{\mscr F_v})_v \in \adele_\mscr F^{\infty,\times}$ such that
\begin{align}\label{eq:2.4.12}
d \widehat{ \mscr O} \cap \mscr F  = \mfrak d,
\end{align}
where~$\mfrak d$ is the absolute different of~$\mscr F$. We further assume that
\begin{align}\label{eq:2.4.13}
d_{\mscr F_v} = -2\delta_w
\end{align}
if~$v$ divides~$\mfrak F \Sigma_p$. Note that \eqref{pol-3} ensures that \eqref{eq:2.4.13} does not conflict with the condition \eqref{eq:2.4.12} we imposed on~$d$.
\par
We will now choose an~$\mscr O_v$-basis~$\{e_{1,v},e_{2,v}\}$ for~$\mscr R_v$ for each finite place~$v$ of~$\mscr F$ which we call the level structure. Recall that for an ideal~$\mfrak a \subset \mscr O$, we write~$\mfrak a^* = \mfrak a^{-1}\mfrak d^{-1}$. For a finite place~$v$ of~$\mscr F$ such that~$v\nmid p\frak C\mfrak C^c$, we want to choose~$\{e_{1,v},e_{2,v}\}$ so that~$\mscr R_v = \mscr O_v e_{1,v} \oplus \mscr O_v^* e_{2,v}$. In order to achieve this, we define
\begin{align}\label{eq:e1}
e_{1,v}=
\begin{cases}
e_{\overline w}\text{ if~$v$ splits in~$\mscr M$}
\\
\bs \theta_v\text{ if~$v$ is inert or ramified in~$\mscr M$}
\end{cases}
\end{align}
and
\begin{align}\label{eq:e2}
e_{2,v}=
\begin{cases}
d_{{\mscr F_v}}\cdot e_{w}\text{ if~$v$ splits in~$\mscr M$}
\\
d_{{\mscr F_v}}\cdot 1\text{ if~$v$ is inert or ramified in~$\mscr M$}.
\end{cases}
\end{align}
\par
Having fixed the level structure, we give the definition of the CM points in the rest of this subsection. Recall that we defined~$V$ to be the space of row vectors of length two over~$\mscr F$, and we define~$e_1=(1,0)$ and~$e_2=(0,1)$. Also recall that the choice of~$\delta$ gives rise to an isomorphism~$q_\delta$ from~$\mscr M$ to~$V$.
For any place~$v$ of~$\mscr F$, let~$q_{\delta,v}$ be the isomorphism~$\mscr M_v \tilde \rightarrow V \otimes _\mscr O \mscr O_v$ extending~$q_\delta$. Recall that we defined in \eqref{eq:07} the left action of~$\GL_2(R)$ on the row vectors, written as~$*$. If~$v$ is a finite place, then we denote by~$\varsigma_v$ the element in~$\GL_2({\mscr F_v})$ which is characterized by
\begin{align}\label{eq:varsigma}
\varsigma_v * e_i  = q_{\delta,v}({e_{i,v}}),\,\,i=1,2.
\end{align}
Unfolding the definitions, for~$v | \mfrak D$,~$\varsigma_v$ is simply
\begin{align}
\varsigma_v =
\begin{cases}
\begin{bmatrix}
d_{\mscr F_v}	& \frac{-t_v}{2} \\
0			& d_{\mscr F_v}^{-1}
\end{bmatrix}
\text{ if }v|D_{\mscr M/\mscr F} \mfrak C^-
\vspace{3mm}
\\
\begin{bmatrix}
-\delta_w	& -\frac 1 2 \\
1			& \frac{-1}{2\delta_w}
\end{bmatrix}
\text{ if }v|p\mfrak C^+\text{ and } w|\Sigma_p\mfrak F.
\end{cases}
\end{align}
\begin{remark}\label{rmk:2.4.18}
Suppose that~$\mscr M$ contains an imaginary quadratic field~$\mscr M_0$. Let~$D_0$ be the positive squarefree integer such that~$\mscr M_0 = \Q(\sqrt{-D_0})$. If~$2$ is split in~$\mscr M_0$, then all~$t_v$ can be chosen to be zero. If~$D_0$ is congruent to~$1$ modulo~$4$, then we choose
\begin{align}
\bs \theta_v =\frac{1+\sqrt{D_0}}{2}
\end{align}
for each place~$v$ of~$\mscr F$ dividing~$2$. In particular, we have~$t_v=1$ for those~$v$. If~$2$ divides~$D_0$, then we choose~$\bs \theta_v = \sqrt{-D_0}$, in which case~$t_v=0$.
\end{remark}
\begin{remark}\label{rmk:2.4.19}
Suppose that we are given a degree~$p$ extension of~$\mscr M/\mscr M$ CM fields, and let~$v'$ be a place of~$\mscr F'$. For each place~$v$ of~$\mscr F$ lying over~$v'$, we can choose~$\bs\theta_v$ to be the image of~$\bs\theta_{v'}$.
\end{remark}
For an infinite place~$v$ of~$\mscr F$, let~$\sigma \in \Sigma$ be the unique element in~$\Sigma$ extending~$v$ and define~$\varsigma_v$ to be 
\begin{align}
\varsigma_v = \begin{bmatrix}
\mathrm{Im}(\sigma(\delta)) & 0
\\
0&1
\end{bmatrix}.
\end{align}
We define
\begin{align}
\varsigma = \prod_v \varsigma_v  \in \GL_2(\adele_\mscr F),\,\,\,\text{and}\,\,\,
\varsigma_f= \prod_{v<\infty}\varsigma_v \in \GL_2(\adele_\mscr F^\infty).
\end{align}
\par
Recall that~$\mfrak H$ denotes the upper half plance and that we defined~$X_+$ to be~$\mfrak H^I$ in \eqref{eq:2}. We can naturally identify~$\mscr M \otimes _\Q \R$ with~$\C^\Sigma$, and further identify~$X_+$ as the subset of~$\C^\Sigma$ which consists of~$\Sigma$-tuple of complex numbers~$(\tau_\sigma)_{\sigma \in \Sigma}$ such that for every~$\sigma$ the imaginary part of~$\tau_\sigma$ is positive. In particulare, if we denote by~$\delta_\Sigma$ the~$\Sigma$-tuple~$(\sigma(\delta))_{\sigma\in\Sigma}$, then~$\delta_\Sigma$ is viewed as an element of~$X_+$.
We finally define the CM point associated to a finite idele~$a \in \adele_{\mscr F}^{(pN),\times}$ to be
\begin{align}\label{eq:cm}
x(a)= \left(\delta_\Sigma,\varrho(a)\varsigma_f\right) \in X_+ \times \GL_2(\adele_\mscr F^\infty).
\end{align}

\subsection{Restriction along the diagonal}\label{ss:rad}
The purpose of this subsection is to define a map which we call the diagonal map, associated to a degree~$p$ extension~$\mscr F/\mscr F'$ of totally real fields, and to study the effect of this diagonal map on the level structure which we defined in the previous subsection. As before,~$I$ denotes the set of embeedings of~$\mscr F$ into~$\R$, and similarly define~$I'$ to be the set of embeddings of~$\mscr F'$ into~$\R$. We similarly denote by~$X_+'$ the~$I'$-tuple of complex upper half plane, which is again identified as a subset of~$\C^{I'}$ as in the previous subsection. Then embedding~$\mscr F'\to \mscr F$ induces an embedding~$X_+' \to X_+$. This is called as the diagonal map in \cite{Kakde 11}, but let us call it a naive diagonal map in order to avoid the conflict with our terminology. In \cite{Kakde 11}, using the naive diagonal map, Kakde proved various non-commutative Kummer congruences for~$p$-adic Dedekind~$L$-function associated to totally real fields. However, we need to modify the naive diagonal map for our purposes because of the two reasons. Firstly, we are working with adeles instead of just upper half planes, so if we were to use the naive diagonal embedding, we need to make a non-canonical identification between adelic automorphic forms and classical automorphic forms on upper half planes. It would be technically undesirable if our computation relies on adelic techniques and we have to keep track of the non-canonical identifications that we have made. The second reason is due to the role of level structure in our work, which was irrelevant in \cite{Kakde 11}; what we need is not only some sort of diagonal map, but also the compatibility between the diagonal map and the level structure we choose. 
\par
We now explain the diagonal map in our sense. Let~$\mscr F/\mscr F'$ be a degree~$p$ extension of totally real fields and let~$\mscr M'/\mscr F'$ be a quadratic CM extension such that every prime of~$\mscr F'$ lying over~$p$ splits in~$\mscr M'$. Let~$\mscr M$ be the compositum of~$\mscr M'$ and~$\mscr F$. For a CM type~$\Sigma'$ of~$\mscr M'$, we denote by~$\Sigma$ the CM type on~$\mscr M$ induce by~$\Sigma$. It is defined to be
\begin{align}
\Sigma:=\{\sigma| \sigma\colon \mscr M \hookrightarrow \C,\,\, \sigma|_{\mscr M'}\in\Sigma'\}.
\end{align}
We introduce a condition:-
\begin{enumerate}
\labitem {P}{P} The relative different~$\mfrak d_{\mscr F/\mscr F'}$ can be generated by an element~$d_{\mscr F/\mscr F'}\in \mscr F$, which is totally positive.
\end{enumerate}
Suppose \eqref{P} holds, and let~$d_{\mscr F/\mscr F'}$ be a generator of~$\mfrak d_{\mscr F/\mscr F'}$ which is totally positive. For a given level structure associated to~$\delta' \in \mscr M'$, we define~$\delta$ to be~$d_{\mscr F/\mscr F'}\delta'~$. Then~$\delta$ induces a level structure for~$\mscr M$, and we want to compare the two level structures in terms of the diagonal embedding~$\widetilde\Delta$, which we now define. Let~$I'$ be the set of embeddings of~$\mscr F'$ into~$\R$ and let~$X_+'$ be the product of~$I'$ copies of complex upper half planes. Then we define the diagonal map, denoted by~$\widetilde\Delta_{\mscr F'}^{\mscr F}$, as
\begin{align}\label{eq:res}
\widetilde\Delta_{\mscr F'}^{\mscr F} \colon X_+' \times \GL_2(\adele_{\mscr F'}^\infty ) &\rightarrow X_+ \times \GL_2(\adele_\mscr F ^\infty) \\
(z',h'^\infty)
&\mapsto
\begin{bmatrix}
d_{\mscr F/\mscr F'}	&	0 \\
0				&	1
\end{bmatrix}
\cdot(z', h'^\infty)\cdot
\begin{bmatrix}
1	&	0 \\
0	&	d_{\mscr F/\mscr F'}^\infty
\end{bmatrix}^{-1}.
\end{align}
Precisely speaking, our map~$\widetilde\Delta_{\mscr F'}^{\mscr F}$ depends on~$d_{\mscr F/\mscr F'}$, but we suppress~$d_{\mscr F/\mscr F'}$ in order to ease the notation. When it is clear what~$\mscr F$ and~$\mscr F'$ are, we often write~$\widetilde\Delta$ instead of~$\widetilde\Delta_{\mscr F'}^{\mscr F}$.
\par
Denote by~$\mscr O'$ (resp.~$\mscr R'$) the ring of integers of~$\mscr F'$ (resp.~$\mscr M'$). Let~$\mfrak d'$ be the absolute different ideal of~$\mscr F'$. Suppose that we are given~$\delta'\in\mscr M'$ such that
\begin{align}
(x',y')\mapsto \mathrm{Tr}_{\mscr M'/\mscr F'}\left(\frac{x'{y'}^c}{2\delta'}\right)
\end{align}
induces an isomorphism~$\mscr R' \wedge _{\mscr O'} \mscr R' \tilde \longrightarrow (\mfrak c' \mfrak d')^{-1}$ for an integral ideal~$\mfrak c'\subset \mscr O'$ prime to~$p$. Define~$\delta = \delta'd_{\mscr F/\mscr F'}$ and~$\mfrak c = (\mfrak c' \mscr O)\cdot \mfrak d_{\mscr F/\mscr F'}$. Then
\begin{align}
(x,y)\mapsto \mathrm{Tr}_{\mscr M/\mscr F}\left(\frac{xy^c}{2\delta}\right)
\end{align}
induces an isomorphism~$\mscr R \wedge _{\mscr O} \mscr R \tilde \longrightarrow (\mfrak c \mfrak d)^{-1}$. We fix a finite idele~$d_{\mscr F'} \in \adele_{\mscr F'}^\infty$ with~$d_{\mscr F'}\widehat{ \mscr O}'\cap {\mscr F}'=\mfrak d'$, and define~$d_\mscr F=d_{\mscr F'}d_{\mscr F/\mscr F'}$. Then we have~$d_{\mscr F}\widehat{ \mscr O}\cap {\mscr F}=\mfrak d$. Using the choice of~$\delta'$, we fix~$\mscr O'_{v'}$-basis~$(e_{1.v'},e_{2,v'})$ for~$\mscr R'_{v'}$, for each finite place~$v'$ of~$\mscr F'$, satisfying the conditions \eqref{eq:e1} and \eqref{eq:e2}. Viewing~$\mscr R' \subset \mscr R$ via the embedding induced by~$\mscr M'\subset \mscr M$, we put 
\begin{align}\label{eq:d1}
e_{1,v}=e_{1,v'} \text{ and }e_{2,v}=d_{\mscr F/\mscr F'} e_{2,v'}.
\end{align}
Then~$(e_{1,v},e_{2,v})$ satisfies the conditions described in the previous subsection for~$\mscr M/\mscr F$. After all, we have two embeddings
\begin{align}
\varrho &\colon\adele_\mscr M ^{\infty pN,\times} \longrightarrow \GL_2(\adele_\mscr F^\infty )\\
\varrho'&\colon  \adele_{\mscr M'} ^{\infty pN,\times} \longrightarrow \GL_2(\adele_{\mscr F'}^\infty ).
\end{align}
\begin{proposition}Assume \eqref{P} and we make the choices as above. We have
\begin{align}
x(1) = \widetilde\Delta_{\mscr F'}^{\mscr F}\left( x'(1) \right),
\end{align}
where~$x(1)$ (resp.~$x'(1)$) is the CM point on~$X_+\times\GL_2(\adele_\mscr F^\infty )$ (resp. on~$X_+'\times\GL_2(\adele_{\mscr F'}^\infty)$) given by \eqref{eq:cm} with respect to our choice of~$\delta$ and~$\varsigma$(resp.~$\delta'$ and~$\varsigma'$).
\end{proposition}
\begin{proof}
The proposition can be written as
\begin{align}\label{eq:res1}
\widetilde\Delta_{\mscr F'}^{\mscr F}\left( \delta'^{\Sigma'},\varsigma'^\infty\right) = \left(\delta^\Sigma, \varsigma^\infty \right).
\end{align}
By the definition \eqref{eq:res} of~$\widetilde\Delta_{\mscr F'}^\mscr F$, the left hand side equals
\begin{align}\label{eq:res2}
\widetilde\Delta_{\mscr F'}^{\mscr F}\left( \delta'^{\Sigma'},\varsigma'^\infty\right) = \left(\left(d_{\mscr F/\mscr F'}\delta'\right)^\Sigma,
\begin{bmatrix}
d_{\mscr F/\mscr F'}	&	0	\\
0				&	1
\end{bmatrix}
\cdot \varsigma '^\infty\cdot
\begin{bmatrix}
1		&		0	\\
0		&		d_{\mscr F/\mscr F'}^\infty
\end{bmatrix}^{-1}
\right)
\end{align}
As we defined~$\delta = d_{\mscr F/\mscr F'}\delta'$, the equality in the first component of \eqref{eq:res1} holds. Now we verify the equality in the second component of \eqref{eq:res1}. Denote by~$\xi$ the second component of the right hand side of \eqref{eq:res2}, and by~$\xi_v$ the~$v$-component of~$\xi$ for a finite place~$v$ of~$\mscr F$. Let~$v'$ be the restriction of~$v$ to~$\mscr F'$. Although we often identify~$\GL_2(\adele_{\mscr F'}^\infty)$ as a subset of~$\GL_2(\adele_\mscr F^\infty)$, let us denote the natural embedding by~$j\colon \GL_2(\mscr F'_{v'}) \hookrightarrow \GL_2(\mscr F_{v})$, in order to emphasize the ambient set. Similarly, we slightly abuse notation by writing~$j\colon V'\hookrightarrow V$. In particular,~$j(e_i')=e_i$. By the defining property \eqref{eq:varsigma} of~$\varsigma$, it suffices to prove that for each finite place~$v$ of~$\mscr F$,
\begin{align}
\xi_v * e_i = q_{\delta,v}(e_{i,v}),\,\,i=1,2.
\end{align}
For~$i=1$,
\begin{align}
\xi_v * e_1 &=
\left(
\begin{bmatrix}
d_{\mscr F/\mscr F'}	&	0	\\
0				&	1
\end{bmatrix}
\cdot j(\varsigma '_{v'})\cdot
\begin{bmatrix}
1		&		0	\\
0		&		d_{\mscr F/\mscr F'}
\end{bmatrix}^{-1}
\right)
*e_1
\\
&=
\left(
\begin{bmatrix}
d_{\mscr F/\mscr F'}	&	0	\\
0				&	1
\end{bmatrix}
\cdot j(\varsigma '_{v'})
\right)
*(d_{\mscr F/\mscr F'})^{-1}e_1
\\ \label{eq:a1}
&=
\begin{bmatrix}
d_{\mscr F/\mscr F'}	&	0	\\
0				&	1
\end{bmatrix}
*(d_{\mscr F/\mscr F'})^{-1}j\left(q_{\delta',v'}({e_{1,v'}})\right)
\\ \label{eq:a2}
&=
\begin{bmatrix}
d_{\mscr F/\mscr F'}	&	0	\\
0				&	1
\end{bmatrix}
*(d_{\mscr F/\mscr F'})^{-1}q_{\delta'd_{\mscr F/\mscr F'},v}({e_{1,v}})
\cdot
\begin{bmatrix}
d_{\mscr F/\mscr F'}		&	0\\
0				&	1
\end{bmatrix}
\\
&=
\begin{bmatrix}
d_{\mscr F/\mscr F'}	&	0	\\
0				&	1
\end{bmatrix}
*(d_{\mscr F/\mscr F'})^{-1}q_{\delta,v}({e_{1,v}})
\cdot
\begin{bmatrix}
d_{\mscr F/\mscr F'}		&	0\\
0				&	1
\end{bmatrix}
\\
&=
(d_{\mscr F/\mscr F'})^{-1}q_{\delta,v}({e_{1,v}})
\cdot
\begin{bmatrix}
d_{\mscr F/\mscr F'}		&	0\\
0				&	1
\end{bmatrix}
\cdot
\begin{bmatrix}
1	&	0	\\
0	&	d_{\mscr F/\mscr F'}
\end{bmatrix}
\\
&=
q_{\delta,v}(e_{1,v}).
\end{align}
From \eqref{eq:a1} to \eqref{eq:a2}, we used Lemma~\ref{lem:lem1} and~$j(e_{1,v'})=e_{1,v}$.
For~$i=2$,
\begin{align}
\xi_v * e_2 &=
\left(
\begin{bmatrix}
d_{\mscr F/\mscr F'}	&	0	\\
0				&	1
\end{bmatrix}
\cdot j(\varsigma '_{v'})\cdot
\begin{bmatrix}
1		&		0	\\
0		&		d_{\mscr F/\mscr F'}
\end{bmatrix}^{-1}
\right)
*e_2
\\
&=
\left(
\begin{bmatrix}
d_{\mscr F/\mscr F'}	&	0	\\
0				&	1
\end{bmatrix}
\cdot j(\varsigma '_{v'})
\right)
*e_2
\\ \label{eq:a3}
&=
\begin{bmatrix}
d_{\mscr F/\mscr F'}	&	0	\\
0				&	1
\end{bmatrix}
* j(q_{\delta',v'}(e_{2,v'}))
\\ \label{eq:a4}
&=
\begin{bmatrix}
d_{\mscr F/\mscr F'}	&	0	\\
0				&	1
\end{bmatrix}
* q_{\delta,v}(j(e_{2,v'})) \cdot 
\begin{bmatrix}
d_{\mscr F/\mscr F'}	&	0	\\
0				&	1
\end{bmatrix}
\\
&=
q_{\delta,v}(j(e_{2,v'})) \cdot 
\begin{bmatrix}
d_{\mscr F/\mscr F'}	&	0	\\
0				&	1
\end{bmatrix}
\cdot
\begin{bmatrix}
1	&	0	\\
0	&	d_{\mscr F/\mscr F'}
\end{bmatrix}
\\ \label{eq:a5}
&=
d_{\mscr F/\mscr F'}q_{\delta,v}(j(e_{2,v'})) 
\\ \label{eq:a6}
&=
q_{\delta,v}(e_{2,v})).
\end{align}
From \eqref{eq:a3} to \eqref{eq:a4}, we used Lemma~\ref{lem:lem1}, and from From \eqref{eq:a5} to \eqref{eq:a6} we used \eqref{eq:d1}, the definition of~$e_{2,v}$. The proof of the proposition is complete.
\end{proof}


\section{Toric Eisenstein series and~$p$-adic~$L$-functions}\label{s3}
In this section, we define the Whittaker model of certain adelic Eisenstein series, the associated~$p$-adic Eisenstein measure, and the construction of the~$p$-adic~$L$-function. In the construction of the Whittaker model, we rather closely follow \cite{Hsieh mu,Hsieh nonvanishing} except the two minor deviations we make. Firstly, we allow imprimitive characters. That is, we fix an ideal~$\mfrak C$ of a CM field~$\mscr M$, and we will consider an algebraic Hecke character~$\chi$ of~$\mscr M$ such that the prime-to-$p$ part of the conductor of~$\chi$ merely divides~$\mfrak C$, instead of being equal. In \cite{Hsieh mu, Hsieh nonvanishing},~$\mfrak C$ was assumed to be of exact conductor of~$\chi$, or the primitive characters are used. It forces us to slightly change the choice of local sections for inert or ramified places. It also makes some formulae to appear in this subsection look different from theirs. Secondly, our notation is occasionally different from that used in \cite{Hsieh mu,Hsieh nonvanishing}. When it comes to the construction of the~$p$-adic measure, we will adopt the language of~$p$-adic modular forms with coefficients in certain Iwasawa algebra, while Hsieh used the language of~$p$-adic measures. Of course, the correspondence between Iwasawa algebra and the space of~$p$-adic measures will identify two languages with little difficulty, but the author believes that our later discussions become simpler in terms of Iwasawa algebras.
\par
As in the previous sections, we consider a~$p$-ordinary CM field~$\mscr M$ with maximal real subfield~$\mscr F$ and a CM type~$\Sigma$ and associated~$p$-adic CM type~$\Sigma_p$. Let~$\chi$ be a continuous group homomorphism
\begin{align}
\chi \colon \adele_\mscr M^\times \rightarrow \C^\times
\end{align}
such that the kernel~$\chi$ contains~$\mscr M$ embedded diagonally in~$\adele_{\mscr M}^\times$. Let~$\Z[\Sigma\cup c\Sigma]$ be the free abelian group generated by~$\Sigma \cup c \Sigma$, and let~$\Z[\Sigma]$ be the free abelian group generated by~$\Sigma$.
\begin{definition}
The continuous group homomorphism~$\chi$ as above is called an algebraic Hecke character if there exists an integer~$k$ and~$m \in \Z[\Sigma]$ with the following properties. For every~$z_\infty=(z_\sigma)_{\sigma\in\Sigma} \in \left({\C^\times}\right)^\Sigma$, 
\begin{align}
\chi(z_\infty) = \prod_{\sigma\in\Sigma} z_\sigma^k\left( \frac{z_\sigma}{z_\sigma^c} \right) ^{m_\sigma}
\end{align}
holds, where~$m$ is understood as~$\sum_{\sigma} m_\sigma \sigma$ and~$m_\sigma$ is an integer. We call 
\begin{align}
k\Sigma +(1-c)m \in \Z[\Sigma\cup c\Sigma]
\end{align}
the infinity type of~$\chi$.
\end{definition}
For an algebraic Hecke character~$\chi$ of~$\mscr M$, we define
\begin{align}
\chi_+ = \chi |_{\adele_{\mscr F}^\times},
\end{align}
to be the restriction of~$\chi$ to~$\adele_{\mscr F}^\times$.
\par
We identify~$\GL_2(\mscr F_\infty)$ with~$\prod_{\sigma \in \Sigma}\GL_2(\R)$ by sending~$\left[\begin{smallmatrix}a\otimes a'&b\otimes b'\\c\otimes c' &d\otimes d'\end{smallmatrix}\right]$ to~$\left(\left[\begin{smallmatrix}\sigma(a)a'&\sigma(b)b'\\\sigma(c)c' &\sigma(d)d' \end{smallmatrix} \right] \right)_{\sigma \in \Sigma }$, and define the subgroup~$U_\infty\subset \GL_2(\mscr F_\infty)$ to be the maximal compact subgroup corresponding to~$\prod_{\sigma\in\Sigma}\SO_2(\R)$. We denote by~$B$ the subgroup of~$\GL_2$ consisting of upper triangular matrices. Recall that a function~$\phi \colon \GL_2(\adele_\mscr F)\rightarrow \C$ is called~$U_\infty$-finite if the~$\C$-linear span of~$\{u_\infty \phi | u_\infty \in U_\infty\}$ in the space of the continuous functions from~$\GL_2(\adele_\mscr F)$ to~$\C$ is finite dimensional.
\par
For~$s \in \C$, we define~$I(s,\chi_+)$ to be the space consisting of smooth~$U_\infty$-finite functions~$\phi \colon \GL_2(\adele_\mscr F) \rightarrow \C$ satisfying
\begin{align}\label{eq:3.0.31}
\phi \left( \begin{bmatrix}a&b\\0&d\end{bmatrix}g\right) = \chi_+^{-1}(d) \left|\frac{a}{d}\right|_{\adele_{\mscr F}}^s\phi(g)
\end{align}
for every~$\left[\begin{smallmatrix}a&b\\0&d\end{smallmatrix}\right] \in B(\adele_{\mscr F})$ and every~$g \in \GL_2(\adele_{\mscr F})$. We call the functions in~$I(s,\chi_+)$ as sections. For each section~$\phi \in I(s,\chi_+)$, we define the adelic Eisenstein series associated to~$\phi$ by
\begin{align}\label{eq:3.0.34}
E_\adele(g,\phi) = \sum_{\gamma \in B(\mscr F)\backslash \GL_2(\mscr F)}\phi(\gamma g),
\end{align}
whenever the series is absolutely convergent. It is well known that~$E_\adele(g,\phi)$ is absolutely convergent when the real part of~$s$ is sufficiently large, and it has meromorphic continuation to the entire complex plane.
\subsection{Normalization of epsilon factors}
We make explicit the normalization of epsilon factors, which we use throughout the paper. Let~$\psi_\Q$ be the additive character
\begin{align}
\psi_\Q \colon \adele/\Q \to \C^\times
\end{align}
such that~$\psi_\Q=\prod_v \psi_{\Q_v}$ with each~$\psi_{\Q_v}$ specified by the following rules. For~$v=\infty$, we put
\begin{align}
\psi_{\R} (x_\infty) = \exp(-2\pi i x_\infty).
\end{align}
For a finite place~$v=\ell$, let~$(\Q/\Z)[\ell^\infty]$ be the~$\ell$-primary subgroup of~$\Q/\Z$. Denote by~$\langle x_\ell \rangle$ the image of~$x_\ell$ under the canonical homomorphism
\begin{align}
\Q_\ell \to \Q_\ell /\Z_\ell \tilde\rightarrow (\Q/\Z) [\ell^\infty] \hookrightarrow \Q/\Z.
\end{align}
We may view~$\langle x_\ell \rangle$ as the fractional part of~$x_\ell$. Then, we define
\begin{align}
\psi_{\Q_\ell} (x_\ell) = \exp ( 2\pi i \langle x_\ell \rangle ).
\end{align}
For a finite extension~$K$ of~$\Q_v$, where~$v$ may be infinite or finite, we define
\begin{align}
\psi_K = \psi_{\Q_v} \circ \mathrm{Tr}_{K/ \Q_v}
\end{align}
where~$\mathrm{Tr}_{K/\Q_v}$ denotes the trace map~$K\to\Q_v~$. For a number field~$H$, define
\begin{align}
\psi_H = \psi_\Q \circ \mathrm{Tr}_{H/ \Q}
\end{align}
where~$\mathrm{Tr}_{H/ \Q}$ denotes the trace map~$\adele_H \to \adele_\Q$. In other words, if~$v$ is a place of~$H$, then~$\psi_H$ can be written as
\begin{align}
\psi_H = \prod_{v}\psi_{H_v}
\end{align}
where~$v$ runs over the set of all places of~$H$. When it is clear what~$H$ is, we simply write~$\psi$ and~$\psi_v$ instead of~$\psi$ and~$\psi_{H_v}$, respectively. For a non-archimedean local field~$K$ with valuation ring~$\mscr O_K$, we write~$\mathfrak d _K$ for the absolute different, and choose a generator~$d_K$ for~$\mathfrak d_K$. For a continuous character~$\chi_K \colon K^\times \rightarrow \C^\times$, let~$e_K(\chi_K)$ denote the exponent of the conductor of~$\chi_K$. In other words,~$e_K(\chi_K)=0$ if~$\chi_K$ is trivial on~$\mscr O_K^\times$, and otherwise it is the smallest integer~$n$ for which~$\chi_K$ is trivial on~$1 + \mathfrak m_K^n$, where~$\mathfrak m_K$ is the maximal ideal of~$\mscr O_K$. We define the epsilon factor as
\begin{align}
\epsilon (s, \chi_K, \psi_K) =  |c|_K^s \int_{c^{-1}\mscr O_K} \chi_K^{-1}(x) \psi_K(s) d_Kx
\end{align}
where~$c=d_K \varpi_K^{e_K(\chi_K)}$, and~$d_Kx$ is the Haar measure on~$K$ which is self-dual with respect to~$\psi_K$. According to our normalization of~$\psi_K$ 
\begin{align}
\mathrm{Vol}(\mscr O_K, d_Kx) = |\mscr O_K / \mfrak d_K |^{-\frac 1 2 }
\end{align}
where~$|\mscr O_K / \mfrak d_K |$ is the number of elements in~$\mscr O_K / \mfrak d_K~$. The local root number~$W(\chi_K)$ is defined as 
\begin{align}
W(\chi_K) = \epsilon ( \frac 1 2 , \chi_K, \psi_K).
\end{align}
We refer the reader to \cite{Tate epsilon factor} for general facts about epsilon factors. It is well known that
\begin{align}
|W(\chi_K)|_\C=1
\end{align}
if~$\chi_K$ is unitary.
\begin{remark}
Note that the self-dual Haar measure~$d_Kx$ is different from the Haar measure we use for the period integral of Eisenstein series which is normalized so that the volume of~$\mscr O_K$ is always one.
\end{remark}
\subsection{Fourier coefficients of Eisenstein series}\label{subsection:FC}
In this subsection, we review the computation of Fourier coefficients of an Eisenstein series. We keep the notation introduced in the beginning of the current section. Let~$\mathbf w$ be the matrix~$ \left[\begin{smallmatrix}0&-1\\1&0\end{smallmatrix}\right]$. For a place~$v$ of~$\mscr F$, we define~$I_v(s,\chi_+)$, the space of local sections, to be the space consisting of the Bruhat-Schwartz functions~$\phi_v\colon \GL_2({\mscr F_v})\rightarrow \C$ satisfying
\begin{align}
\phi_v\left(\begin{bmatrix}a&b\\0&d\end{bmatrix}g\right) = \chi_{+,v}^{-1}\left(d\right)\left|\frac{a}{d}\right|_v^s\phi_v(g)
\end{align}
for every~$\left[\begin{smallmatrix}a&b\\0&d\end{smallmatrix}\right] \in B({\mscr F_v})$ and~$g \in \GL(2,{\mscr F_v})$. For a non-archimedean place~$v$, a function~$\phi_v\colon \GL_2(\mscr F_v)$ is called Bruhat-Schwartz if it is locally constant and compactly supported. For an archimedean place~$v$, such~$\phi_v$ is called Bruhat-Schwartz if it is smooth and of moderate growth. Since our arguments in the current paper will not involve the analytical issues arising from the growth condition of archimedean sections, we refer the readers to Section~2.8 of \cite{Bump} for the precise definition and general analytic theory of archimedean local sections. For a place~$v$ of~$\mscr F$, let~$dx_v$ be the Haar measure on~$\mscr F_v$ which is normalized by
\begin{align}
\mathrm{Vol}(\mscr O_{\mscr F_v}, dx_v) = 1.
\end{align}
For each~$\phi_v \in I_v(s,\chi_+)$,~$g_v\in \GL_2({\mscr F_v})$, and~$\beta \in {\mscr F_v}$, we define the~$\beta$-th local Whittaker integral~$W_\beta\left(\phi_v, g_v\right)$ by
\begin{align}
W_\beta(\phi_v,g_v) = \int_{{\mscr F_v}} \phi_v \left(\mathbf w
\begin{bmatrix}
1 & x_v \\
0 & 1
\end{bmatrix}
g_v \right) \psi \left(-\beta x_v \right) dx_v,
\end{align}
and the intertwining operator~$M_{\mathbf w}$ by
\begin{align}
\left(M_{\mathbf w} \phi_v\right)(g_v) = \int_{{\mscr F_v}} \phi_v \left(\mathbf w
\begin{bmatrix}
1&x_v\\ 0& 1
\end{bmatrix}g_v\right)
dx_v.
\end{align}
By definition,~$\left(M_{\mathbf w}\phi_v\right)(g_v)$ is equal to the~$0$-th local Whittaker integral. It is well known that local Whittaker integrals converge absolutely if the real part of~$s$ is sufficiently large, and have meromorphic continuation to all~$s\in \C$. We say that a section~$\phi\in I(s,\chi_+)$, which is defined in~\eqref{eq:3.0.31} is a decomposable section if we can write~$\phi$ in the form~$\phi=\bigotimes_v \phi_v$ for a family of local sections~$\phi_v \in I_v(s,\chi_+)$. In that case,~$E_\adele(g,\phi)$ has the following Fourier expansion
\begin{align}
E_\adele(g,\phi) = \phi(g) + M_{\mathbf w}\phi(g) + \sum_{\beta \in \mscr F} W_\beta (E_\adele, g),
\end{align}
where
\begin{align}
M_{\mathbf w} \phi (g) = \frac{1}{\sqrt{|{D_\mscr F|}}}\prod_v M_{\mathbf w}\phi _v (g_v)
\\
W_\beta(E_\adele,g) = \frac{1}{\sqrt{|{D_\mscr F|}}}\prod_v  W_\beta (\phi_v, g_v).
\end{align}
The sum~$\phi(g) + M_{\mathbf w} \phi(g)$ is called the constant term of~$E_\adele(g,\phi)$. We refer the reader to Section~2.8 of \cite{Bump} for the general analytic properties of local Whittaker integrals and the constant term.
\par
Now we recall the choice of local sections made in \cite{Hsieh mu, Hsieh nonvanishing}. We first introduce some notation. For a place~$v$ of~$\mscr F$, we simplify our notation by putting~$K={\mscr F_v}$, and we fix a generator~$d_K$ for~$\mathfrak d _K$, the abolute different of~$K$. If~$v$ is a finite place of~$\mscr F$, define~$\mscr O_v =\mscr O_K$, and let~$\varpi=\varpi_v$ be the uniformizer of~$\mscr O_v$. For a set~$Y$, denote by~$\mathbb I_Y$ the characteristic function of~$Y$. Now we describe our choice of local sections, beginning with the archimedean case.
\par
Let~$v$ be an infinite place of~$\mscr F$. Then~$K=\R$ and there is a unique~$\sigma \in \Sigma$ whose restriction to~$\mscr F$ is~$v$. For~$g=
\begin{bsmallmatrix}a&b\\c&d\end{bsmallmatrix} \in \GL(2,\R)$, we define
\begin{align}
J(g,i) = ci+d, \overline{J(g,i)}=-ci+d
\end{align}
\begin{align}
\nu(g) = |\det (g)|\cdot \left|J(g,i) \overline{J(g,i)}\right|^{-1}.
\end{align}
Define the sections~$\phi^{\mathrm{h}}_{k,s,\sigma}$ and~$\phi^{\mathrm{n.h}}_{k,m_\sigma,s,\sigma}$ in~$I_v(s,\chi_+)$ by
\begin{align}
\phi^{\mathrm{h}}_{k,s,\sigma} (g) &= J(g,i)^{-k}\nu (g)^s\\
\phi^{\mathrm{n.h}}_{k,m_\sigma,s,\sigma}&= J(g,i)^{-k-m_\sigma} \overline{J(g,i)}^{m_\sigma}\nu (g)^s.
\end{align}
The action of the intertwining operator~$M_{\mathbf w}$ on~$\phi^{\mathrm{h}}_{k,s,\sigma}$ is given by
\begin{align}
\left(M_{\mathbf w}\phi^{\mathrm{h}}_{k,s,\sigma}\right)(g) = i^k (2\pi) \frac{\Gamma(k+2s -1)}{\Gamma(k + s ) \Gamma(s)}\cdot \overline{J(g,i)}^k \det(g)^{-k} \nu (g) ^{1-s}.
\end{align}
\par
Define
\begin{align}
\mfrak D = p \mfrak d_{\mscr M/\mscr F} \mfrak C \mfrak C^{c}.
\end{align}
We will give the definition of the local sections for each place~$v$ of~$\mscr F$ by dividing them into several cases. We first consider the case when either~$v$ divides~$p\mathfrak F \mathfrak F^c$ or~$v$ is prime to~$\mathfrak D$. Recall that we simply write~$\mscr F_v$ as~$K$. Denote by~$\mathcal S(K)$ (resp.~$\mathcal S(K\oplus K)$) the space of Bruhat-Schwartz functions on~$K$ (resp.~$K\oplus K$). We define the Fourier transform~$\widehat\Phi$ of a function~$\Phi \in \mathcal S(K)$ as
\begin{align}
\widehat\Phi (y) = \int_K \Phi(x) \psi_K(yx)dx.
\end{align}
For a character~$\mu \colon K ^\times \rightarrow \C^\times$, we define a function~$\Phi_{\mu}\in \mathcal S(K)$ by
\begin{align}
\Phi_{\mu}(x) = \mathbb I_{\mscr O_v^\times} (x) \mu(x).
\end{align}
If~$v$ divides~$p\mathfrak F \mathfrak F^c~$, then~$v$ splits in~$\mscr M$, and we choose a place~$w$ of~$\mscr M$ over~$v$ which satisfies either~$w \mid \mathfrak F$ or~$w \in \Sigma_p$. Let~$\overline w$ be the unique place of~$\mscr M$ which is lyinng over~$v$ and different from~$w$. We put
\begin{align}
\Phi_w = \Phi_{\chi_w}\\
\Phi_{\overline w} = \Phi_{\chi_{\overline w}}.
\end{align}
To a Bruhat-Schwartz function~$\Phi \in \mcal S (K\oplus K)$, we associate a Godement section~$f_{\Phi, s} \in I_v (s,\chi_+)$ defined by
\begin{align}
f_{\Phi,s}(g) = |\det (g) | ^s \int_{K^\times }\Phi\left(\left(0,x\right)g\right) \chi_+ (x) |x|_K^{2s} d^\times x,
\end{align}
and define the local section~$\phi_{\chi,s,v}$ by
\begin{align}
\phi_{\chi,s,v} = f_{\Phi_v^0,s},
\end{align}
where 
\begin{align}
\Phi_v^0(x,y)=
\begin{cases}\label{BS choice}
\mathbb I_{\mscr O_v}(x) \mathbb I_{\mscr O_v^*}(y)
&\text{if } v\nmid \mathfrak D \\
\Phi_v^0(x,y)= \Phi_{\overline w}(x) \widehat \Phi_w (y)
&\text{if } v \mid p \mathfrak F \mathfrak F^c.
\end{cases}
\end{align}
For every~$u \in \mscr O_v^\times$ with~$v\mid p$, let~$\Phi_{\overline w}^1$ and~$\Phi_w^{[u]}$ be the Bruhat-Schwartz functions on~$K$ defined by
\begin{align}
\Phi_{\overline w}^1 (x) &= \mathbb I_{1 + \varpi \mscr O_v}(x) \chi_{\overline w}^{-1}(x)
\\
\Phi_w ^{[u]} (x) &= \mathbb I_{u(1+\varpi \mscr O_v)}(x) \chi_w (x).
\end{align}
We define~$\Phi_v^{[u]} \in \mathcal S(K\oplus K)$ by
\begin{align}\label{BS u}
\Phi_{v}^{[u]} (x,y) = \frac{1}{\vol (1+ \varpi \mscr O_v d^\times x )} \Phi_{\overline w}^1(x) \widehat \Phi _w^{[u]}(y) .
\end{align}
From our normalization of the measure, we also have
\begin{align}
\frac{1}{\vol (1+ \varpi \mscr O_v d^\times x )} \Phi_{\overline w}^1(x) \widehat \Phi _w^{[u]}(y) = \left( |\varpi|_K ^{-1} -1 \right) \Phi_{\overline w}^{1} (x) \widehat \Phi _w ^{[u]} (y).
\end{align}
\par
Now we consider the case~$v \mid \mathfrak C^{-} \cdot D_{M/F}$. Let~$w$ be the unique place of~$\mscr M$ lying over~$v$. In this case,~$\mscr M\otimes _\mscr F {\mscr F_v}$ is naturally isomorphic to~$\mscr M_w$. Recall the embedding~$\varrho$ defined in \eqref{eq:26}, and let~$\varrho_v$ be the embedding~$\mscr M_w^\times \to\GL_2(\mscr F_v)$ which extends~$\varrho$. Then we  have~$\GL_2(\mscr F_v)= B(\mscr F_v)\varrho_v(\mscr M_w)$, where~$B$ consists of upper triangular matrices. We define~$\phi_{\chi,s,v}$ to be the unique smooth section in~$I_v(s,\chi_+)$ satisfying
\begin{align}\label{eq:98}
\phi_{\chi,s,v}\left(\begin{bmatrix}
a&b\\0&d
\end{bmatrix}
\varrho_v \left(z\right) \varsigma_v
\right)
= \chi_{+,v}^{-1}(d) \left| \frac a d \right|_v^s \cdot \chi_w^{-1}(z)
\end{align}
for every~$b\in B(\mscr F_v)$ and~$z\in \mscr M_w^\times$. It is easy to check that the right hand side of \eqref{eq:98} is well-defined.
\begin{remark}\label{remark:3.2.26}
Note the difference with the choice made in \cite{Hsieh mu}, which differs by a factor of~$L(s,\chi_v)$, the local Euler factor of~$\chi_v$. If we had added a factor of~$L(s,\chi_v)$ as in \cite{Hsieh mu}, then we could not guarantee the~$p$-integrality of Fourier-Whitakker coefficients of~$\phi_{\chi,s,v}$ because we are allowing~$\mfrak C$ to be strictly smaller than the conductor of~$\chi$.
\end{remark}

\subsection{The local Whittaker integrals}
The purpose of this subsection is to summarize the formulae of the local Whittaker integrals. In this subsection, we assume $k\ge1$.
\begin{proposition}\label{prop:local integral}
The local Whittaker integrals are given as follows.
Suppose that~$v \mid \infty$. Let~$\sigma$ be the unique element in~$\Sigma$ whose restriction to~$\mscr F$ is~$v$, and we use $v$ and $\sigma$ interchangeably. Then, we have
\begin{align}
W_\beta \left( \phi^{\mathrm h}_{k,s,\sigma} ,
\begin{bmatrix}
y	&	x \\
0	&	1
\end{bmatrix}
\right)
{\Big |_{s=0}} = \frac{\left(2\pi i \right)^k}{ \Gamma(k)} \sigma (\beta )^{k-1} \exp\left(2\pi i \sigma ( \beta ) ( x + yi ) \right ) \cdot \mathbb I_{\R_+}(\sigma (\beta)).
\end{align}
If~$v$ is finite and~$v \nmid \mathfrak D$, then
\begin{align}
W_\beta \left( \phi_{\chi,s,v},
\begin{bmatrix}
1	&	 0 \\
0 	&	 \bs c_v^{-1}
\end{bmatrix}
\right) = \sum_{j=0} ^{\mathrm {val}_v(\beta \bs c_v ) } \chi_{+,v} \left( \varpi_v ^j  \bs c_v\right) |\varpi_v|_v^{-j} \cdot |\mathfrak d_\mscr F|_v^{-1} \mathbb I _{\mscr O_v} \left( \beta \bs c_v \right).
\end{align}
Recall that if~$v \mid p\mathfrak F \mathfrak F^c$, then we denote by~$w$ the unique place of~$\mscr M$ above~$v$ such that~$w\mathfrak F$ or~$w \in \Sigma_p$. If~$v \mid \mathfrak D$, then
\begin{align}\label{eq:3.3.4}
W_\beta \left( \phi_{\chi,s,v},1\right)\big|_{s=0} =
\begin{cases}
\chi_w (\beta) \mathbb I _{\mscr O_v^\times }(\beta) \cdot |\mathfrak d_\mscr F|_v^{-1}
&\text{if }v\mid p \mathfrak F \mathfrak F^c\\
 C_\beta ( \chi_v) \cdot |\mathfrak d_\mscr F |_v ^{-1} \psi_v \left( \frac {t_v} {2d_{{\mscr F_v}}}\right)
&\text{if } v \mid \mfrak I D_{\mscr M/\mscr F},
\end{cases}
\end{align}
where
\begin{align}\label{eq:3.3.5}
C_\beta ( \chi_v) = \int _{{\mscr F_v}} \chi_v^{-1} (x + 2^{-1} \delta_v ) \psi_v \left( \frac{-\beta x}{d_{{\mscr F_v}}}\right) dx.
\end{align}
If~$v = w\overline w$ with~$w \in \Sigma_p$, and~$u \in \mscr O_v ^\times$, then we have
\begin{align}
W_\beta \left( f _ {\Phi_v^{[u]}}, 1 \right)\Big |_{s=0} = \chi_{+,v} (\beta) \mathbb I_{u(1 + \varpi_v \mscr O_v)}(\beta) \cdot |\mathfrak d_\mscr F |_v^{-1}.
\end{align}
In particular, we have
\begin{align}
W_\beta( \phi_{\chi,s,v},1) \big|_{s=0} = \sum_{u} W_\beta (f_{\Phi_v^{[u]}},1)\big |_{s=0}
\end{align}
where~$u$ runs over the torsion subgroup of~$\mscr O_v ^\times$
\end{proposition}
\begin{proof}
It is proved in Proposition~4.1 of \cite{Hsieh mu} and Section~4.3 of \cite{Hsieh nonvanishing}. Note that in \cite{Hsieh mu}, the local character~$\chi_v$ is denoted by~$\lambda_v$, and the integral~$C_\beta(\chi_v)$ is denoted by~$A_\beta(\lambda_v)$. As explained in Remark~\ref{remark:3.2.26}, our local section for a place~$v$ dividing~$\mfrak I D_{\mscr M / \mscr F}$ differs by a factor from those in \cite{Hsieh mu}, so our Fourier-Whittaker coefficient in \eqref{eq:3.3.4} differs by the same factor.
\end{proof}
\begin{remark} As explained in Remark~4.2 of \cite{Hsieh mu}, it is clear from the above proposition that the local Whitakker integrals at all finite places are~$p$-integral.
\end{remark}

\subsection{Normalized Eisenstein series}
We introduce some normalized Eisenstein series.
\begin{definition}
Let~$\bullet$ denote~$\textrm{n.h.}$ or~$\textrm{h}$. They are meant to suggest non-holomorphic and holomorhpic, respectively. For each Bruhat-Schwartz function~$\Phi$ on~$\mscr F_p \oplus \mscr F_p$ of the form~$\Phi_p=\otimes_{v\nmid p} \Phi_v$, we define
\begin{align}
\phi_{\chi,s}^\bullet \left(\Phi_p\right)= \left(\otimes_{\sigma\in\Sigma}\phi^\bullet_{k,s,\sigma} \right)\otimes \left(\otimes_{v<\infty,v\nmid p}\phi_{\chi,s,v} \right) \otimes \left(\otimes_{v\mid p} f_{\Phi_v,s}\right),
\end{align}
and define the adelic Eisenstein series
\begin{align}
E^\bullet_\chi\left(\Phi_p\right)(g)= E_\adele\left(g, \phi^\bullet_{\chi,s}\left(\Phi_p\right)\right)\Big|_{s=0}
\end{align}
by \eqref{eq:3.0.34}. Let~$D_\mscr F$ be the absolute discriminant of~$\mscr F$. We define the holomorphic (resp. nearly holomorphic Eisenstein series)~$\mathbb E_\chi^\mathrm{h}\left(\Phi_p \right)$ (resp.~$\mathbb E_\chi^{\mathrm{n.h.}}\left(\Phi_p\right)$) by
\begin{align}
\mathbb E_\chi^\mathrm{h}\left(\Phi_p \right)(\tau, g_f) &= \frac{\Gamma_\Sigma\left(k\Sigma\right)}{\sqrt{|D_{\mscr F}|}\left(2\pi i\right)^{k\Sigma}} E^\mathrm{h}_\chi \left(g_\infty,g_f\right) \cdot \bs J(g_\infty, \bs i)^{k\Sigma}
\\
\mathbb E_\chi^\mathrm{n.h.}\left(\Phi_p \right)(\tau,g_f) &= \frac{\Gamma_\Sigma\left(k\Sigma\right)}{\sqrt{|D_{\mscr F}|}\left(2\pi i\right)^{k\Sigma}}  E^\mathrm{n.h.}_\chi \left(g_\infty,g_f\right) \cdot \bs J(g_\infty, \bs i)^{k\Sigma+2 m} \left(\det(g_\infty)\right)^{-m},
\end{align}
where~$(\tau,g_f) \in X_+\times \GL_2(\adele_{\mscr F}^{\infty})$,~$\bs i = (i)_{v \in I} \in X_+$,~$g_\infty \in \GL_2(\mscr F_\infty)$ with~$g_\infty \bs i = \tau$.
\end{definition}
Let~$\Phi_p^0= \otimes_{v\mid p}\Phi_v^0$ be the Bruhat-Schwartz function on~$\mscr F_p\oplus \mscr F_p$ defined in \eqref{BS choice}. We define
\begin{align}
\mathbb E^\mathrm h_\chi &= \mathbb E^{\mathrm h}_\chi(\Phi_p^0)
\\
\mathbb E^{\mathrm {n.h.}}_\chi &= \mathbb E^\mathrm{n.h.}_\chi(\Phi_p^0).
\end{align}
For every~$u= (u_v)_{v\mid p} \in \prod_{v\mid p} \mscr O_v^\times = \mscr O_{p}^\times$, let~$\Phi_{p}^{[u]} = \otimes_{v\mid p} \Phi_v^{[u_v]}$ be the Bruhat-Schwartz function as defined in \eqref{BS u} and define
\begin{align}
\mathbb E^\mathrm h _{\chi , u} = \mathbb E^\mathrm h _\chi \left( \Phi _p ^ {[u ] }\right).
\end{align}
We choose an integer~$N$ of the form~$\mathrm{Norm}_{\mscr M/\Q} \left(\mathfrak C \mathfrak d _{\mscr M/\mscr F}\right)^{j_0}$ for a sufficiently large integer~$j_0$ so that~$\phi_{\chi,s,v}$ are invariant by~$U(N)$ for every~$v|N$, and put~$K:=U(N)$.
Then the section~$\phi_{\chi,s} \left(\Phi^{[u]}_p\right)$ is invariant by~$K_1^n$ for a sufficiently large~$n$.
\begin{definition}\label{definition:3.4.9}
Let~$\bs c = (\bs c_v)_{v}$ be an element of~$\adele_{\mscr F}^{\infty,\times}$ such that~$\bs c_v=1$ for~$v | \mathfrak D$, and let~$\mathfrak c= \bs c \cdot \adele_{\mscr F}^{\infty} \cap \mscr F$ be the fractional ideal of~$\mscr F$. We call such finite idele~$\bs c$ an idele associated to~$\mfrak c$.
\end{definition}
For each~$\beta \in \mscr F_+$, the set of all totally positive elements in~$\mscr F$, we define the prime-to-$p$ Fourier coefficient~$a_\beta^{(p)} ( \chi,\mathfrak c)$ as
\begin{align}
a _\beta ^{(p)} \left(\chi,\mathfrak c \right)= \frac{1}{|D_{\mscr F}|_\R|D_{\mscr F}|_{\Q_p}}\mathrm{Norm} _{\mscr F/\Q} (\beta^{-1}) \prod_{v\nmid p} W_\beta \left( \phi_{\chi,s,v}, 
\begin{bmatrix}
1	&	0	\\
0	&	\bs c_v^{-1}
\end{bmatrix}
\right)
\Big |_{s=0}\cdot \mathbb I _{\mscr O_p^\times}(\beta).
\end{align}
It is equal to
\begin{align}\label{FC outside p}
\beta^{-\Sigma} \prod_{w\mid \mathfrak F} \chi_w (\beta) \mathbb I_{\mscr O_v^\times}(\beta) \cdot \prod_{v \nmid \mathfrak D} \left(\sum_{j=0}^{\val_v(\bs c_v \beta)} \chi_{+,v} (\varpi_v^j)|\varpi|_v^{-j}\right) \prod_{v\nmid \mathfrak I D_{M/F}}C_\beta(\chi_v)\psi\left(\frac{-t_v}{2d_{{\mscr F_v}}}\right)\cdot \mathbb I _{\mscr O_p^\times}(\beta)
\end{align}
by the formulae of the local Whittaker integrals summarized in Proposition~\ref{prop:local integral} and the fact that
\begin{align}
\prod_v |D_\mscr F|_v=1
\end{align}
where the product is taken over the set of all places of~$\mscr F$. It is clear that~$a _\beta ^{(p)} (\chi,\mathfrak c)$ is~$p$-integral.
\par
We recall the Fourier expansion of Hilbert modular forms. Let~$\mfrak c$ be a fractional ideal prime to~$\mfrak D$. Let~$f \colon X_+ \times \GL_2(\adele_\mscr F^\infty)\to \C$ be either~$\mathbb E_{\chi,u}^{\mathrm h}$ or~$\mathbb E_{\chi,u}^{\mathrm {n.h}}$. Let~$\bs c$ be the finite idele associated to~$\mfrak c$ introduced in Definition~\ref{definition:3.4.9}. We define~$f|_{\left( \mscr O, \mfrak c ^{-1}\right)}$ to be
\begin{align}
f|_{\left(\mscr O , \mfrak c ^{-1}\right)}(\tau) = f
\left(
\tau, 
\begin{bmatrix}
1&0\\
0&\bs c ^{-1}
\end{bmatrix}
\right)
\end{align}
and it is called the~$q$-expansion of~$f$ at the cusp~$\left(\mscr O, \mfrak c ^{-1}\right)$. It is well-defined independently of the choice of $\bs c$.

\begin{remark}
In the language of geometric modular forms, the~$q$-expansion at the cusp~$\left(\mscr O, \mfrak c ^{-1}\right)$ is interpreted as evaluating at the Tate curve associated to the pair~$\left(\mscr O, \mfrak c ^{-1}\right)$ of fractional ideals.
\end{remark}
\begin{proposition}\label{EFC}
The Eisenstein series~$\mathbb E^{\mathrm h} _{\chi, u}$ belongs to~$M_k(K_1^n,\C)$ for a sufficiently large~$n$. The~$q$-expansion of~$\mathbb E^{\mathrm h} _{\chi, u}$ at the cusp~$(\mscr O, \mathfrak c^{-1})$
\begin{align}
\mathbb E^{\mathrm h} _{\chi, u} |_{(\mscr O,\mathfrak c^{-1})}(q)=
\sum_{\beta \in \left(N^{-1}\mathfrak c^{-1}\right)_+} a_\beta ( \mathbb E^\mathrm h _{\chi, u} , \mathfrak c ) q^\beta \in \C 
[[q^{\left(N^{-1}\mathfrak c^{-1}\right)_+}]]
\end{align}
has no constant term and the~$\beta$-th Fourier coefficient of~$\mathbb E^\mathrm h _{\chi, u}$ is given by
\begin{align}
a_\beta(\mathbb E^\mathrm h _{\chi, u}, \mathfrak c) = a _\beta ^{(p)} ( \chi, \mathfrak c ; u) \beta^{k\Sigma}\prod_{v\mid p} \mathbb I _{u_v(1+\varpi_v \mscr O_v)} (\beta) \cdot \prod_{w \in \Sigma_p} \chi_w(\beta).
\end{align}
Therefore,~$\mathbb E^{\mathrm h}_{\chi,u} (g;\mathfrak c)$ belongs to~$M_k(\mathfrak c, K_1^n, \mscr O_{H,(p)})$, for a number field~$H$ contained in~$\C$. Furthermore, we have
\begin{align}
\mathbb E^\mathrm h _{\chi} (g;\mathfrak c) = \sum_{u} \mathbb E^\mathrm h _{\chi, u}(g;\mathfrak c),
\end{align}
where~$u$ runs over the elements of the torsion subgroup of~$\mscr O_p^\times$.
\end{proposition}
\begin{proof}
See Proposition~4.4 of \cite{Hsieh mu}.
\end{proof}
\par
As explained in Remark~(4.11) of \cite{Hsieh mu}, an important feature of our Eisenstein series~$\mathbb E^\mathrm h _{\chi,u}$ and~$\mathbb E^\mathrm h _{\chi}$ is that they are toric Eisenstein series of eigencharacter~$\chi$. Let us denote by~$|[a]$ the right translation action of~$\varsigma^{-1}_f \varrho(a) \varsigma_f \in \GL_2(\adele_\mscr M^\infty)$ acting on the space of Hilbert modular forms. We define
\begin{align}
\mathcal T =  {\prod_{v}}'\, \mathcal T_v \subset \adele_{\mscr M}^{\infty,\times},
\end{align}
where~${\prod_{v}}'$ denotes the restricted product taken over all finite places of~$\mscr F$ and
\begin{align}
\mathcal T_v =
\begin{cases}
\mscr R_v^\times {\mscr F_v}^\times 	&\text{ if } v\text{ is split in }M\\
\mscr M_v^\times 				&\text{ if } v\text{ is not split in }M.
\end{cases}
\end{align}
Note that in Remark~4.5. of \cite{Hsieh mu} Hsieh calls the action of the operator~$[a]$ as Hecke action, but we adopt the terminology stabilizer action, as~$\mcal T$ stabilizes CM points in the Shimura variety via the action~$|[a]$. In the next proposition, we record the action of~$\mathcal T$ on our Eisenstein series, from which the terminology toric arises.
\begin{proposition}[(4.11) of \cite{Hsieh mu}]\label{prop:toric}
For every~$a \in \mathcal T$, we have
\begin{align}
\mathbb E^\mathrm h _ \chi  \big| [a] &= \chi^{-1} (a)\mathbb E ^\mathrm h _\chi \\
\label{eq:3.4.19}
\mathbb E^\mathrm h _ {\chi ,u} \big| [a] &= \chi^{-1} (a)\mathbb E ^\mathrm h _{\chi,u.a^{1-c}}.
\end{align}
Here~$u.a^{1-c}$ denotes the element~$(u'_v)_{v\mid p}\in \prod_{v\mid p} \mscr O_v^\times$ with~$u'_v = u_v a_w(a_{cw})^{-1}$,~$w$ being the element in~$\Sigma _p$ whose restriction to~$\mscr F$ is~$v$. In particular, if~$p$-component of~$a$ is trivial, then~$ua^{1-c}=u$.
\end{proposition}
\begin{proof}
See the Remark~4.5 of \cite{Hsieh mu}.
\end{proof}

\begin{corollary}
Let~$v$ be a root of unity in~$\mscr M$. Using the natural isomorphism~$\prod_{v|p} \mscr F_v \cong \prod_{w \in \Sigma_p} \mscr M_w$, we regard~$v$ as an element of~$\mscr O_p^\times$. Then, we have
\begin{align}
\mathbb E^\mathrm h _ {\chi,u}\big|[v] = \mathbb E^\mathrm h _ {\chi, uv^2}.
\end{align}
\end{corollary}
\begin{proof}
Let~$v$ be a root of unity in~$\mscr M$. We regard~$v$ as an idele in~$\adele_{\mscr M}^{\infty,\times}$ via the diagonal embedding. Since~$\chi(v)=1$, \eqref{eq:3.4.19} is reduced to
\begin{align}
\mathbb E^\mathrm h _ {\chi ,u} \big| [v] &= \chi^{-1} (a)\mathbb E ^\mathrm h _{\chi,u.v^{1-c}}.
\end{align}
We also observe~$u.v^{1-c}=uv^2$, since~$v^c = v^{-1}$. 
\end{proof}

\subsection{Measures and Iwasawa algebra}
We review the basic theory of~$p$-adic measures and Iwasawa algebras. In the current subsection, let~$Z$ be a profinite group and~$\Lambda$ be the Iwasawa algebra~$\Lambda(Z)$. Let~$R$ be a~$p$-adic ring. Any continuous function
\begin{align}
\upphi \colon Z \to R
\end{align}
extends to a~$\Z_p$-linear map
\begin{align}
\upphi \colon \Lambda \to R
\end{align}
which we denote by the same symbol. If we are given an element~$\lambda \in \Lambda$ in addition, then we write
\begin{align}\label{eq:3.5.3}
\upphi (\lambda)=\int _Z \upphi d\lambda
\end{align}
suggesting that we interpret~$\lambda$ as a measure which we use to integrate the function~$\upphi$.

\subsection{$p$-adic interpolation of Fourier coefficients}
In this subsection, we reformulate Proposition~\cite{Hsieh mu} in terms of~$\Lambda$-adic forms so that it becomes easier later for us to prove the necessary congruence between~$p$-adic~$L$-functions. For simplicity, we write~$Z(\mfrak C)=Z$ and~$\Lambda (Z(\mfrak C)) = \Lambda$. In the followings, we will interpret Eisenstein series~$\mbb E_{\chi,u}^{\mathrm h}(\cdot, \begin{bsmallmatrix}1&0\\0&\bs c^{-1}\end{bsmallmatrix}))$ as a specialization of a~$\Lambda$-adic Eisenstein series.
\par
We begin with the definition of domain of interpolation:
\begin{definition}
Let $\mfrak X_+$ be the set of algebraic Hecke characters $\chi$ whose infinity type is $k\Sigma$ with $k \ge 1$.
\end{definition}
\par
We want to define~$A(\beta,v, \bs c_v)\in \Lambda$ for a totally positive~$\beta \in \mscr F$, a finite place~$v$ of~$\mscr F$, and~$\bs c_v \in \mscr F_v$, and show that they interpolate Fourier-Whittaker coefficients introduced in Subsection~\ref{subsection:FC}. We assume that~$\bs c_v=1$ for~$v \nmid p\mfrak C D_{\mscr M/\mscr F}$. Recall that we use the convention that if~$v\mid p\mfrak F \mfrak F^c$, then~$w$ denotes the place of~$\mscr M$ contained in~$\Sigma_p$ or dividing~$\mfrak F$. If~$v$ is not split in~$\mscr M$,~$w$ denotes the unique place of~$\mscr M$ lying above~$v$.
\begin{lemma}\label{lemma:361}
Let~$v$ be a place of~$\mscr F$ which does not divide~$\mfrak I D_{\mscr M/\mscr F}$. Define~$A(\beta ,v,\bs c_v)\in \Lambda$ by
\begin{align}
 A(\beta ,v,\bs c_v)
=
\begin{cases}
\rec_w(\beta)\mbb I_{\mscr O_v^\times}(\beta)
&\text{ if~$v|\mfrak C$ splits in~$\mscr M$,~$w\mid v$,~$w\mid \mfrak F$}
\vspace{2mm}
\\
\displaystyle{
\sum_{j =0}^{\val _v(\beta \bs c_v)}
\rec_{\mscr M,v}(\varpi_v^j \bs c_v)|\varpi_v|_v^{-j} \mbb I_{\mscr O_v} ( \beta \bs c_v)
}
&\text{ if~$v<\infty$ and~$v\nmid p\mfrak C D_{\mscr M/\mscr F}$}\\
\end{cases}
\end{align}
Then, for every~$\chi$
\begin{align}
\int_Z \widehat \chi d A(\beta,v,\bs c_v) = |\mfrak d|_v \cdot  W_\beta\left(\phi _{\chi,s,v},
\begin{bmatrix}
1 & 0 \\
0 & \bs c_v^{-1}
\end{bmatrix}
\right)|_{s=0}
\end{align}
holds.
\end{lemma}
\begin{proof}
The assertion of the lemma is a restatement of the formulas given in Proposition~\ref{prop:local integral}.
\end{proof}
\par
We also need an analogue of the previous lemma for places~$v$ dividing~$\mfrak I D_{\mscr M/\mscr F}$. As usual,~$w$ denotes the unique place of~$\mscr M$ lying above~$v$. We first analyze the integral
\begin{align}
C_\beta(\chi_v)= \int _{{\mscr F_v}} \chi_v^{-1} (x + 2^{-1} \bs \theta_v ) \psi_v \left( \frac{-\beta x}{d_{{\mscr F_v}}}\right) dx
\end{align}
which was defined in \eqref{eq:3.3.5}. Let~$\mfrak l$ be the maximal ideal of~$\mscr O_v$. There exists sufficiently large non-negative integers~$j_0$ and~$j_1$ such that
\begin{align}\label{eq:3.5.5}
C_\beta(\chi_v)= \mathrm{Vol}(\mfrak l^{j_1}, dx)\sum_{x\in \mfrak l^{-j_0}/\mfrak l^{j_1}} \chi_v^{-1} (x + 2^{-1} \bs \theta_v ) \psi_v \left( \frac{-\beta x}{d_{{\mscr F_v}}}\right)
\end{align}
because the conductor of~$\chi_v$ divides~$\mfrak C \mscr R_w$ for all~$\chi \in \mfrak X_+$. It is important that we can choose~$j_0$ and~$j_1$ which satisfies \eqref{eq:3.5.5} for a given~$\beta \in \mscr F$ independently of~$\chi \in \mfrak X_+$. Now we can readily formulate the following analogue of Lemma~\ref{lemma:361}.
\begin{lemma}\label{lemma:3.5.6}
Let~$v$ be a place of~$\mscr F$ which divides~$\mfrak I D_{\mscr M/\mscr F}$. Define~$A(\beta ,v,\bs c_v)\in \Lambda$ as
\begin{align}
 A(\beta ,v,\bs c_v)
=
\psi_v\left(\frac{-t_v}{2d_{\mscr F_v}}\right)\mathrm{Vol}(\mfrak l^{j_1}, dx)\sum_{x\in \mfrak l^{-j_0}/\mfrak l^{j_1}} \rec_v (x + 2^{-1} \bs \theta_v )^{-1} \psi_v \left( \frac{-\beta x}{d_{{\mscr F_v}}}\right).
\end{align}
Then, for every~$\chi \in \mfrak X_+$
\begin{align}
\int_Z \widehat \chi d A(\beta,v,1) = |\mfrak d|_v \cdot W_\beta\left(\phi _{\chi,s,v},1 \right)|_{s=0}
\end{align}
holds.
\end{lemma}
\begin{proof}
It follows immediately from \eqref{eq:3.5.5}.
\end{proof}
\begin{definition}
We define~$A^{(p)}(\beta,\mfrak c) \in \Lambda$ as
\begin{align}
A^{(p)}(\beta,\mfrak c)
=
\rec^\infty_{\mscr M}(\beta)\mathrm{Norm}_{\mscr F / \Q}(\beta)^{-1} 
\prod_{v\nmid p}  A(\beta,v,\bs c_v).
\end{align}
The superscript~$(p)$ in the notation suggests that it is Fourier coefficients away from~$p$. 
\par
For every~$u \in \mscr O_p^\times$, which is of finite order, we define~$ A(\beta,\mfrak c;u) \in \Lambda$ as
\begin{align}
 A(\beta,\mfrak c;u)
=
 A^{(p)}(\beta,\mfrak c)
\rec_{\Sigma_p}(\beta)
\mbb I_{u(1+\varpi_p \mscr O_p)}(\beta).
\end{align}
\end{definition}
\begin{proposition}
We have
\begin{align}\label{eq015}
\int_Z \widehat \chi d A^{(p)}(\beta,\mfrak c)
=
i_p
\left(
A_\beta^{(p)} (\chi,\mfrak c)
\right),
\end{align}
and
\begin{align}\label{eq016}
\int _Z \widehat \chi d  A(\beta, \mfrak c ; u) = A _\beta (\chi, \mfrak c;u)
\end{align}
for every~$\chi \in \mfrak X_+$.
\end{proposition} 
\begin{proof}
By \eqref{FC outside p}, Lemma~\ref{lemma:361}, and Lemma`\ref{lemma:3.5.6}, the equality \eqref{eq015} follows from
\begin{align}\label{eq019}
\widehat\chi
\left(
\rec^\infty_{\mscr M}(\beta)
\right)
\mathrm{Norm}_{\mscr F/ \Q}(\beta)^{-1} 
=
\beta^{(k-1)\Sigma}.
\end{align}
Indeed, by the definition of~$p$-adic avatar
\begin{align}
\widehat \chi ( \rec^\infty _{\mscr M} (\beta ))
=&
\chi (\beta^\infty) i_p
\left(
\beta^{k\Sigma}
\right)
\\
=&
i_p
\left(
\beta^{k\Sigma}
\right).
\end{align}
Also,~$\beta ^\Sigma = \mathrm{Norm}_{\mscr F/\Q}(\beta)$ as~$\beta \in \mscr F$, whence \eqref{eq019} follows. On the other hand, \eqref{eq016} directly follows from Proposition~\ref{EFC}.
\end{proof}
Let~$\mfrak c$ be a fractional ideal of~$\mscr F$ prime to~$p\mfrak C D_{\mscr M/\mscr F}$. We define~$E(\mfrak c;u)$ as a formal~$q$-expansion
\begin{align}
E(\mfrak c;u) = \sum_{\beta} A (\beta, \mfrak c ; u) q^\beta.
\end{align}
For a continuous function~$\upphi \colon Z \to \overline \Q_p$, we write~$E(\mfrak c; u)(\upphi)$ the formal~$q$-series whose~$\beta$-th coefficient is~$\int_Z\upphi d A(\beta, \mfrak c ; u)$.
Then Proposition~\ref{EFC} implies that for~$\chi \in \mfrak X_+$ we have
\begin{align}
E(\mfrak c ; u) (\widehat \chi) =  \widehat{\mbb E^{\mathrm h}_{\chi,u}|_{(\mscr O,\mfrak c^{-1})}}.
\end{align}
In particular,~$E(\mfrak c; u)$ is a~$\Lambda$-adic modular form in the sense of Definition~\ref{def:pmf}.
\subsection{A modification factor}
In this subsection, we will define a modification factor which we denote by~$\mcal C$. It will be used in our definition of~$p$-adic~$L$-function.
\par
Let us begin with the notation for uniformizers in the completions of~$\mscr M$. Let~$v$ be a place of~$\mscr F$ lying over~$p$, and let~$w$ be the unique place in~$\Sigma_p$ lying over~$v$. Let~$\overline w$ be the complex conjugate of~$w$. Fix a uniformizer~$\varpi_w$ of~$\mscr M_w$. The complex conjugation~$c$ extends to an isomorphism
\begin{align}
\mscr M_w \rightarrow \mscr M_{\overline w}
\end{align}
and let~$\varpi_{\overline w}$ be the image is~$\varpi_w$ under the above isomorphism. In particular,~$\varpi_{\overline w}$ is a uniformizer of~$\mscr M_{\overline w}$.
\par
As before, for any place~$w$ of~$\mscr M$, we denote by~$\rec_w$ the local Artin reciprocity map
\begin{align}
\rec_w \colon \mscr M_w ^\times \to Z
\end{align}
and define the modification factor~$\mcal C~$ to be
\begin{align}\label{eq:3.7.3}
\mcal C=
\prod_{w\in \Sigma_p}\left(1 - \frac{\rec_{w}\left( \varpi_{w}\right)}{\rec_{\overline w}\left( \varpi_{\overline w}\right)}\right),
\end{align}
which is an element of~$\Lambda$, the Iwasawa algebra of~$Z$. Of course,~$\mcal C$ depends on our choice of uniformizer~$\varpi_w$, but we will make it explicit whenever there is a preferred choice.
\begin{remark}
One might attempt to explain why a particular choice of~$\varpi_w$ is preferred to others using Lubin-Tate theory, but it will not be one of our main concerns since we do not need Lubin-Tate theory for cyclotomic extensions, which is our main example.
\end{remark}
Let~$\chi_0$ be a Hecke character associated to an elliptic curve with complex multiplication by imaginary quadratic field~$\mscr M_0$. Assume that~$p$ splits in~$\mscr M_0$, and has good ordinary reduction at the prime~$p$. In this case, a CM type~$\Sigma_0$ of~$\mscr M_0$ is merely choice of an embedding of~$\mscr M_0$ into~$\C$, which we choose to be the restriction of~$i_\infty$ to~$\mscr M_0$. Let~$w$ be the unique element in~$\Sigma_{0,p}$, the~$p$-adic CM type of~$\mscr M_0$ induced by~$\Sigma_0$, and let~$p_{w}$ be the image of~$p$ under the natural isomorphism
\begin{align}
\Q_p \longrightarrow \mscr M_{0,w}
\end{align}
which is of course a uniformizer of~$\mscr M_{0,w}$. If we let~$p_{\overline w}$ be the image of~$p$ under the natural isomorphism
\begin{align}
\Q_p \longrightarrow \mscr M_{0,\overline w}
\end{align}
then the isomorphism
\begin{align}
\mscr M_{0, w}\longrightarrow \mscr M_{0,\overline w}
\end{align}
induced by the complex conjugation maps~$p_w$ to~$p_{\overline w}$. Define~$\alpha \in \C$ to be
\begin{align}
\alpha :=\chi_0(p_{w})
\end{align}
which is in fact an algebraic number. Furthermore, we assume that~$\alpha$ is a~$p$-adic unit with respect to our fixed embedding~$i_p$. If not, using the assumption that~$p$ splits in~$\mscr M_0$, we can make~$\alpha$ to be a~$p$-adic unit by replacing~$\chi_0$ with its complex conjugate. 
\par
Suppose that we are given a CM field~$\mscr M$ which contains~$\mscr M_0$, and define
\begin{align}
\chi := \chi_0^k \circ\mathrm{Norm}_{\mscr M/ \mscr M_0}
\end{align}
for any integer~$k$, where~$\mathrm{Norm}_{\mscr M/ \mscr M_0}\colon \adele^\times_\mscr M \to \adele ^\times_{\mscr M_0}$ denotes the map between ideles induced by the norm map. If we let~$\kappa$ be the algebraic Hecke character
\begin{align}\label{eq:3.7.10}
\kappa \colon \adele_{\mscr M}^\times &\to \C^\times
\\
x&\mapsto \left | x \right |_{\adele _ \mscr M}
\end{align}
then~$p$-adic avatar~$\widehat \kappa$ of~$\kappa$ is the~$p$-adic cyclotomic character, because we have normalized the reciprocity maps geometrically. 
\par
Important examples of such~$\mscr M$ are of the form
\begin{align}
\mscr M_r = \mscr M_0(\mu_{p^r})
\end{align}
for any non-negative integer~$r$. In this case, we choose the uniformizer~$\varpi_r$ for the unique element~$w_r$ in the~$p$-adic CM type~$\Sigma_p$ of~$\mscr M_r$ to be
\begin{align}\label{eq:3.7.13}
\varpi_r = 1- \zeta_{p^r}
\end{align}
for a primitive~$p^r$-th root of unity~$\zeta_{p^r}$, and furthermore we may choose them so that
\begin{align}
\mathrm{Norm}_{\mscr M_{r+1,w_{r+1}}/\mscr M_{r,w_r}} (\varpi_{r+1}) = \varpi_r
\end{align}
for any positive integer~$r$, and
\begin{align}
\mathrm{Norm}_{\mscr M_{1,w_1}/\mscr M_{0,w_0}} (\varpi_{1}) = p
\end{align}
for~$r=1$ case. 
\begin{proposition}
Suppose that~$\mscr M$ is of the form~$\mscr M_0(\mu_{p^r})$. In particular, we choose~$\varpi_w \in \mscr M_w$ for the unique element~$w \in \Sigma_p$ such that
\begin{align}\label{eq:3.7.17}
\mathrm{Norm}_{\mscr M_w/\Q_p}(\varpi_w) = p
\end{align}
holds. Then
\begin{align}
\int_Z \widehat{\chi\kappa^n} d\mcal C = 1- \frac{\alpha^{k}p^k}{\overline \alpha ^{k}}=1-\alpha^{2k}p^{2k}
\end{align}
holds for any integers~$k$ and~$n$.
\end{proposition}
\begin{proof}
We can achieve \eqref{eq:3.7.17} by \eqref{eq:3.7.13}. Note that according to the geometric normalization of the reciprocity map,~$\alpha \overline \alpha=p^{-1}$. Also note the fact that~$\kappa(\varpi_w) = \kappa(\varpi_{\overline w})$. The assertion of the proposition easily follows from the definitions of~$\chi$,~$\alpha$, and~$p$-adic avatars of them.
\end{proof}
\subsection{$p$-adic~$L$-function}\label{ss:0308}
In this subsection, we define the~$p$-adic~$L$-function and describe its interpolation property which characterizes it. We make some modifications to the~$p$-adic~$L$-function defined in \cite{Hsieh mu}, which is crucial for our proof of main result.
\par
For a number field~$H$ denote by~$Cl(H)$ the ideal class group of~$H$. The transfer map induces a map~$Cl(\mscr F) \to Cl(\mscr M)$ and let~$Cl_-(\mscr M)$ be its cokernel. We often write~$Cl_-$ instead of~$Cl_-(\mscr M)$ for simplicity. Let~$\mcal D$ a set of representatives for~$Cl_-$, chosen from~$\adele_\mscr M^{\mfrak D\infty,\times}$. As before, we write~$Z$ for~$Z(\mfrak C)$, and~$\Lambda$ denotes the Iwasawa algebra over~$Z$, and let~$\mcal C \in \Lambda$ be the modification factor defined in \eqref{eq:3.7.3}.
\begin{definition}\label{def:3.8.1}
We define a measure~$\mcal L_{\Sigma,\mfrak C,\delta}$ on~$Z$ to be
\begin{align}\label{eq:KHT measure}
\mcal L_{\Sigma,\mfrak C, \delta} =
\mcal C \cdot
\sum_{u\in U/U^{\mathrm {alg}}}
\sum_{a\in \mcal D}\rec_{\mscr M}(a) E(\mfrak c(a); u)(x(a)) \in \Lambda,
\end{align}
and we often write~$\mcal L = \mcal L _{\Sigma, \mfrak C, \delta}$.
\end{definition}
\par
We compare our normalization of~$p$-adic~$L$-function with some of those appeared in the literature. In \cite{Hsieh mu},~$\mcal D$ is chosen to be a representative set for~$Cl_-$ modulo the image of~$\mcal T$ in~$Cl_-$ under~$\rec_\mscr M$, and the sum indexed by~$u$ is taken over~$U$. The difference in~$\mcal D$ is harmless because the image of~$\mcal T$ in~$Cl_-$ is~$2$-primary, whence it simply introduce a constant factor which prime to~$p$. However, the difference introduced by~$U_{\mathrm{alg}}$ is a constant factor, which is not always prime to~$p$. In \cite{Hsieh mu},~$p$ is always assumed to be unramified in~$\mscr F$, which implies that the order of~$U_{\mathrm{alg}}$ is prime to~$p$. On the other hand, when~$p$ ramifies in~$\mscr F$, the order of~$U_{\mathrm{alg}}$ may be divisible by~$p$, and we possibly obtain more optimal~$\mcal L$ by considering~$U/ U_{\mathrm{alg}}$ instead of~$U$. Indeed, this factor is what made the~$p$-adic~$L$-function of \cite{Katz, Hida-Tilouine} sometimes differ from the interpolation formula predicted by non-commutative Iwasawa theory, as observed in \cite{CMFT}. The factor~$\mcal C$ which we defined in \eqref{eq:3.7.3} did not appear in the previous work of~$p$-adic Hecke~$L$-functions, but it appeared in the interpolation formula for the~$p$-adic Rankin-Selberg~$L$-functions.
\par
For a place~$w$ that divides~$p\mfrak C^+$, recalling that we have chosen~$\bs \theta_v$ in Subsection~\ref{ss:cmp}, let
\begin{align}
\Eul_w(\chi_w)=\chi_w(2\delta_w) \frac{L(0,\chi_w)}{\epsilon (0, \chi_w, \psi)L(1, \chi_w^{-1})}
\end{align}
and define
\begin{align}
\Eul_{p}(\chi) = \prod_{w \mid \Sigma_p} \Eul(\chi_w)
,\hspace{5mm}
\Eul_{\mfrak C^+}(\chi) = \prod_{w \mid \mfrak F} \Eul(\chi_w).
\end{align}
Also we write
\begin{align}
\mcal C_\infty (\chi) =i_\infty \left( i_p^{-1} \mcal C(\widehat\chi)\right).
\end{align}
Let~$L^{p\mfrak C}(\chi,s)$ be the imprimitve complex~$L$-function, obtained by removing Euler factors dividing~$p\mfrak C$ from the primitive complex~$L$-function~$L(\chi,s)$. We call~$\mcal L$ a~$p$-adic~$L$-function because it interpolates special values of complex Hecke~$L$-functions in the following way.
\begin{proposition}
Let~$\chi$ be algebraic Hecke character of infinity type~$k\Sigma + (1-c)m$ satisfying either
\begin{align}
k\ge 1,\,\,\text{and}\,\,m_\sigma\ge0\text{ for all $\sigma$}
\end{align}
or
\begin{align}
k \le 1,\,\,\text{and}\,\,k+m\ge 0\text{ for all $\sigma$}.
\end{align}
Let~$r$ be the number of distinct prime factors of the relative discriminant~$D_{\mscr M/ \mscr F}$. Then
\begin{align}\label{eq:3.8.7}
\frac{\int _Z \widehat \chi d\mcal L}{\Omega_p^{k\Sigma+2m}} = L^{p\mfrak C}(\chi,0)\cdot\mcal C_\infty(\chi)\Eul_p(\chi)\Eul_{\mfrak C^+}(\chi)
\times
\frac
{2^r\pi^m\Gamma_\Sigma(k\Sigma+m)}
{\sqrt{|D_\mscr F|_\R} (\mathrm{Im}(\delta))^m \Omega_\infty^{k\Sigma + 2m}}
\end{align}
holds. The equality is understood as the assertion that the both sides are algebraic numbers in~$\C$ and~$\overline \Q_p$ respectively, and they are equal to each other with respect to the embeddings~$i_\infty$ and~$i_p$.
\end{proposition}
\begin{proof}
By definition of~$\mcal L$, we have
\begin{align}
\int _Z \widehat \chi d\mcal L = \mcal C(\widehat\chi)\sum_{u\in U / U ^ {\mathrm{alg}}} \sum_{a\in\mcal D} \widehat \chi(a) E(\mfrak c(a),u)(\widehat\chi)(x(a)).
\end{align}
The right hand side is computed in Proposition~4.9 of \cite{Hsieh mu}, from which we deduce \eqref{eq:3.8.7}.
\end{proof}
\par
\section{Transfer congruence}\label{s4}
In this section, we prove the transfer congruence, which is the main result of the current paper.
\subsection{The normalized diagonal embedding}\label{subsection:4.1}
In this subsection, we introduce the normalized diagonal embedding which we denote by~$\Delta$, and study its effect on~$q$-expansions on Hilbert modular forms. As before,~$\mscr F/\mscr F'$ is a degree~$p$ extension of totally real fields, and~$\Delta$ will be a slight modification of~$\widetilde \Delta$, which were introduced in Subsection~\ref{ss:rad}. Unfortunately, in order to have a good normalized~$\Delta$, we need to impose a condition on~$\mscr F/\mscr F'$ which is stronger than \eqref{P} imposed therein. Throughout this section, we will assume that
\begin{enumerate}
\labitem {P'}{P'} The relative different~$\mfrak d_{\mscr F/\mscr F'}$ is~$p\mscr O$.
\end{enumerate}
Clearly, \eqref{P'} implies \eqref{P}, which is introduced in Subsection~\ref{ss:rad}. Furthermore, whenever \eqref{P'} is true, we always take~$d_{\mscr F/\mscr F'}$ in \eqref{P} to be~$p$.
\begin{definition}
Recall that~$\widetilde \Delta$ was defined in \eqref{eq:res}. Define the normalized diagonal embedding~$\Delta$ to be 
\begin{align}
\Delta\colon  X_+'\times \GL_2(\adele_{\mscr F'}^\infty) &\to X_+\times \GL_2(\adele_{\mscr F}^\infty)
\\
x
&\mapsto
\begin{bmatrix}
p^{-1}&0
\\
0&p
\end{bmatrix}
\cdot  \widetilde\Delta(x).
\end{align}
\end{definition}
\par
It is easy to see that if~$f$ is an automorphic form on~$X_+\times \GL_2(\adele^\infty_{\mscr F})$, then the pull-back of~$f$ along~$\Delta$ which denote by~$\Delta^*f$ is an automorphic form on~$X_+'\times \GL_2(\adele^\infty_\mscr F)$. The strategy for the proof of our main result is to compare an automorphic form~$f$ on~$X_+\times \GL_2(\adele^\infty _\mscr F)$ and~$f'$ on~$X_+\times \GL_2(\adele^\infty_{\mscr F'})$ by comparing the~$q$-expansions of~$\Delta^*f$ and~$f'$. In order to do that, it is important to express the~$q$-expansion of~$\Delta^*f$ in terms of~$f$. Before we formulate it precisely, we introduce some notation. Let~$f$ be a holomorphic automorphic form on~$X_+\times \GL_2(\adele^\infty_\mscr F)$. Suppose that we are given a finite idele~$g^\infty \in \GL(\adele^\infty_\mscr F)$, and consider the function~$F$ on~$X_+$ which is defined by
\begin{align}
F(\tau) = f\left(\left(\tau,g^\infty\right)\right).
\end{align}
Let~$\tau_\sigma$ to be the~$\sigma$-component of~$\tau$ for each~$\sigma \in I$, then we have a Fourier expansion of the form
\begin{align}
F(\tau)=\sum_{\beta} a_\beta(F) \mathrm{exp}\left( 2\pi i\sum_{\sigma \in I} \sigma(\beta) \tau_\sigma\right)
\end{align}
for some~$a_\beta(F) \in \C$, where~$\beta$ runs over the set of all totally non-negative elements in~$\mscr F$. In fact, as a function of~$\beta$,~$a_\beta( F)$ is supported on a fractional ideal of~$\mscr F$, but it is irrelevant for our purpose to make that ideal explicit. Now consider a finite idele~$g'^\infty$ in~$\GL_2(\adele^\infty_{\mscr F'})$. Define~$f'$ to be~$\Delta^*f$ and define~$F'$ to be the restriction of~$f'$ to~$X_+'\times \{g'^\infty \}$. Then,~$F'$ will also have a Fourier expansion of the form
\begin{align}
F'(\tau')=\sum_{\beta'} a_\beta'(F') \mathrm{exp}\left( 2\pi i\sum_{\sigma' \in I'} \sigma'(\beta') \tau_{\sigma'}\right).
\end{align}
\begin{proposition}\label{prop:411}
In the situation above, assume in addition that
\begin{align}
g^\infty =
\begin{bmatrix}1&0\\0& \bs c^{-1}\end{bmatrix}
\text{, and }
g'^\infty =
\begin{bmatrix}1&0\\0& (\bs c')^{-1}\end{bmatrix},
\end{align}
and that the natural map~$\adele^\infty _{\mscr F'} \to \adele ^\infty _{\mscr F}$ sends~$\bs c'$ to~$\bs c$. Then, we have
\begin{align}
a_{\beta'}(F') =\sum_{\beta} a_\beta(F),
\end{align}
where the sum is taken over the set consisting of~$\beta \in \mscr F$ such that~$\mathrm{Tr}_{\mscr F/\mscr F'}(\beta) = p\beta'$.
\end{proposition}
\begin{proof}Let~$\tau'$ an arbitrary element in~$X_+'$, and let~$\tau$ be the element of~$X_+$ such that for each~$\sigma \in \Sigma$,~$\tau_\sigma$ equals to~$\tau_{\sigma'}'$ where~$\sigma'$ is the restriction of~$\sigma$ to~$\mscr F'$. In other words,~$\tau$ is the image of~$\tau'$ under the naive diagonal embedding~$X_+' \to X_+$. 
The assertion of the proposition follows from a straightforward calculation. Indeed, we have
\begin{align}
F'(\tau')
&=
f'
\left(\left(
\tau',
\begin{bmatrix}1&0\\0& (\bs c')^{-1}\end{bmatrix}
\right)\right)
\\
&=
\Delta^*f\left(\left(
\tau',
\begin{bmatrix}1&0\\0& (\bs c')^{-1}\end{bmatrix}
\right)\right)
\\
&=
f
\left(
\Delta\left(
\tau',
\begin{bmatrix}1&0\\0& (\bs c')^{-1}\end{bmatrix}
\right)\right)
\\
&=
f
\left(
\begin{bmatrix}p^{-1}&0\\0&p\end{bmatrix}
\cdot
\widetilde\Delta\left(
\tau,
\begin{bmatrix}1&0\\0& (\bs c')^{-1}\end{bmatrix}
\right)\right)
\\
&=
f
\left(
\begin{bmatrix}p^{-1}&0\\0&p\end{bmatrix}
\cdot
\left(
p\tau,
\begin{bmatrix}p&0\\0& (p\bs c)^{-1}\end{bmatrix}
\right)\right)
\\
&=
f
\left(\left(
p^{-1}\tau,
\begin{bmatrix}1&0\\0& \bs c^{-1}\end{bmatrix}
\right)\right)
\\
&=
F(p^{-1}\tau)
\\
&=
\sum_{\beta} a_\beta(F) \mathrm{exp}\left( 2\pi i\sum_{\sigma \in I} \sigma(\beta) p^{-1}\tau_\sigma\right)
\\
&=
\sum_{\beta} a_\beta(F) \mathrm{exp}\left( 2\pi i\sum_{\sigma' \in I'}\sum_{\sigma|_{\mscr F'}=\sigma'} \sigma(\beta) p^{-1}\tau_\sigma\right)
\\
&=
\sum_{\beta} a_\beta(F) \mathrm{exp}\left( 2\pi i\sum_{\sigma' \in I'}p^{-1}\tau_{\sigma'}\sum_{\sigma|_{\mscr F'}=\sigma'} \sigma(\beta) \right)
\\
&=
\sum_{\beta} a_\beta(F) \mathrm{exp}\left( 2\pi i\sum_{\sigma' \in I'}p^{-1}\tau_{\sigma'}\mathrm{Tr}_{\mscr F/\mscr F'}(\beta) \right)
\\
&=
\sum_{\beta'\in \mscr F'}\sum_{\beta\colon \mathrm{Tr}_{\mscr F/\mscr F'}(\beta)=\beta'} a_\beta(F) \mathrm{exp}\left( 2\pi i\sum_{\sigma' \in I'}p^{-1}\tau_{\sigma'}\beta' \right)
\\
&=
\sum_{\beta'\in \mscr F'}\left(\sum_{\beta\colon \mathrm{Tr}_{\mscr F/\mscr F'}(\beta)=p\beta'} a_\beta(F) \right) \mathrm{exp}\left( 2\pi i\sum_{\sigma' \in I'}\tau_{\sigma'}\beta' \right).
\end{align}
From this we conclude the assertion of the proposition.
\end{proof}
\subsection{Properties of Katz'~$p$-adic modular forms}\label{ss:42}
As promised in Remark~\ref{rmk:235}, we introduce the notion of~$p$-adic modular forms following Katz. It is stronger than Hida's definition which we stated as Definition~\ref{def:pmf}, in the sense that a~$p$-adic modular form in the sense of Hida is naturally a~$p$-adic modular form in the sense of Katz. In contrast to the elementary definition of Hida, which uses nothing more than~$q$-expansions, that of Katz is built upon the theory of integral models of Shimura varieties and the Igusa tower. We will not go into them, and we will just summarize some properties of Katz' modular form that we need. 
\begin{remark}
For the readers who are interested in the definition, not just properties, of Katz'~$p$-adic modular forms, we suggest consulting the paper \cite{Hsieh mu} by Hsieh, the original paper \cite{Katz} of Katz, and the book \cite{Hida p-adic} of Hida. However, there is some technical complications that need to be pointed out. The compactification of the integral model of Shimura variety and the Igusa tower that need to be constructed first appeared in the thesis of Rapoport. However, the thesis contained a mistake, which makes the proof of~$q$-expansion principle for~$p$-adic modular forms incomplete when~$p$ divides the discriminant of the totally real field being considered. Note that the excluded cases are precisely what interest us, while the most of literature deals with the opposite situation in which~$p$ is unramified in it. The work \cite{DP} of Deligne and Pappas fixes the problem, from which we can deduce the~$q$-expansion principle along the smooth locus of the Shimura variety. Over the ordinary locus, the two moduli problems defined by Deligne-Pappas and Rapoport agree, whence the cusps and the ordinary CM points belong to the same component of the special fiber modulo~$p$. In particular, we can apply the~$q$-expansion principle to the CM points~$x(a)$ that are considered in the current article.
\end{remark}

As before,~$\mscr F$ will denote a totally real field and we fix an ideal~$\mfrak c$ of~$\mscr F$, which plays the role of a polarization ideal. Let~$K$ be an open compact subgroup of~$\GL_2(\adele_\mscr F^\infty)$. We keep the conditions on~$K$ that we imposed in Subsection~\ref{ss:22}. Let~$\Q_p^{\mathrm{ur}}$ be the maximal unramified extension of~$\Q_p$, and let~$\mcal I$ be the~$p$-adic completion of the valuation ring of~$\Q_p^\mathrm{ur}$. For a~$p$-adic ring~$R$ which is also an~$\mcal I$-algebra, we define~$V(\mfrak c, K; R)$ to be the space of Katz~$p$-adic modular forms with coefficients in~$R$ defined in 2.5.3. of \cite{Hsieh mu}.
\par
We now want to explain some properties of the space~$V(\mfrak c, K; R)$ which will be needed later. We begin with~$q$-expansions. Recall that we defined the formal power series~$R[[q^L]]$ for a lattice~$L$ of~$\mscr F$ in \eqref{eq:225}. Then,~$V(\mfrak c, K; R)$ is endowed with the injective~$q$-expansion homomorphism
\begin{align}\label{eq:422}
V(\mfrak c, K ; R) &\hookrightarrow R[[q^L]]
\\
f &\mapsto \sum_{\beta} a_\beta(f ) q^\beta
\end{align}
for some~$L$. If we view~$a_\beta(F)$ as a function of a variable~$\beta \in \mscr F$, then it is supported on the set of totally positive elements in~$L$.
\par
The next property of Katz~$p$-adic modular form is its relation with classical automorphic forms. Recall that~$M_k(\mfrak c, K_1^n, R)$ was defined for a subring~$R$ of~$\C$ in \eqref{eq:228}. Assume that~$R_0$ is a subring of~$\C$ such that~$i_p(R_0)$ is contained in a subring~$R$ of~$\C_p$ which is again an~$\mcal I$-algebra. Then we have a natural map
\begin{align}\label{eq:424}
M_k(\mfrak c, K_1^n , R_0) &\to V(\mfrak c, K; R)
\\
f &\mapsto \widehat f
\end{align}
and~$\widehat f$ is called the~$p$-adic avatar of~$f$. The~$p$-adic avatar~$\widehat f$ of a Hilbert modular form~$f$ in~$M_k(\mfrak c , K_1^n, R_0)$ is characterized the property that the~$q$-expansion of~$f$ in as a Hilbert modular form is equal to the~$q$-expansion of~$\widehat f$ as a~$p$-adic modular form.
\par
The next property of Katz'~$p$-adic modular form is that we have a natural injective map from the space of~$p$-adic modular forms in the sense of Hida. Recall that the latter was defined in Definition~\ref{def:pmf}, and~$V_0(\mfrak c, K; R)$ denotes the space of them. Assume that~$R$ is an~$\mcal I$-algebra. We have a canonical map
\begin{align}\label{eq:426}
V_0(\mfrak c, K; R) \to V(\mfrak c, K; R)
\end{align}
which is characterized by the preservation of the~$q$-expansion. That is, denoting by~$i$ the inclusion \eqref{eq:426} temporarily, to say that for every~$f \in V_0(\mfrak c, K; R)$, the~$q$-expansion of~$i(f)$ is equal to the formal power series that defines~$F$. From now, we identify~$V_0(\mfrak c, K; R)$ as a subspace of~$V(\mfrak c, K; R)$.
\par
Another property of Katz'~$p$-adic modular forms that distinguish it from Hida's is that it is functorial in the coefficients. Let~$R$ and~$R'$ be two~$p$-adic rings that are algebras over~$\mcal I$. For any~$\mcal I$-algebra homomorphism~$\alpha \colon R \to R'$, we have a canonical map
\begin{align}
\alpha_* \colon V(\mfrak c, K ; R) \longrightarrow V(\mfrak c, K ; R').
\end{align}
In terms of~$q$-expansions,~$\alpha_*$ is characterized by the natural map~$R[[q]] \to R'[[q]]$ induced by~$\alpha$.
\par
The last property of Katz'~$p$-adic modular form that we will need is the evaluation homomorphism. For each CM point~$x(a)$, we have
\begin{align}
V(\mfrak c, K ; R) \longrightarrow R
\\
F \mapsto F(x(a))
\end{align}
which satisfies the following properties. Firstly, for~$\alpha \colon R \to R'$, we have
\begin{align}
\alpha(F(x(a)))=(\alpha_*F)(x(a)).
\end{align}
Also, we have
\begin{align}
f(x(a))=\widehat f(x(a)).
\end{align}
\subsection{A map between~$p$-adic modular forms}
In this subsection, we keep the assumptions from the previous section, and we study a map between~$p$-adic modular forms induced by~$\Delta$. In particular, we work with a degree~$p$ extension~$\mscr F/ \mscr F'$ such that the relative different~$\mfrak d_{\mscr F/ \mscr F'}$ is generated by~$p$. Also, we fix quadratic CM extension~$\mscr M'/\mscr F'$ with CM type~$\Sigma'$, and define~$\mscr M$ to be~$\mscr M'\mscr F$. Let~$\delta'$ be an element in~$\mscr M'$ satisfying the polarization conditions \eqref{pol-1}, \eqref{pol-2}, and \eqref{pol-3} for~$\mscr M'/\mscr F'$. Let~$\mfrak c'$ be the ideal of~$\mscr F'$ associated to~$\delta'$ by the condition \eqref{pol-3}. Let~$\delta$ be~$p\delta' \in \mscr M$. Then, by our discussion in Subsection~\ref{ss:rad},~$\delta$ satisfies the polarization conditions for~$\mscr M/\mscr F$. Also, the polarization ideal~$\mfrak c$ associated to~$\delta$ equals~$\mfrak c' \cdot \mscr O$.
\par
We will use the notation from Subsection~\ref{ss:pmf} and Section~\ref{s3}. Furthermore, we take ideals~$\mfrak C$ of~$\mscr F$ and~$\mfrak C'$ in~$\mscr F'$ in such a way that~$\mfrak C' \mscr O = \mfrak C$. Let~$K$ (resp.~$K'$) be the open compact subgroup in~$\GL_2(\adele_\mscr F^\infty)$ (resp.~$\GL_2(\adele^\infty _{\mscr F'})$) introduced in Section~\ref{s3}. By shrinking~$K'$ if necessary,~$\Delta^*$ induces a map
\begin{align}
\Delta^* \colon M_k(\mfrak c, K_1^n, \C) \to M_{pk}(\mfrak c', K_1'^n, \C).
\end{align}
Recall that~$V(\mfrak c, K; R)$ denotes the space of~$R$-adic modular forms. Let~$R_0$ be a subring of~$\overline \Q$ such that~$i_p(R_0)$ is contained in the valuation ring of~$\C_p$, and let~$R$ be the completion of~$i_p(R_0)$ in~$\C_p$. Then there is a canonical embedding
\begin{align}
M_k(\mfrak c, K_1^n , R) &\to V(\mfrak c, K; R)
\\
f &\mapsto \widehat f
\end{align}
which is compatible with the~$q$-expansions on both sides. We claim that~$\Delta^*$ in fact induces a map between~$\Lambda$-adic forms
\begin{align}
\Delta^* \colon V_0(\mfrak c, K; \Lambda) \longrightarrow V_0(\mfrak c', K'; \Lambda),
\end{align}
Indeed, if~$\widehat \chi \in \frak X$,~$f\in V_0(\mfrak c, K ; \Lambda)$, then~$(\Delta^*(f))(\widehat \chi)$ equals~$\Delta^*(f(\widehat\chi))$ at least as two formal~$q$-expansions. Since~$f(\widehat\chi)$ belongs to~$M_{k}(\mfrak c, K, \C)$ for a Zariski dense set of characters~$\widehat \chi$,  it follows that~$(\Delta^*(f))(\widehat \chi')$ belongs to~$M_{kp}(\mfrak c' , K' , \C)$ for such characters~$\widehat \chi$. It follows that that~$\Delta^*f$ belongs to~$V_0(\mfrak c', K'; \Lambda)$.

\subsection{Proof of the main theorem}\label{ss:44}
In this subsection, assuming the condition \eqref{P} and another condition \eqref{C} which is defined below, we prove Theorem~\ref{thm:main}, the main theorem of the current paper. We will introduce two~$p$-adic modular forms and deduce the transfer congruence between~$p$-adic~$L$-functions from the congruence between these two~$p$-adic modular forms and applying the~$q$-expansion principle. We begin with formulating the~$q$-expansion principle that we are going to use.
\begin{theorem}[$q$-expansion principle]\label{q-exp}
Let~$R$ be a~$p$-adic ring. Let~$f$ be an~$R$-adic modular form in~$V(\mfrak c, K; R)$. Suppose that~$f$ has~$q$-expansion
\begin{align}
f (q) = \sum_\beta a_\beta (f) q^\beta.
\end{align}
Let~$x(1)$ be the CM point defined in Subsection~\ref{ss:cmp}. Suppose that there is an ideal~$J$ of~$R$ such that for all~$\beta$,~$a_\beta(f)$ belongs to~$J$. Then we have
\begin{align}
f (x(1)) \in J.
\end{align}
\end{theorem}
\begin{proof}
Let~$\mcal O_{\C_p}$ be the valuation ring of~$\C_p$. Let~$Y$ be a Zariski dense set of homomorphisms~$\upphi \colon R \to \mcal O_{\C_p}$ such that~$f(\upphi)$ is has a~$q$-expansion
\begin{align}
f (\upphi) = \sum_{\beta} i_p (a_\beta(f_{\upphi}))q ^\beta
\end{align}
for some~$f_\upphi \in M_k(\mfrak c, K_1^n, i_p^{-1}(\mcal O_{\C_p}))$. Such~$Y$ exists by definition of~$R$-adic modular forms. Let~$J_\upphi$ be the ideal of~$\mcal O_{\C_p}$ generated by the image of~$J$ under~$\upphi\colon R \to \mcal O_{\C_p}$. Then, of course~$a_\beta(f(\upphi))\equiv 0$ modulo~$J_\upphi$. As~$x(1)$ is an integral point on the Shimura variety, the~$q$-expansion principle implies that~$f_\upphi ( x(1))$ belongs to~$J_\upphi$. Now~$f (x(1))$ is, by definition, the element of~$R$ such that~$\upphi( f(x(1))) = f_\upphi(x(1))$. Such~$f(x(1))$ is unique since~$Y$ is Zariski dense for~$R$. Then it follows that~$f( x(1))\in J$. 
\end{proof}
Now we define the relevant~$\Lambda$-adic forms. Fix a set of representatives for~$U/ U_{\mathrm{alg}}$, and introduce a linear combination of Eisenstein series
\begin{align}
E(\mfrak c ) = \mcal C \cdot \sum_{u\in U/ U_{\mathrm{alg}}}\sum_{a\in \mcal D} \rec_\mscr M(a)E(\mfrak c(a) ; u )|[a] \in V(\mfrak c, K; \Lambda).
\end{align}
By the action of the operator~$|[a]$ on~$\mbb E^{\mathrm h}_{\chi,u}$ described in Proposition~\eqref{prop:toric},~$E(\mfrak c)$ is independent of the choice of~$\mcal D$. We record an important property of~$E(\mfrak c)$.
\begin{proposition}
The~$p$-adic~$L$-function~$\mcal L$ of Definition~\eqref{def:3.8.1} satisfies
\begin{align}
E(\mfrak c) (x(1)) = \mcal L
\end{align}
where~$\mfrak c$ is associated to~$\delta$ by \eqref{pol-3}.
\end{proposition}
\begin{proof}
For~$\widehat \chi \in \mfrak X_+$, we have
\begin{align}
E(\mfrak c) (x(1))(\widehat \chi)
&= 
\widehat\chi\left(\mcal C\right)\sum_{u\in U/ U^{\mathrm{alg}}}\sum_{a\in \mcal D} \widehat \chi(a) E(\mfrak c(a) ; u )(\widehat \chi)|[a] (x(1))
\\
&=
\widehat\chi\left(\mcal C\right)\sum_{u\in U/ U^{\mathrm{alg}}}\sum_{a\in \mcal D} \widehat\chi(a)E(\mfrak c(a);u)(\widehat \chi)(x(a))
\\
&=
\int_Z \widehat \chi d\mcal L.
\end{align}
The second equality follows from the definition of the operator~$|[a]$. Indeed, we defined~$|[a]$ to be the translation~$\big | (\varsigma^{-1}_f \varrho(a) \varsigma_f)$, which sends~$x(1)$ to~$x(a)$ according to our definition of CM points. The third equality directly follows from the definition of~$\mcal L$. 
\end{proof}
\par
The transfer congruence, which is the main result of current paper is the following. We remind the reader that the condition \eqref{P'} was introduced in the beginning of Subsection~\ref{subsection:4.1}, and the condition \eqref{C} will be defined in \eqref{eq:C}. See Remark~\ref{rmk:4.5.24} for a discussion on these two conditions.
\begin{theorem}\label{thm:main}
Suppose that the conditions \eqref{P} and \eqref{C} hold. Then we have the transfer congruence
\begin{align}
\mcal L \equiv \ver(\mcal L') \hspace{5mm}\textrm{(mod~$T$)}.
\end{align}
\end{theorem}
The rest of the current subsection will be devoted to the proof of above theorem. We outline the strategy of the proof. In order to show that 
\begin{align}\label{eq:tr}
\mcal L \equiv \ver(\mcal L')\hspace{5mm}(\mathrm{mod}~T)
\end{align}
it suffices to show that
\begin{align}\label{eq:172}
\Delta^*E(\mfrak c ) (x'(1)) \equiv \ver_* E'(\mfrak c')(x'(1)) \hspace{5mm}(\mathrm{mod}\,\,T)
\end{align}
because~$\Delta(x'(1))=x(1)$. By applying the~$q$-expansion principle, in the form of Theorem~\ref{q-exp}, applied to~$\Delta^*E(\mfrak c)$ and~$\ver_*E'(\mfrak c')$, \eqref{eq:172} follows from
\begin{align}\label{eq:175}
a_{\beta'} \left(\Delta^* E(\mfrak c ) \right)\equiv a_{\beta'}\left(\ver_*E'(\mfrak c')\right)\hspace{5mm}(\mathrm{mod}\,\,T)
\end{align}
for every~$\beta'\in\mscr F'$. We will prove \eqref{eq:175}.
\par
Let us write the~$\beta$-th Fourier coefficient of~$E(\mfrak c(a);u)$ as
\begin{align}
a_\beta(E(\mfrak c(a);u)) = A(\beta, \mfrak c(a);u)
\end{align}
and similarly define~$A'(\beta', \mfrak c'(a');u')$ for~$E'(\mfrak c'(a');u')$. In particular, if we define~$B(\beta, \mfrak c)$ and~$B'(\beta', \mfrak c')$ to be the~$\beta$-th Fourier coefficient of~$E(\mfrak c)$ and the~$\beta'$-th Fourier coefficient of~$E'(\mfrak c')$ respectively, then 
\begin{align}
B(\beta, \mfrak c ) =\mcal C \cdot \sum_{a\in \mcal D}\rec_\mscr M(a)\sum_{U/U_\mathrm{alg}}A(\beta, \mfrak c(a); u)
\end{align}
holds, and similarly 
\begin{align}
B'(\beta', \mfrak c' ) = \mcal C' \sum_{a'\in \mcal D'}\rec_{\mscr M'}(a')\sum_{U'/U'_\mathrm{alg}}A'(\beta', \mfrak c'(a'); u')
\end{align}
holds. By the description about the effect of~$\Delta$ on~$q$-expansion given in Proposition~\ref{prop:411}, the assertion of~\eqref{eq:175} is equivalent to the assertion that for each totally positive~$\beta' \in \mscr F'$
\begin{align}\label{eq:FCTC}
\sum_{\beta} B(\beta,\mfrak c)
\equiv
\ver\left( B'(\beta ' , \mfrak c')\right)\hspace{5mm}(\mathrm{mod}\,\,T)
\end{align}
holds, where the sum indexed by~$\beta$ is taken over the set~$\{\beta |\beta \in \mscr F,\, \mathrm{Tr}_{\mscr F/ \mscr F'} (\beta) = p\beta ' \}$.
\par
Either side of \eqref{eq:FCTC} does not depend on the choice of representatives that we make, but it will be convenient for us to fix a good representatives for the computation. Precisely speaking, we will regard~$Cl_-$ (resp.~$U'/U'_\mathrm{alg}$) as a subset of~$Cl_-$(resp.~$U/ U_{\mathrm{alg}}$), which is possible since the natural diagonal map induces an embedding. We first choose representatives for~$U'/ U'_\mathrm{alg}$ and~$\mcal D'$. Then we choose the representatives for~$U/ U_\mathrm{alg}$ and~$\mcal D$ by extending the chosen ones. Furthermore, the representatives for~$\mcal D$ and~$U/U_{\mathrm{alg}}$ can be chosen to be~$G$-sets, where~$G$ is the Galois group~$\Gal(\mscr M/\mscr M')$. In the rest of the subsection, we fix such representatives.
\par
The next two propositions are the criterions which we will use to prove \eqref{eq:FCTC}.
\begin{proposition}\label{prop:4.4.18}
Recall that~$G = \Gal(\mscr M'/\mscr M)$ acts on~$Z$, which extends to an action on~$\Lambda$. Suppose we have an element~$\lambda \in \Lambda$ of the form
\begin{align}
\lambda = \sum_{\gamma\in W} \lambda_\gamma
\end{align}
where the sum is taken over a finite set~$W$ with~$G$-action, and~$\lambda_\gamma\in \Lambda$. Let the action of~$G$ on~$\gamma \in W$ be written as~$\gamma\mapsto \gamma^g$ for~$g\in G$. Let~$W^G$ be the subset of~$W$ consisting of the elements of~$W$ that are fixed by~$G$. Then~$\lambda$ belongs to the trace ideal~$T$ if
\begin{align}
g \lambda_\gamma = \lambda_{\gamma^g}
\end{align}
for all~$\gamma\in W-W^G$ and all~$g\in G$, and
\begin{align}
\lambda_\gamma \in p\Lambda
\end{align}
for all~$\gamma \in W^G$.
\end{proposition}
\begin{proof}
Define
\begin{align}
\overline \lambda_\gamma
=
\begin{cases}
\lambda_\gamma &\text{ if } \gamma \in W-W^G
\\
p^{-1}\lambda_\gamma &\text { if } \gamma \in W^G
\end{cases}
\end{align}
for each~$\gamma \in W$. By the assumptions,~$\overline \lambda_\gamma$ is well defined and belongs to~$\Lambda$. It is easy to see that~$\lambda$ is the image of
\begin{align}
\sum_{\gamma \in W/G}\overline \lambda_\gamma
\end{align}
under the trace map, hence the proof is complete.
\end{proof}
A variant of the previous proposition is the following.
\begin{proposition}\label{prop:4.4.24}
Keep the notation from Proposition~\ref{prop:4.4.18}. Identify~$Z$ with the image of~$Z$ in~$\Lambda$ under the natural map~$Z\to\Lambda$. Suppose we have an element~$\lambda \in \Lambda$ of the form
\begin{align}
\lambda = \sum_{\gamma\in W} a_\gamma z_\gamma
\end{align}
 Then~$\lambda$ belongs to the trace ideal~$T$ if
\begin{align}\label{eq:4.4.20}
g z_\gamma = z_{\gamma^g}
\end{align}
for any~$\gamma\in W$ and any~$g\in G$,
\begin{align}\label{eq:4.4.21}
a_\gamma = a_{g\gamma}
\end{align}
for any~$\gamma \in W - W^G$, and
\begin{align}\label{eq:4.4.22}
a_\gamma \in p\Z_p
\end{align}
for any~$\gamma \in W^G$.
\end{proposition}
\begin{proof}
It is a special case of Proposition~\ref{prop:4.4.18}, where we have taken~$\lambda_\gamma$ to be of the form~$a_\gamma z_\gamma$.
\end{proof}
We introduce a condition from which we deduce \eqref{eq:FCTC}, which will be proved later under a mild assumption.
\begin{itemize}
\labitem{C}{C} Let~$G = \Gal(\mscr M / \mscr M')$. Then
\begin{align}\label{eq:C}
\mcal C \cdot\sum_{a \in \mcal D^G}\rec_\mscr M(a)
=
\ver\left(
\mcal C '\cdot
\sum_{a'\in \mcal D'}\rec_{\mscr M'}(a')\right)
\end{align}
holds.
\end{itemize}
\begin{proposition}\label{prop:4.4.30}
Assume \eqref{C}. Then 
\eqref{eq:FCTC} holds.
\end{proposition}
\begin{proof}
Unfolding the definition, \eqref{eq:FCTC} is
\begin{align}\label{eq:4.4.25}
\mcal C\cdot \sum_{u\in U/ U_{\mathrm{alg}}}\sum_a \sum_\beta A(\beta, \mfrak c(a); u)
\equiv
\ver\left( \mcal C'\cdot \sum_{u'\in U'/ U'_{\mathrm{alg}}}\sum_{a'} A'(\beta ', \mfrak c ' (a') ; u')\right)
\end{align}
where the sum on the left hand side is indexed over the triples~$(\beta,a,u)$, where~$u\in U/U_{\mathrm{alg}}$,~$a \in \mcal D$, and~$\beta\in \mscr F$ such that~$\mathrm{Tr}(\beta) = p \beta'$, and the sum on the right hand side is indexed over the pairs~$(u',a')$ where~$u' \in U'/U'_{\mathrm{alg}}$ and~$a' \in \mcal D'$.
\par
We first show~$G$ acts on the term~$A (\beta, \mfrak c(a);u)$ by
\begin{align}\label{eq:4.4.26}
g A (\beta, \mfrak c(a);u) =  A (\beta^g, \mfrak c(a^g);u^g).
\end{align}
Recall the definitions
\begin{align}\label{eq:4.4.27}
 A^{(p)}(\beta,\mfrak c)
&=
\rec^\infty(\beta)\mathrm{Norm}_{\mscr F/\Q}(\beta^{-1})\prod_{v\nmid p}  A(\beta,v,\bs c_v),
\\
A(\beta,\mfrak c;u)
&=
A^{(p)}(\beta,\mfrak c)\rec_{\Sigma_p}(\beta) \mbb I_{u(1+\varpi_p \mscr O_p)}(\beta),
\end{align}
from which we deduce that
\begin{align}
&gA(\beta,\mfrak c;u)
\\
=&gA^{(p)}(\beta,\mfrak c)\times g\rec_{\Sigma_p}(\beta) \times \mbb I_{u(1+\varpi_p \mscr O_p)}(\beta)
\\
=&gA^{(p)}(\beta,\mfrak c)\times \rec_{\Sigma_p}(\beta^g) \times \mbb I_{u^g(1+\varpi_p \mscr O_p)}(\beta^g).
\end{align}
Therefore \eqref{eq:4.4.26} is reduced to
\begin{align}
gA^{(p)}(\beta,\mfrak c(a)) = A^{(p)}(\beta^g,\mfrak c(a^g)),
\end{align}
which is further reduced, using \eqref{eq:4.4.27}, to the assertion that
\begin{align}\label{eq:017}
g\prod_{v\mid v'}A (\beta, v, \bs c(a)_v) = 
\prod_{v\mid v'}A (\beta^g, v, \bs c(a^g)_{v})
\end{align}
holds for each~$v' \nmid p$. Indeed, first we observe
\begin{align}
g\prod_{v\mid v'}A (\beta, v, \bs c(a)_v) 
=&
\prod_{v\mid v'}gA (\beta, v, \bs c(a)_v) 
\\
=&
\prod_{v\mid v'}A (\beta^g, v^g, \left(\bs c(a)_v\right)^g) 
\end{align}
where the second equality follows from formulae in Lemma~\ref{lemma:361}. Continuing the computation, we obtain
\begin{align}
\prod_{v\mid v'}A (\beta^g, v^g, \left(\bs c(a)_v\right)^g) 
=&
\prod_{v\mid v'}A (\beta^g, v^g, \bs c(a^g)_{v^g}) 
\\
=&
\prod_{v^g\mid v'}A (\beta^g, v^g, \bs c(a^g)_{v^g}) 
\\
=&
\prod_{v\mid v'}A (\beta^g, v, \bs c(a^g)_{v}) 
\end{align}
where we permuted the index set by~$v\mapsto v^g$ in the second equality.
\par
Define the index set~$W$ to be
\begin{align}
W = \{(\beta, a, u)| \mathrm{Tr}_{\mscr F/ \mscr F'}(\beta) = p\beta', a \in \mcal D, u \in U/U_{\mathrm{alg}}\}
\end{align}
on which~$G$ acts by
\begin{align}
g(\beta, a, u)= (\beta^g, a^g, u^g.)
\end{align}
In other words, the triple summation appearing in the left had side of \eqref{eq:4.4.25} is taken over the set~$W$. By Proposition~\ref{prop:4.4.24}, the validity of \eqref{eq:4.4.26} implies that
\begin{align}
\sum_{(\beta,a,u)\in W-W^G}A(\beta,\mfrak c(a); u) \in T
\end{align}
and it follows that
\begin{align}
\mcal C \cdot \sum_{(\beta,a,u)\in W-W^G}A(\beta,\mfrak c(a); u) \in T
\end{align}
because~$T$ is an ideal.
\par
Now it is sufficient to prove
\begin{align}\label{eq:02}
\mcal C \cdot \sum_{(\beta,a,u)\in W^G}
 \rec_\mscr M(a)
 A(\beta, \mfrak c(a); u)
\equiv
\ver\left(\mcal C' \cdot \sum_{(u',a')}
\rec_{\mscr M'}(a') A' (\beta', \mfrak c'(a');u')\right)
\hspace{5mm} \text{(mod~$T$)}
\end{align}
in order to show Proposition~\ref{prop:4.4.30}. In fact, the terms on the left hand side with~$u\not \in U'$ are zero, since~$\beta ^g = \beta$ and~$\mathrm{Tr}(\beta)=p\beta'$ imply~$\beta = \beta'$, and for those~$u$ we have~$\mbb I_{u(1 + \varpi_v \mscr O_v)}(\beta)=0$, whence~$ A(\beta, \mfrak c(a); u)$ is zero. We have reduced \eqref{eq:02} to the assertion
\begin{align}\label{eq:4.4.50}
\mcal C \cdot \sum_{(u',a)}
\rec_\mscr M(a) A (\beta', \mfrak c(a); u)
\equiv
\ver\left(\mcal C' \cdot \sum_{(u',a')}
\rec_\mscr M(a)  A' (\beta', \mfrak c'(a');u')\right)
\hspace{5mm} \text{(mod~$T$)},
\end{align}
where~$u'$ varies in~$U'$ in both sides, and~$a'$ varies in~$\mcal D'$ on the right hand side, and~$a$ varies among elements of~$\mcal D$ such that~$\mfrak c(a^g) = \mfrak c(a)$. By identifying~$\mcal D^G = \mcal D'$, we write
\begin{align}\label{eq:4.4.51}
\mcal C \cdot \sum_{(u',a')}
\rec_\mscr M(a) A (\beta', \mfrak c(a'); u)
\equiv
\ver\left(\mcal C' \cdot \sum_{(u',a')}
\rec_\mscr M(a)  A' (\beta', \mfrak c'(a');u')\right)
\hspace{5mm} \text{(mod~$T$)},
\end{align}
and the index sets on both sides are the same. In particular, we can further reduce \eqref{eq:4.4.51} to the assertion that for each~$u' \in U'$
\begin{align}\label{eq:03}
\mcal C \cdot \sum_{a'}
\rec_\mscr M(a') A (\beta', \mfrak c(a'); u')
\equiv
\ver\left(\mcal C' \cdot \sum_{a'}
\rec_\mscr M(a)  A' (\beta', \mfrak c'(a');u')\right)
\hspace{5mm} \text{(mod~$T$)}
\end{align}
holds. We rewrite \eqref{eq:03} equivalently in the form
\begin{align}\label{eq:4.4.53}
&\mcal C \cdot \sum_{a'}
\rec_\mscr M(a')\times
\mbb I_{u'(1+\varpi_p \mscr O_p)}(\beta')\times
\rec_{\Sigma_p}(\beta')\times
\prod_{v'\nmid p}
\left(\prod_{v\mid v'} A (\beta',v, \bs c(a')_v)\right)
\notag
\\
\equiv
&\ver\left(
\mcal C' \cdot \sum_{a'}
\rec_\mscr M(a') \times
\mbb I_{u'(1+\varpi'_{p} \mscr O_{p}')}(\beta')\times
\rec_{\Sigma'_p}(\beta')\times
\prod_{v'\nmid p} A (\beta', v',\bs c'(a')_{v'})\right)
\hspace{5mm} \text{(mod~$T$)}.
\end{align}
Since we assumed \eqref{C}, it suffices to show that
\begin{align}\label{eq:04}
\mbb I_{u'(1+\varpi_p \mscr O_p)}(\beta')\rec_{\Sigma_p}(\beta')
=
\ver\left(\mbb I_{u'(1+\varpi'_p \mscr O_{p}')}(\beta')\rec_{\Sigma'_p}(\beta')\right)
\end{align}
holds, and for each~$v' \nmid p$ and each~$a' \in \mcal D'$,
\begin{align}\label{eq:4.4.29}
\prod_{v\mid v'} A (\beta',v, \bs c(a')_v)
\equiv
\ver\left(
A (\beta', v',\bs c'(a')_{v'})
\right)
\hspace{5mm} \text{(mod~$T$)}
\end{align}
holds. Indeed, we deduce
\begin{align}
&\mcal C \cdot \sum_{a \in \mcal D^G}\rec_\mscr M(a)
\prod_{v'\nmid p}
\left(\prod_{v\mid v'} A (\beta',v, \bs c(a)_v)\right)
\\
\label{eq:4.4.58b}
\equiv&
\mcal C' \cdot \sum_{a' \in \mcal D'}\rec_{\mscr M'}(a')
\cdot\ver\left(
\prod_{v'\nmid p}
A (\beta',v', \bs c'(a')_{v'})
\right)
\\
\equiv&
\ver\left(
\mcal C' \cdot \sum_{a' \in \mcal D'}\rec_{\mscr M'}(a')
\right)
\cdot
\ver\left(
\prod_{v'\nmid p}
A (\beta',v', \bs c'(a')_{v'})
\right)
\hspace{5mm} \text{(mod~$T$)}
\end{align}
where the first and the second congruence follows from \eqref{eq:4.4.29} and \eqref{C}, respectively. By multiplying \eqref{eq:04} to the congruence
\begin{align}
&\mcal C \cdot \sum_{a \in \mcal D^G}\rec_\mscr M(a)
\prod_{v'\nmid p}
\left(\prod_{v\mid v'} A (\beta',v, \bs c(a)_v)\right)
\\
\equiv&
\ver\left(
\mcal C' \cdot \sum_{a' \in \mcal D'}\rec_{\mscr M'}(a')
\prod_{v'\nmid p}
A (\beta',v', \bs c'(a')_{v'})
\right)
\hspace{5mm} \text{(mod~$T$)}
\end{align}
we obtain \eqref{eq:4.4.53}, which is equivalent to \eqref{eq:03}.
\par
From the definition of~$u$ and~$u'$ \eqref{eq:04} follows immediately, and we now prove \eqref{eq:4.4.29} by dividing it into several lemmas.
\begin{lemma}
Suppose~$v'$ splits in~$\mscr F$,~$v'$ splits in~$\mscr M'$, and the place~$w'$ of~$\mscr M'$ lying above~$v'$ divides~$\mfrak F$. Then, the assertion of \eqref{eq:4.4.29} is true.
\end{lemma}
\begin{proof}
From Lemma~\ref{lemma:361}, we observe that the assertion of \eqref{eq:4.4.29} is just
\begin{align}\label{eq:4.4.58}
\prod_{v\mid v'}
\rec_w(\beta) \mbb I_{\mscr O_v^\times}(\beta')
\equiv
\ver\left(
\rec_{w'}(\beta') \mbb I_{{\mscr O'_{v'}}^\times}(\beta')
\right)
\hspace{5mm} \text{(mod~$T$)},
\end{align}
which follows immediately from the definition of~$\ver$ in terms of ideles. In fact, the equality holds in that case as well.
\end{proof}
\begin{lemma}\label{lemma:4.4.63}
Suppose~$v'$ splits in~$\mscr F$, and~$v'$ does not divide~$\mfrak D$. Further assume that~$v'$ is not ramified in~$\mscr F$, which is satisfied if we assume \eqref{P}. Then, the assertion of \eqref{eq:4.4.29} is true.
\end{lemma}
\begin{proof}
If we use the formulae in Lemma~\ref{lemma:361}, in this case \eqref{eq:4.4.29} becomes
\begin{align}
\label{eq:4.4.59}
&\prod_{v\mid v'}
\left(
\sum_{j=0}^{\val_v(\beta'\bs c(a')_v)}\rec_w(\varpi_w^j \bs c(a')_v) |\varpi_w|_w ^{-j} \mathbb I_{\mscr O_v}(\beta' \bs c(a')_v)
\right)
\\
\equiv
&\ver\left(
\sum_{j=0}^{\val_{v'}(\beta'\bs c'(a')_{v'})}\rec_{w'}({\varpi_{w'}}^j \bs c'(a')_{v'}) |\varpi'_{w'}|_{w'} ^{-j} \mathbb I_{\mscr O'_{v'}}(\beta' \bs c'(a')_{v'})
\right)
\hspace{5mm} \text{(mod~$T$)}
\end{align}
because~$v'$ does not divide~$\mfrak D$. We will apply Proposition~\ref{prop:4.4.24} after rearranging the terms of  \eqref{eq:4.4.59}. Define
\begin{align}
n_v =\val_v(\beta ' \bs c(a')_v)
\end{align}
and we take the set~$W$ to be
\begin{align}
W = \{(j_v)_{v|v'} |j_v=0,1,2,\cdots, n_v\}
\end{align}
consisting of tuples of integers indexed by the set of places~$v$ of~$\mscr F$ dividing~$v'$. We equip~$W$ with the~$G$-action
\begin{align}\label{eq:4.4.68}
\left((j_v)_{v|v'}\right)^g = (j_{v^{g^{-1}}})_{v|v'}
\end{align}
given by the opposite action of~$G$ on the indices. Expanding the product in \eqref{eq:4.4.59}, we obtain the sum
\begin{align}
\sum_{j\in W}\lambda_j 
\end{align}
where~$j$ denotes the tuple~$(j_v)_{v|v'}$, and for each~$j\in W$,~$\lambda_j$ is given by
\begin{align}
\prod_{v|v'}\rec_w(\varpi_w^{j_v} \bs c(a')_v) |\varpi_w|_w ^{-{j_v}} \mathbb I_{\mscr O_v}(\beta' \bs c(a')_v).
\end{align}
By routine calculations, letting~$h=g^{-1}$ for simplicity, we obtain
\begin{align}
g\lambda_j
&=
g\prod_{v|v'}\rec_w(\varpi_w^{j_v} \bs c(a')_v) |\varpi_w|_w ^{-{j_v}} \mathbb I_{\mscr O_v}(\beta' \bs c(a')_v)
\\
&=
\prod_{v|v'}\rec_{w^g}(\varpi_{w^g}^{j_v} \bs c(a')_{v^g}) |\varpi_{w^g}|_{w^g} ^{-{j_v}} \mathbb I_{\mscr O_{v^g}}(\beta' \bs c(a')_{v^g})
\\
&=
\prod_{v|v'}\rec_w(\varpi_w^{j_{v^{h}}} \bs c(a')_v) |\varpi_w|_w ^{-{j_{v^{h}}}} \mathbb I_{\mscr O_v}(\beta' \bs c(a')_v)
\\
&=
\lambda_{j^g},
\end{align}
which is the first condition of Proposition~\ref{prop:4.4.18}. We deduce that
\begin{align}
\sum_{j \in W-W^G}\lambda_j \in T
\end{align}
and it remains to show that
\begin{align}\label{eq:4.4.76}
\sum_{j \in W^G} \lambda_j \equiv
\ver\left(
\sum_{j=0}^{\val_{v'}(\beta'\bs c'(a')_{v'})}\rec_{w'}({\varpi_{w'}}^j \bs c'(a')_{v'}) |\varpi'_{w'}|_{w'} ^{-j} \mathbb I_{\mscr O'_{v'}}(\beta' \bs c'(a')_{v'})
\right)
\hspace{5mm} \text{(mod~$T$)}
\end{align}
holds. Firstly, it is clear that~$W^G$ consists of the parallel tuples~$(j_v)_{v|v'}$ such that~$j_v = n$ for some~$n=0,1,2,\cdots, \val_{v}(\beta' \bs c(a')_v)$. By our assumption,~$v'$ is unramified in~$\mscr F$. Noting that our~$\bs c$ is fixed by~$G$, we observe
\begin{align}
\val_{v'}(\beta' \bs c(a')_{v'}) = \val_v(\beta' \bs c(a')_v)
\end{align}
from which we deduce that \eqref{eq:4.4.76} holds as an equality. The proof of the lemma is complete.
\end{proof}

\begin{lemma}\label{lemma:4.4.78}
Suppose~$v'$ splits in~$\mscr F$, and~$v'$ is either inert or ramified in~$\mscr M'$. Then, the assertion of \eqref{eq:4.4.29} is true.
\end{lemma}
\begin{proof}
According to the definition of~$A(\beta', v, \bs c(a')_v)$,
\begin{align}\label{eq:4.4.79}
A(\beta',v,\bs c(a')_v) = 
\psi_v\left(\frac{-t_v}{2d_{\mscr F_v}}\right)\mathrm{Vol}(\mfrak l^{j_1}, dx)\sum_{x\in \mfrak l^{-j_0}/\mfrak l^{j_1}} \rec_v (x + 2^{-1} \bs \theta_v )^{-1} \psi_v \left( \frac{-\beta' x}{d_{{\mscr F_v}}}\right)
\end{align}
given in Lemma~\ref{lemma:3.5.6}, for any pair of sufficiently large integers~$j_0$ and~$j_1$. Since it does not depend on the finite idele~$\bs c(a')_v$, and we will consider a fixed~$\beta'$, we simply write
\begin{align}
A(v)=A(\beta', v, \bs c(a')_v)
\end{align}
and similarly write
\begin{align}
A'(v')=A'(\beta', v', \bs c'(a')_{v'})
\end{align}
temporarily. Recall that we defined Remark~\ref{rmk:2.4.18}. Note that
\begin{align}
\prod_{v|v'}\psi_v\left(\frac{-t_v}{2d_{\mscr F_v}}\right) = \psi_{v'}\left(\frac{-pt'_{v'}}{2d_{\mscr F_{v}}}\right)=\psi_{v'}\left(\frac{-t'_{v'}}{2d_{\mscr F'_{v'}}}\right)
\end{align}
holds. Also note that
\begin{align}
\vol(\mfrak l^{j_1}, dx_v) =\vol(\mfrak l'^{j_1}, dx')
\end{align}
holds, where~$dx_v$ and~$dx'$ are Haar measures on~$\mscr F_v$ and~$\mscr F'_{\mscr F'}$ such that
\begin{align}
\vol(\mscr O_v, dx_v) = \vol(\mscr O'_{v'}, dx')=1.
\end{align}
Of course,~$dx_v$ and~$dx$ are the same, and the subscript merely specifies the place~$v$. Also note that
\begin{align}
g\cdot \rec_v (x + 2^{-1} \bs \theta_v )^{-1} \psi_v \left( \frac{-\beta' x}{d_{{\mscr F_v}}}\right)
=
\rec_v (x^g + 2^{-1} \bs \theta_v )^{-1} \psi_v \left( \frac{-\beta' x}{d_{{\mscr F_v}}}\right)
\end{align}
for any~$g \in G$ and any~$x \in \mfrak l^{-j_0}/\mfrak l^{j_1}$. Using these properties, one can show the assertion of the lemma by an argument similar to the proof of Lemma~\ref{lemma:4.4.63}. One may take the index set~$W$ to be
\begin{align}
W = \{(x_v)_{v|v'} | x_v \in \mfrak l_v^{-j_0}/\mfrak l_v^{j_1}\}
\end{align}
where~$\mfrak l_v$ is the maximal ideal of~$\mscr O_v$, and the~$G$-action on~$W$ is given by
\begin{align}
\left((x_v)_{v|v'}\right)^g = (x_{v^{g^{-1}}})_{v|v'}
\end{align}
which is the inverse action on the indices similar to \eqref{eq:4.4.68}. We omit the details.
\end{proof}
\begin{lemma}
Suppose~$v'$ is inert in~$\mscr F$, and~$v'$ does not divide~$\mfrak D'$. Then, the assertion of \eqref{eq:4.4.29} is true.
\end{lemma}
\begin{proof}
In this case, there is a unique place~$v$ of~$\mscr F$ lying above~$v'$. Because~$v'$ is inert in~$\mscr F$, we have
\begin{align}
|\varpi_{v'}|^p_{v'}=|\varpi_v|_v 
\end{align}
and, in particular, we have
\begin{align}\label{4.4.90}
|\varpi_{v'}|^p_{v'}\equiv |\varpi_v|_v \hspace{5mm}\text{(mod~$p$)}.
\end{align}
Also, since~$v$ is unramified over~$v'$, we have
\begin{align}
\val_{v'}(\beta' \bs c(a')_{v'}) &=\val_{v}(\beta' \bs c(a')_{v}).
\end{align}
The assertion of the lemma is
\begin{align}
&\sum_{j =0}^{\val _v(\beta \bs c(a')_v)}
\rec_{\mscr M,v}(\varpi_v^j \bs c(a')_v)|\varpi_v|_v^{-j} \mbb I_{\mscr O_v} ( \beta \bs c(a')_v)
\\
\equiv
&\ver\left(
\sum_{j =0}^{\val _{v'}(\beta \bs c(a')_{v'})}
\rec_{\mscr M',{v'}}(\varpi_{v'}^j \bs c'(a')_{v'})|\varpi_{v'}|_{v'}^{-j} \mbb I_{\mscr O'_{v'}} ( \beta' \bs c'(a')_{v'})
\right)
\hspace{5mm}\text{(mod~$T$)}
\end{align}
which equivalent to
\begin{align}
&\sum_{j =0}^{n}
\rec_{\mscr M,v}(\varpi_v^j \bs c(a')_v)|\varpi_v|_v^{-j} \mbb I_{\mscr O_v} ( \beta \bs c(a')_v)
\\
\equiv
&
\sum_{j =0}^{n}
\rec_{\mscr M,v}(\varpi_v^j \bs c(a')_v)|\varpi_{v'}|_{v'}^{-j} \mbb I_{\mscr O_v} ( \beta \bs c(a')_v)
\hspace{5mm}\text{(mod~$T$)}
\end{align}
where we have written~$n=\val_{v'}(\beta' \bs c(a')_{v'}) =\val_{v}(\beta' \bs c(a')_{v})$.
We apply Proposition~\ref{prop:4.4.24} to
\begin{align}
&\sum_{j =0}^{n}
\rec_{\mscr M,v}(\varpi_v^j \bs c(a')_v)\left(|\varpi_v|_v^{-j}-|\varpi_{v'}|_{v'}^{-j}\right) \mbb I_{\mscr O_v} ( \beta \bs c(a')_v).
\end{align}
Since~$\bs c$ is the image of~$\bs c'$, it is fixed by~$G$, and hence
\begin{align}
g\cdot \rec_{\mscr M,v}(\varpi_v^j \bs c(a')_v) = \rec_{\mscr M,v}(\varpi_v^j \bs c(a')_v)
\end{align}
and the criterion of Proposition~\ref{prop:4.4.24} is satisfied by our previous observation \eqref{4.4.90}. The proof of the lemma is complete.
\end{proof}
\begin{lemma}
Suppose~$v'$ is inert in~$\mscr F$, and~$v'$ is splits in~$\mscr M'$. Then, the assertion of \eqref{eq:4.4.29} is true.
\end{lemma}
\begin{proof}
From the formula of Lemma~\ref{lemma:361}, it immediately follows from the fact that the equality
\begin{align}
\ver \left(\rec_{\mscr M', v'}(\beta' )\mbb I_{\mscr O'_{v'}} ( \beta').\right)=\rec_{\mscr M, v}(\beta' )\mbb I_{\mscr O_v} ( \beta')
\end{align}
holds. In particular, the congruence claimed in \eqref{eq:4.4.29} holds.
\end{proof}
\begin{lemma}
Suppose~$v'$ is inert in~$\mscr F$, and~$v'$ is either inert or ramified in~$\mscr M'$. Then, the assertion of \eqref{eq:4.4.29} is true.
\end{lemma}
\begin{proof}
Keeping the notation from Lemma~\ref{lemma:4.4.78}, we begin with the definition
\begin{align}\label{eq:4.4.790}
A(v)= 
\psi_v\left(\frac{-t_v}{2d_{\mscr F_v}}\right)\mathrm{Vol}(\mfrak l^{j_1}, dx)\sum_{x\in \mfrak l^{-j_0}/\mfrak l^{j_1}} \rec_v (x + 2^{-1} \bs \theta_v )^{-1} \psi_v \left( \frac{-\beta' x}{d_{{\mscr F_v}}}\right)
\end{align}
given in Lemma~\ref{lemma:3.5.6}, for any pair of sufficiently large integers~$j_0$ and~$j_1$. Increasing~$j_0$ and~$j_1$ if necessary, we have the same definition
\begin{align}
A'(v') = 
\psi_{v'}\left(\frac{-t'_{v'}}{2d_{\mscr F'_{v'}}}\right)\mathrm{Vol}({\mfrak l'}^{j_1}, dx)\sum_{x\in {\mfrak l'}^{-j_0}/{\mfrak l'}^{j_1}} \rec_{v'} (x + 2^{-1} \bs \theta_{v'} )^{-1} \psi_{v'} \left( \frac{-\beta' x}{d_{{\mscr F'_{v'}}}}\right)
\end{align}
with the same pair of integers~$j_0$ and~$j_1$ that we used for \eqref{eq:4.4.790}. Note that
\begin{align}
\psi_{v'}\left(\frac{-t'_{v'}}{2d_{\mscr F'_{v'}}}\right)=\psi_{v}\left(\frac{-t_{v}}{2d_{\mscr F_{v}}}\right)
\end{align}
since~$\psi_v = \psi_{v'} \circ \mathrm{Tr}_{\mscr F/ \mscr F'}$,~$d_{\mscr F_v} = pd_{\mscr F'_{v'}}$, and~$t'_{v'}=t_v$. Also note that~$\bs \theta_{v}$ is the image of the~$\bs \theta_{v'}$ under the diagonal embedding. Using these two facts, we obtain
\begin{align}\label{eq:4.4.840}
\ver\left(A'(v') \right)
&= 
\psi_v\left(\frac{-t_v}{2d_{\mscr F_v}}\right)\mathrm{Vol}({\mfrak l'}^{j_1}, dx')\sum_{x\in {\mfrak l'}^{-j_0}/{\mfrak l'}^{j_1}} \rec_{v} (x + 2^{-1} \bs \theta_{v'} )^{-1} \psi_v \left( \frac{-\beta' x}{d_{{\mscr F_v}}}\right)
\end{align}
where~$dx'$ denotes the Haar measure on~$\mscr O'_{v'}$ such that~$\vol(\mscr O'_{v'}, dx')=1$. Since~$v'$ is inert in~$\mscr F$, we have
\begin{align}
\vol(\mfrak l'^{j_1}, dx')^p = \vol(\mfrak l^{j_1}, dx)
\end{align}
and, in particular,
\begin{align}
\vol(\mfrak l'^{j_1}, dx') \equiv \vol(\mfrak l^{j_1}, dx) \hspace{5mm} \text{(mod~$p$)}
\end{align}
holds. From the equation \eqref{eq:4.4.840}, we obtain
\begin{align}\label{eq:4.4.870}
\ver\left(A'(v') \right)
&\equiv
\psi_v\left(\frac{-t_v}{2d_{\mscr F_v}}\right)\mathrm{Vol}({\mfrak l}^{j_1}, dx)\sum_{x\in {\mfrak l'}^{-j_0}/{\mfrak l'}^{j_1}} \rec_{v} (x + 2^{-1} \bs \theta_{v'} )^{-1} \psi_v \left( \frac{-\beta' x}{d_{{\mscr F_v}}}\right)
\end{align}
modulo~$T$. Now we observe that
\begin{align}\label{eq:4.4.880}
\left(\mfrak l^{-j_0}/\mfrak l^{j_1}\right)^G = {\mfrak l'}^{-j_0}/{\mfrak l'}^{j_1}.
\end{align}
Indeed, one can easily show \eqref{eq:4.4.880} by induction on~$j_0+j_1$. So \eqref{eq:4.4.790} becomes, letting~$W=\mfrak l'^{-j_0}/\mfrak l'^{j_1}$,
\begin{align}\label{4.4.890}
A(v) = &
\psi_v\left(\frac{-t_v}{2d_{\mscr F_v}}\right)\mathrm{Vol}(\mfrak l ^{j_1}, dx)\sum_{x\in W^G} \rec_v (x + 2^{-1} \bs \theta_v )^{-1} \psi_v \left( \frac{-\beta' x}{d_{{\mscr F_v}}}\right)
\notag
\\
&+
\psi_v\left(\frac{-t_v}{2d_{\mscr F_v}}\right)\mathrm{Vol}(\mfrak l^{j_1}, dx)\sum_{x\in W-W^G} \rec_v (x + 2^{-1} \bs \theta_v )^{-1} \psi_v \left( \frac{-\beta' x}{d_{{\mscr F_v}}}\right).
\end{align}
Applying Proposition~\ref{prop:4.4.24} to the difference of \eqref{eq:4.4.870} and \eqref{4.4.890}, we conclude the lemma.
\end{proof}
Combining the lemmas that we have proved, we conclude \eqref{eq:4.4.29}. The proof of Proposition~\ref{prop:4.4.30} is complete.
\end{proof}
We have completed the proof of Theorem~\ref{thm:main}.

\subsection{verifying \eqref{P'} and \eqref{C}}
In this subsection, we show that \eqref{P'} and \eqref{C} hold under mild assumptions. For a discussion of these assumptions, see Remark~\ref{rmk:4.5.24}.
\begin{proposition}\label{prop:4.5.1}
Suppose that~$\mu_{p^{r-1}}\subset \mscr M'$,~$\mscr M = \mscr M'(\mu_{p^r})$ for some~$r\ge 2$. Further assume that~$\mscr M' / \Q (\mu_{p^{r-1}})$ is unramified at the unique prime of~$\Q(\mu_{p^{r-1}})$ lying above~$p$. Then \eqref{P'} holds for~$\mscr M/ \mscr M'$. In particular, \eqref{P} holds when~$\mscr M'$ is given by~$\mscr M_0(\mu_{p^{r-1}})$ for a~$p$-ordinary CM field~$\mscr M_0$ which is unramified at~$p$.
\end{proposition}
\begin{proof}
Let~$\varphi$ be the Euler totient function. From conductor-discriminant formula, the discriminant of~$\Q(\mu_{p^r})$ is given by
\begin{align}
e_r :=\sum_{j=1}^r (\varphi(p^j)-\varphi(p^{j-1}))p^j,
\end{align}
because, for each~$j\ge 1$, there are~$\varphi(p^j)-\varphi(p^{j-1})$ distinct Dirichlet characters of conductor~$p^j$. Solving a recurrence relation, for every~$r\ge 2$, we obtain the formula~$\mfrak d_{\Q(\mu_{p^r})/\Q(\mu_{p^{r-1}})} = p\mscr O_{\Q(\mu_{p^r})}$ for the relative different.
\par
Consider the diagram 
\begin{align}
\xymatrix{
 				&\mscr M 		&
\\
\Q(\mu_{p^r})	\ar@{-}[ur]	&\mscr M'\ar@{-}[u]		&\mscr F \ar@{-}[ul]
\\
\Q(\mu_{p^{r-1}})\ar@{-}[ur]\ar@{-}[u]	&			&\mscr F'\ar@{-}[ul]\ar@{-}[u]
}
\end{align}
of field extensions. We first observe that~$\mscr M/ \Q(\mu_{p^r})$ is unramified by our assumption. On the other hand, the extensions~$\mscr M/\mscr F$ and~$\mscr M'/\mscr F'$ are always unramified. We give an argument for~$\mscr M/\mscr F$, since the same argument works for~$\mscr M'/\mscr F'$. Since~$\mscr M$ is~$p$-ordinary, the extension~$\mscr M/\mscr F$ is unramified at every place lying above~$p$ in particular. However,~$\mscr F(\mu_{p^r})$ is a non-trivial subextension of~$\mscr M/\mscr F$ since~$\mscr F$ is totally positive. Since~$\mscr M/\mscr F$ is quadratic,~$\mscr M$ is equal to~$\mscr F(\mu_{p^r})$. It follows from the properties of the cyclotomic fields that the extension~$\mscr M/\mscr F$ is unramified outside of~$p$.
\par
By the transifive property of relative different ideals, we have
\begin{align}
\mfrak d_{\mscr M/ \mscr M'} = \mfrak d_{\Q(\mu_{p^r})/\Q(\mu_{p^{r-1}})} \cdot \mscr R= p\mscr R
\end{align}
and
\begin{align}
\mfrak d_{\mscr M/ \mscr M'} = \mfrak d_{\mscr F/ \mscr F'} \cdot \mscr R.
\end{align}
From these to equalities, it follows that~$\mfrak d_{\mscr F / \mscr F'} = p\mscr O$, which is \eqref{P'}.
\end{proof}
\par
We introduce some notation that are necessary to analyze the condition \eqref{C}. Let us decompose
\begin{align}
Cl_- = A\oplus B,\hspace{10mm}Cl_-'= A' \oplus B'
\end{align}
where~$A$ and~$A'$ are the maximal~$p$-primary subgroups of~$Cl_-$ and~$Cl_-'$ respectively. Then the transfer map induces an isomorphism
\begin{align}\label{eq:4.5.5}
J_p\colon B' \tilde \rightarrow B^G.
\end{align}
However, the map~$J_p\colon A' \to A^G$ is not necessarily an isomorphism. The cokernel of~$A' \to A^G$ is described by the ramification of~$\mscr F/\mscr F'$.
\begin{proposition}\label{prop:4.5.6}
Suppose that~$\mscr M = \mscr M'(\mu_{p^{r+1}})$ for some positive integer~$r$. Then we have
\begin{align}
\mathrm{Coker}(A' \to A^G) \cong \prod_{v'}\Z/p\Z
\end{align}
where the product is taken over the set of places of~$\mscr F'$ lying over~$p$ that are ramified in~$\mscr F'$. Moreover, each factor of~$\Z/p\Z$ is generated by the ideal class
\begin{align}
\mfrak p \overline{\mfrak p}^{-1}
\end{align}
where~$\mfrak p$ is the prime ideal associated to the unique element of~$\Sigma_p$ whose restriction to~$\mscr F'$ is~$v'$.
\end{proposition}
\begin{proof}
See the Corollary~1.3.6 of \cite{Greenberg book}.
\end{proof}
\par
\begin{proposition}\label{prop:4.5.9}
Suppose the condition of Proposition~\ref{prop:4.5.6} holds. Further assume that~$\mscr F/\mscr F'$ is ramified at every place of~$\mscr F'$ lying above~$p$. Then the condition \eqref{C} holds.
\end{proposition}
\begin{proof}
As before, we fix representatives~$\mcal D$ and~$\mcal D'$ for~$Cl_-$ and~$Cl_-'$ respectively. In particular, we choose~$\mcal D$ so that the diagonal embedding maps~$\mcal D'$ into~$\mcal D$, and we choose~$\mcal D$ to be a~$G$-set with respect to the natural action of~$G$ on ideles. We further fix the decompositions
\begin{align}
\mcal D &= \mcal A \times \mcal B
\\
\mcal D' &= \mcal A' \times \mcal B'
\end{align}
where~$\mcal A$,~$\mcal B$,~$\mcal A'$, and~$\mcal B'$ are representatives for~$A,B,A'$, and~$B'$ respectively. We may also assume that~$\mcal A$ and~$\mcal B$ are~$G$-sets so that the decomposition~$\mcal D = \mcal A \times \mcal B$ becomes a decomposition of a~$G$-set, and that the natural diagonal embedding maps~$\mcal A'$ and~$\mcal B'$ into~$\mcal A$ and~$\mcal B$ respectively.
\par
The map \eqref{eq:4.5.5} being an isomorphism is now equivalent to the map
\begin{align}
\mcal B' \to \mcal B^G
\end{align}
being a bijection. On the other hand, the conclusion of Proposition~\ref{prop:4.5.6} in terms of the~$\mcal A$ and~$\mcal A'$ can be reformulated as the assertion that we can choose~$\mcal A$ so that we have
\begin{align}\label{eq:4.5.13}
\mcal A^G = \coprod_{j}\left( \mcal A' \cdot \prod_w\left( \frac{\varpi_w}{\varpi_{\overline w}}\right) ^{j_w}\right)
\end{align}
where~$j$ runs over the set
\begin{align}
\{(j_w)_{w\in\Sigma_p}| j_w = 0,1,2,\cdots,p-1\}
\end{align}
of~$\Sigma_p$-tuples of integers. The bijection \eqref{eq:4.5.13} given above implies that
\begin{align}
\sum_{a\in \mcal A^G} \rec_\mscr M (a) = \left(\prod_{w\in\Sigma_p}\left(\sum_{j_w=0}^{p-1} \frac{\rec_w(\varpi_w)}{\rec_{\overline w}(\varpi_{\overline w})}\right)\right)
\sum_{a' \in \mcal A'}\ver\left(\rec_{\mscr M'}(a')\right)
\end{align}
holds. Multiplying~$\mcal C$ to both sides, we obtain
\begin{align}\label{eq:4.5.16}
\mcal C \cdot \sum_{a\in \mcal A^G} \rec_\mscr M (a) = \ver \left( \mcal C' \cdot \sum_{a'\in\mcal A'} \rec_{\mscr M'}(a') \right).
\end{align}
It is not very difficult to obtain \eqref{C} from this. Indeed, 
\begin{align}
\mcal C \cdot \sum_{a \in \mcal D^G} \rec_\mscr M (a)
&=
\left( \mcal C \cdot \sum_{a \in \mcal A^G}\rec_\mscr M(a)\right) \times \left(\sum_{b\in \mcal B^G} \rec_\mscr M(b)\right)
\\
&=
\ver \left( \mcal C' \cdot \sum_{a'\in\mcal A'} \rec_{\mscr M'}(a') \right)\times\ver\left(\sum_{b'\in \mcal B'} \rec_{\mscr M'}(b')\right)
\\
&=
\ver \left( \mcal C' \cdot \sum_{(a',b')\in\mcal A'\times \mcal B'} \rec_{\mscr M'}(a'b')\right)
\\
&=
\ver \left( \mcal C' \cdot \sum_{a'\in\mcal D'} \rec_{\mscr M'}(a')\right)
\end{align}
after routine calculations. The proof of the proposition is complete.
\end{proof}
\begin{corollary}
Let~$\mscr K$ be an imaginary quadratic field and let~$p$ be an odd prime that splits in~$\mscr K$. For an integer~$r\ge 2$, put~$\mscr M = \mscr K (\mu_{p^{r}})$ and~$\mscr M' = \mscr K(\mu_{p^{r-1}})$. Then \eqref{C} and \eqref{P'} holds.
\end{corollary}
\begin{proof}
It is clear that the extension~$\mscr M/ \mscr M'$ satisfies the assumptions of Proposition~\ref{prop:4.5.1}, Proposition~\ref{prop:4.5.6}, and Proposition~\ref{prop:4.5.9}.
\end{proof}
\begin{remark}\label{rmk:4.5.24}
We explain why the condition~\eqref{P'} and \eqref{C} are considered as mild. Suppose we have a tower of CM fields consisting of~$\mscr M_r$ for each positive integer~$r$, such that~$\mscr M_r \subset \mscr M_{r+1}$ is a degree~$p$ extension for all~$r$. Let~$\mscr M_\infty~$ be the union of the all~$\mscr M_r$, and suppose that~$\mscr M_\infty$ is Galois over~$\mscr M_1$ with the Galois group~$\Gal(\mscr M_\infty / \mscr M )$ being isomorphic to~$\Z_p$. Let~$\mscr F_r$ be the maximal totally real subfield of~$\mscr M_r$. If we let~$\mscr F_\infty$ be the union of all~$\mscr F_r$, then~$\mscr F_\infty$ is exactly the maximal totally real subfield of~$\mscr M_\infty$, and it is Galois over~$\mscr F_1$ with the Galois group~$\Gal(\mscr F_\infty / \mscr F )$ being isomorphic to~$\Z_p$. Under the validity of Leopoldt's conjecture for~$\mscr F_1$, which we do not know in full generality yet,~$\mscr F_\infty$ must be the cyclotomic~$\Z_p$-extension of~$\mscr F_1$. Then it is easy to see that the condition~\eqref{P} and \eqref{C} are satisfied for the extension~$\mscr M_{r+1}/\mscr M_r$ when~$r$ is sufficiently large. Note that the Leopoldt's conjecture is proved by Brumer for abelian totally real number fields.
\end{remark}
\section{Applications to the false Tate curve extensions}\label{s5}
The current section has two aims. The first one is to explain the relation of our transfer congruence to the existence of non-commutative~$p$-adic~$L$-functions, and the second is to study the explicit congruence between~$L$-values whic is implied by our work.
\par
Roughly speaking, in order to prove the non-commutative Iwasawa main conjecture for a given motive and a~$p$-adic Lie extension, it is sufficient to prove various commutative Iwasawa main conjecture for intermediate extensions and prove certain non-commutative Kummer congruences between these commutative~$p$-adic~$L$-functions. This strategy was successfully applied for the Tate motives over totally real fields in the works of Kakde and Ritter-Weiss. We want to explain what are the relevant intermediate extensions for the false Tate curve extensions, and how to deduce the necessary non-commutative Kummer congruence between them from our transfer congruence. Unfortunately, we do not know at present all of the commutative main conjectures that are necessary for the non-commutative one, in contrast to the situation of the Tate motives over totally real fields for which the commutative main conjecture is known in a wide generality.
\par
While our speculations on the non-commutative Iwasawa main conjecture will rely on other unknown conjectures, our work unconditionally imply the congruence between special values of~$L$-functions. We will study the explicit form of such congruences over the false Tate curve extension in Subsection~\ref{subsection:5.3}.
\subsection{Algebraic preliminaries}
In this subsection, we start with a compact~$p$-adic Lie group~$G$. We always assume that~$G$ has a closed normal subgroup~$H$ such that~$\Gamma =  G/ H$ is isomorphic to~$\Z_p$, and explain in general and vague terms what are the non-commutative congruences for it. In the case when $H$ is abelian, which is our main concern of this paper, will state precisely the description of $K_1$-groups of Iwsawa algebra $\Lambda(G)$ in terms of the transfer congruence. For more details and generalizations, see \cite{Muenster 11}. Following \cite{CFKSV}, we define
\begin{align}
S = \{ \lambda \in \Lambda ( G) | \Lambda( G) / \Lambda(G)\lambda \text{ is finitely generated as a left~$\Lambda( H)$-module.}\}.
\end{align}
As shown in Theorem~2.4 of \cite{CFKSV},~$S\subset \Lambda(G)$ is an Ore set, and we denote the localization of~$\Lambda(G)$ at~$S$ by~$\Lambda(G)_S$. The completion of~$\Lambda(G)_S$ with respect to the~$p$-adic topology is denoted by~$\widehat{\Lambda(G)_S}$. Strictly speaking,~$S$ depends on~$G$ and~$H$, but we omit~$G$ and~$H$ from the notation. In practice,~$G$ is given as a Galois group, and~$\Gamma$ corresponds to the cyclotomic~$\Z_p$-extension.
\par
For any topological ring~$R$, we define~$K_1(R)$ by
\begin{align}
K_1(R) = \GL(R)/[\GL(R),\GL(R)]^-
\end{align}
where~$\GL(R)$ is the inductive limit~$\lim_{\rightarrow n}\GL_n(R)$ with respect to the inclusions~$\begin{bmatrix}
a
\end{bmatrix}
\mapsto
\begin{bmatrix}
a&0\\
0&1
\end{bmatrix}$, and~$[\GL(R),\GL(R)]^-$ is the closure of the commutator subgroup~$[\GL(R),\GL(R)]$ of~$\GL(R)$. 
\par
When~$G$ is abelian,~$\Lambda(G)$ is a semi-local ring and~$K_1(\Lambda(G))=\Lambda(G)^\times$ and~$K_1(\Lambda(G)_S)=\Lambda(G)_S^\times$. One successful approach to understanding~$K_1(\Lambda(G))$ for a non-commutative~$G$ is to describe it in terms of abelian subquotients of~$G$. To be precise, one defines a family of abelian subquotients~$\{G_i| i\in I\}$ for some index set~$I$ and a family of maps~$\theta^i\colon K_1(\Lambda(G))\to K_1(\Lambda(G_i))$, in the hope that studying the map
\begin{align}
\theta =\prod_{i\in I} \theta^i \colon K_1(\Lambda(G))\to \prod_{i \in I} K_1(\Lambda(G))=\prod_{i \in I} \Lambda(G)^\times
\end{align}
shed some light on the structure of~$K_1(\Lambda(G))$.
One tries to show~$\theta$ is injective, or at least describe its kernel, and also describe the image. The relations that an~$I$-tuple~$(\mcal L_i )_i$ needs to satisfy in order to be in the image of~$\theta$ is called non-commutative congruences. The transfer congruence proved in Theorem~\ref{thm:main} is one example of such non-commutative congruences. When~$H$ is pro-$p$ and abelian, then this is in fact sufficient. Indeed, for every integer~$r \ge 0$, consider the unique closed subgroup~$\Gamma_r$ of~$\Gamma$ with~$[\Gamma_r:\Gamma]=p^r$. Let~$G_r$ be the inverse image of~$\Gamma_r$ in~$G$ and let~$G_r^{\mathrm{ab}}$ be the abelianization of~$G_r$. Since~$G_r$ is of finite index in~$G$, we have a norm map~$\mathrm{Nr}_r=K_1(\Lambda( G)) \to K_1(\Lambda ( G_r)$. Similarly, we have the norm map~$\mathrm{Nr}_{r',r}=K_1(\Lambda( G_j)) \to K_1(\Lambda ( G_r)$ whenever~$r'\le r$. Also we have a map~$K_1( \Lambda(G_{r'})) \to K_1(\Lambda(  G_r^{\mathrm{ab}})$ induced by the natural projection~$G_r \to G_r^{\mathrm{ab}}$. Define~$\theta^r$
\begin{align}
\theta^r \colon K_1(\Lambda( G)) \to K_1(\Lambda ( G_r)) \to K_1(\Lambda(  G_r^{\mathrm{ab}}))
\end{align}
to be the composition of the two maps. Similarly, we define the maps
\begin{align}
\theta^r_S \colon K_1(\Lambda( G)_S) \to K_1(\Lambda ( G_r)_S) \to K_1(\Lambda(  G_r^{\mathrm{ab}})_S)
\\
\widehat\theta^r_S \colon K_1(\widehat{\Lambda( G)_S}) \to K_1(\widehat{\Lambda ( G_r)_S}) \to K_1(\widehat{\Lambda(  G_r^{\mathrm{ab}})_S}).
\label{eq:5.1.6}
\end{align}
For each positive integer~$r$, let~$\ver_r \colon \Lambda(G_{r-1}^\mathrm{ab}) \to \Lambda(G_r^\mathrm{ab})$ be the map induced by the transfer map~$G_{r-1}^\mathrm{ab}\to G_r^\mathrm{ab}$ on the group elements. We have a trace map~$\Lambda(G_{r}^\mathrm{ab}) \to \Lambda(G_{r}^\mathrm{ab})$ with respect to the action of~$G_r/G_{r-1}$, and let~$T_{r}$ be the image. Then~$T_{r}$ is an ideal in~$\Lambda(G_r^\mathrm{ab})$ and we write~$\widehat T_{r,S} = T_{r} \widehat{\Lambda(G_{r}^\mathrm{ab} )_S}$. 
\begin{theorem}\label{thm:5.1.7}
Assume that~$H$ is pro-$p$ and abelian. The maps
\begin{align}
\theta \colon K_1(\Lambda(G)) \to \prod_{r\ge 0} \Lambda(G_r^{\mathrm{ab}})^\times
\\
\widehat\theta_S \colon K_1(\widehat{\Lambda(G)_S}) \to \prod_{r\ge 0} \widehat{\Lambda(G_r^{\mathrm{ab}})_S}^\times
\end{align}
are injective. The image of~$\theta$ consists of~$(x_r)$ satisfying
\begin{enumerate}
\labitem {MS1}{MS1}~$\mathrm{Nr}_{j,i}(x_j)=x_i$ for all~$0 \le i \le j~$.
\labitem {MS2}{MS2}~$x_i$ is fixed under the action of~$\Gamma$ for all~$i$.
\labitem {MS3}{MS3}~$x_i \equiv \ver_i(x_{i-1})$ modulo~$T_i$ for all~$i\ge 1$.
\end{enumerate}
The image of~$\widehat\theta_S$ is described by the first two conditions and the congruence modulo~$\widehat T_{i,S}$.
\end{theorem}
\begin{remark}
The conditions \eqref{MS1} and \eqref{MS2} are minimal compatibility conditions. For a general~$G$, the transfer congruence \eqref{MS3} is not sufficient, and we need to consider different kinds of congruences as described in \cite{Kakde 11, Ritter-Weiss}.
\end{remark}
\begin{remark}
It is sometimes necessary to use an extension of~$\Z_p$ as the coefficient for the Iwasawa algebra. When the coefficient ring is unramified, it is proved in \cite{Kakde 11} that the same result holds after twisting~$\ver_i$ by the Frobenius on the coefficient ring.
\end{remark}
\subsection{Specialization to the false Tate curve extensions}
In this subsection, we will specialize to the false Tate curve extension of imaginary quadratic fields, and discuss the non-commutative $p$-adic $L$-function along it. More precisely, we fix an imaginary quadratic field
\begin{align}
\mscr M_0 = \Q(\sqrt {-D_0})
\end{align}
for a positive square-free integer~$D_0$, and an odd prime~$p$ which splits in~$\mscr M_0$. Fix an integer~$m$ which is a~$p$-th power free and~$|m|_\R>1$. Define
\begin{align}
\mscr M_r = \mscr M_0(\mu_{p^r})
\end{align}
for each non-negative integer~$r$, and define
\begin{align}
\mscr K_\infty = \bigcup_{r=0}^\infty \mscr M_r(m^{\frac{1}{p^r}})
\end{align}
to be the two dimensional~$p$-adic Lie extension of~$\mscr M_0$. It easily follows from the Kummer theory that the Galois group~$\Gal (\mscr K_\infty / \mscr M_0)$ is isomorphic to the semidirect product~$\Z_p \rtimes \Z_p^\times$, where~$\Z_p^\times$ acts on~$\Z_p$ by the multiplication.
\par
We need a preliminary reduction step before we apply the result of the previous subsection. Consider the open subgroup~$\Z_p \rtimes(1+p\Z_p)$ of index~$p-1$ in~$\Z_p \rtimes \Z_p^\times$. Consider the map
\begin{align}
\phi \colon K_1(\Lambda(\Z_p \rtimes \Z_p^\times)) &\longrightarrow K_1(\Lambda(\Z_p^\times)) \times K_1(\Lambda(\Z_p \rtimes (1 + p \Z_p)))
\\
\lambda &\longmapsto (\lambda_0, \lambda_1)
\end{align}
where~$\lambda_0$ is the image under map induces by the natural projection~$\Z_p\rtimes\Z_p^\times \to \Z_p^\times$, and~$\lambda_1$ is the image of the norm map. Since~$1 + p\Z_p$ is an open subgroup of~$\Z_p^\times$, we have a norm map 
\begin{align}
\phi_0 \colon K_1(\Lambda(\Z_p^\times)) \longrightarrow K_1(\Lambda( 1 + p \Z_p))
\end{align}
and we have a map
\begin{align}
\phi_1 \colon K_1(\Lambda(\Z_p \rtimes (1 + p\Z_p))) \longrightarrow K_1(\Lambda( 1 + p \Z_p))
\end{align}
induce by the projection map. 
\begin{proposition}\label{prop:5.2.8}
A pair 
\begin{align}
(\lambda_0,\lambda_1) \in K_1(\Lambda(\Z_p^\times)) \times K_1(\Lambda(\Z_p \rtimes (1 + p \Z_p)))
\end{align}
lies in the image of~$\phi$ if and only if~$\phi_0(\lambda_0) = \phi_1(\lambda_1)$.
\end{proposition}
\begin{proof}
We omit the proof.
\end{proof}
\par
The above proposition is the reason why the non-commutative congruences for the group $\Z_p \rtimes \Z_p^\times$ only involve the subquotions of $\Z_p \rtimes (1 + p\Z_p)$. Our strategy to prove existence of a $p$-adic $L$-function for $\Z_p \rtimes \Z_p^\times$ is to construct $\lambda_0$ and $\lambda_1$ satisfying the above compatibility condition. In order to construct $\lambda_1$, we apply Theorem~\ref{thm:5.1.7} to the group~$G = \Z_p \rtimes (1 + p\Z_p)$. The open subgroups~$G_r$ correspond to the groups~$\Gal(\mscr K_\infty/\mscr M_r)$, whose abelianizations are given by
\begin{align}
G_r^{\mathrm{ab}} = \Gal(\mscr M_r(\mu_{p^\infty}, m^{\frac{1}{p^r}}) / \mscr M_r).
\end{align}
\par
Fix a CM type~$\Sigma_0$ for~$\mscr M_0$, which is~$p$-ordinary, then for each~$r\ge 1$,~$\mscr M_r$ has a~$p$-ordinary CM type~$\Sigma_r$ induced from~$\mscr M_0$. For simplicity, let~$\mscr M = \mscr M_1$, and let~$\Sigma$ be the CM type of~$\mscr M$. suppose we have an algebraic Hecke character
\begin{align}
\chi \colon \adele_{\mscr M} \longrightarrow \C^\times
\end{align}
whose infinity type is~$k\Sigma$ with a positive integer~$k$, and unramified at all places dividing~$p$.
\par
\begin{remark}\label{rmk:5.2.12}
In practice, such~$\chi$ can be obtained from the Hecke character~$\chi_E$ associated to an elliptic curve~$E$ defined over~$\mscr M_0$ with complex multiplication by~$\mscr M_0$, with good ordinary reduction at all places dividing~$p$. Such character will have infinity type~$\Sigma_0$, and the character
\begin{align}\label{eq:5.2.12}
\chi_E^k \circ \mathrm{Norm} \colon \adele_\mscr M \to \C^\times
\end{align}
will have infinity type~$k\Sigma$. We will come back to this situation in the next subsection.
\end{remark}
\par
By composing~$\chi$ with the norm map~$\adele_{\mscr M_r}^\times \to \adele_{\mscr M}^\times$, we obtain
\begin{align}
\chi_r \colon \adele_{\mscr M_r}^\times \longrightarrow \C^\times
\end{align}
for each~$r\ge 1$.
\par
Let~$\mfrak C$ be the conductor of~$\chi$, which is an ideal of~$\mscr M$ relatively prime to~$p$. Fix a polarization~$\delta \in \mscr M$, with polarization ideal~$\mfrak c$. Now we are in the situation of the main results of our paper. In particular, for each~$r\ge 1$, we obtain the~$p$-adic~$L$-function
\begin{align}
\mcal L_r \in \Lambda(Z_r)
\end{align}
of Definition~\ref{def:3.8.1}, where~$Z_r$ is the Galois group
\begin{align}
\lim_{\infty \leftarrow n } \Gal(\mscr M_r(p^n\mfrak C) / \mscr M_r)
\end{align}
introduced in \eqref{eq:2.3.7}. Let~$\mcal L_r(\chi_r)$ be the branch of~$\mcal L_r$ along~$G_r$ with respect to~$\chi_r$. That is to say,~$\mcal L_r(\chi_r)$ is the unique element of~$\Lambda(G_r^{\mathrm{ab}})$ such that
\begin{align}\label{eq:5.2.17}
\int_{G_r^\mathrm{ab}} \upphi d\mcal L_r(\chi) = \int_{Z_r} \widetilde \upphi \widehat\chi d\mcal L
\end{align}
for every continuous function~$\upphi \colon G_r^\mathrm{ab} \to \Z_p$, where~$\widetilde \upphi$ denotes the pull-back of~$\upphi$ along the natural projection map~$Z_r \to G_r^\mathrm{ab}$.
\begin{remark}\label{rmk:5.2.18}
When~$\chi$ arises from an imaginary quadratic field in the manner explained in Remark~\ref{rmk:5.2.12}, then we let~$\mcal L_0(\chi_0)$ to be the branch of~$\mcal L_0$ along~$G^\mathrm{ab} = \Gal(\mscr M_\infty/ \mscr M_0)$ with respect to~$\chi_0$. If we put~$\lambda_0 = \mcal L_0(\chi_0)$ and~$\lambda_1 = \mcal L_1(\chi_1)$, then~$\lambda_0$ and~$\lambda_1$ satisfies the compatibility condition in Proposition~\ref{prop:5.2.8}, in the sense that if we let
\begin{align}
\phi_0 \colon \Lambda(\Z_p^\times) \longrightarrow \Lambda(1 + p \Z_p)
\end{align}
be the norm map and let
\begin{align}
\phi_1 \colon \Lambda(\Z_p / p\Z_p \times (1 + p \Z_p)) \to \Lambda(1 + p \Z_p)
\end{align}
be the natural projection, then we have~$\phi_0(u\cdot \lambda_0) = \phi_1(\lambda_1)$ for a $p$-adic unit $u\in \Z_p^\times$.
\end{remark}
\par
As a consequence of Theorem~\ref{thm:main}, we establish the transfer congruence between~$\mcal L_r(\chi_r)$ for~$r\ge1$. The factor $u$ was computed in the proof of Theorem~7.2 in \cite{CMFT}, where we worked with $p\mcal L_1$ instead of $\mcal L_1$.
\begin{theorem}\label{thm:5.2.16}
We have
\begin{align}
\ver_{r-1}(\mcal L_{r-1}(\chi)) = \mcal L_{r} (\chi)  \hspace{5mm}(\textrm{mod }T_r)
\end{align}
for each~$r\ge 2$.
\begin{proof}
The key fact is the commutativity of the diagram
\begin{align}
\xymatrix{
Z_{r-1}\ar[r]\ar[d] & G_{r-1}^\mathrm{ab}\ar[d]
\\
Z_{r}\ar[r] & G_{r}^\mathrm{ab}
}
\end{align}
where the two vertical maps are transfer maps, and the horizontal maps are the natural projections. Having observed this, it is easy to see the congruence from the characterization \eqref{eq:5.2.17} of~$\mcal L_r(\chi_r)$ in terms of~$\mcal L_r$, and the transfer congruence between~$\mcal L_r$. 
\end{proof}
\end{theorem}
\par
Although Theorem~\ref{thm:5.2.16} verifies \eqref{MS3} of Theorem~\ref{thm:5.1.7}, we do not know yet whether
\begin{align}\label{eq:5.2.22}
\mcal L_r(\chi_r) \in \Lambda(G_r^\mathrm{ab})_S^\times
\end{align}
holds for all~$r\ge 1$. We first observe that the conditions \eqref{eq:5.2.22} for~$r\ge 0$ are rather strongly dependent to each other.
\begin{proposition}
Suppose that \eqref{eq:5.2.22} holds for~$r=0$. Then it holds for all~$r\ge 0$.
\end{proposition}
\begin{proof}
We proceed by induction. Since~$G_r^\mathrm{ab}$ is one dimensional for every~$r\ge 0$, an element~$\lambda \in \Lambda(G_r^\mathrm{ab})$ belongs to~$\Lambda(G_r^\mathrm{ab})_S^\times$ if and only if~$\lambda$ is not a zero divisor in~$\Lambda(G_r^\mathrm{ab})/p\Lambda(G_r^\mathrm{ab})$. 
\par
From~$r=0$ to~$r=1$, we use the observation made in Remark~\ref{rmk:5.2.18}, that~$\mcal L_0(\chi_0)$ and~$\mcal L_1(\chi_1)$ satisfies a compatibility condition with respect to~$\phi_0$ and~$\phi_1$. Since~$\phi_0$ maps a non-zero divisor to a non-zero divisor, and~$\phi_1$ maps a zero devisor to a zero divisor, we have the induction step from~$r=0$ to~$r=1$. The remaining induction steps can be verified similarly by comparing~$\mcal L_r(\chi_r)$ and~$\mcal L_{r-1}(\chi_{r-1})$ using relevant norm and projection maps.
\end{proof}
\par
Unfortunately, we do not know how to theoretically prove \eqref{eq:5.2.22} for~$r=0$, and we assume it only for the rest of the current subsection. If we further assume that~$\mcal L_r(\chi_r)$ satisfies the commutative Iwasawa main conjecture for every $r$, then it will follow that there exists a non-commutative~$p$-adic~$L$-function in~$K_1(\Lambda(G)_S)$, which will satisfy the non-commutative Iwasawa main conjecture. Without the assumption on the commutative main conjectures, we need to work with the completion of localized Iwasawa algebras.
\begin{corollary}
Suppose \eqref{eq:5.2.22}. Then there exists an element
\begin{align}
\mcal L(G) \in K_1(\widehat{\Lambda(G)_S})
\end{align}
which satisfies 
\begin{align}
\widehat \theta_S^r (\mcal L(G)) = \mcal L_r(\chi_r)
\end{align}
for every $r \ge 0$, where $\widehat \theta_S^r$ was defined in \eqref{eq:5.1.6}.
\end{corollary}
\begin{proof}
The existence of certain element~$\mcal L(G)$ in~$ K_1(\widehat{\Lambda(G)_S})$ follows from Theorem~\ref{thm:5.1.7}, combined with Proposition~\ref{prop:5.2.8}. We need to verify then the conditions of Theorem~\ref{thm:5.1.7} and Proposition~\ref{prop:5.2.8} are satisfied the family $\widehat \theta_S^r (\mcal L(G))$. The condition \eqref{MS3} is satisfied by Theorem~\ref{thm:main}, and the other conditions, including one for Proposition~\ref{prop:5.2.8}, follow from the interpolation formula of $\mcal L_r$.
\end{proof}
\par
The discussion on the interpolation property of $\mcal L(G)$ is postponed to Subsection~\ref{subsection:5.4}, since we first need to classify the Artin representations of $G$, which we will treat in the next subsection.
\subsection{Congruence between special values of~$L$-functions}\label{subsection:5.3}
In this subsection, we want to discuss the concrete congruences between special values of~$L$-functions, which can be deduced from the work of current paper. We keep the notation from the previous subsection. In particular, we will focus on the characters
\begin{align}
\chi_r \colon \adele_{\mscr M_r}^\times \longrightarrow \C^\times
\end{align}
that are constructed from a character~$\chi_0$ of an imaginary quadratic field, as described in Remark~\ref{rmk:5.2.12}.
\par
We will consider two families of congruences. A special case of the first family of congruences was studied in \cite{CMFT}, but the proof was conditional. Here we have unconditional result which holds in general. The second family arises from main result in Theorem~\ref{thm:main}, which illustrates the non-commutative nature of the transfer congruence. Indeed, they are congruence between special values of~$L$-functions of \emph{different} degrees, while the family of~$L$-functions considered in the commutative Iwasawa have the same degree.
\par
Let us first describe the irreducible representations of~$G$ with finite image on a~$\overline \Q_p$-vector space. The classification we give here can be verified easily from the character theory of finite groups, but we simply state the results. For each~$r$, there is an irreducible representation
\begin{align}\label{eq:5.3.2}
\rho_r \colon G \to \GL_{\varphi(p^r)}(\Q_p)
\end{align}
of dimension~$\varphi(p^r)$, where~$\varphi$ is the Euler totient function. It is the representation induced from a character~$\eta_r$ of~$G_r$ which we define now.
\begin{definition}\label{def:5.3.3}
Define a character~$\eta_r$ of~$G_r$ to be a lift of any character
\begin{align}
\Gal(\mscr M_r(m^{\frac{1}{p^r}})/ \mscr M_r) \to \overline\Q_p^\times
\end{align}
whose order is exactly~$p^r$.
\end{definition}
The representation~$\rho_r$ is independent of the choice of~$\eta_r$, and it can be in fact defined over~$\Z$. All the other irreducible representations of~$G$ are constructed by twisting~$\rho_r$. Precisely speaking, for an irreducible representation~$\rho$ of~$G$ on a finite dimensional~$\overline \Q_p$-vector space, there exists a one dimensional character~$\xi$ of~$G$ such that~$\rho$ is of the form~$\rho_r \otimes \xi$ for some~$r$. On the other hand, if we have two characters~$\xi_1$ and~$\xi_2$ of~$G$, then~$\rho_r \otimes \xi_1$ is isomorphic to~$\rho_r\otimes\xi_2$ if and only if the restrictions of~$\xi_1$ and~$\xi_2$ to~$G_r$ are the same.
\par
Apart from the irreducible ones, we will consider
\begin{align}
\sigma_r \colon G \to \GL_{\varphi(p^r)}(\Q_p)
\end{align}
which is the representation induced from the trivial character of~$G_r$. Of course,~$\sigma_r$ can be defined over~$\Z$, and it decomposes into the sum of~$\varphi(p^r)$ one dimensional representations over~$\overline \Q_p$. 
\par
Both~$\rho_r$ and~$\sigma_r$ are defined over~$\Z$, and it makes sense to ask if they are congruent modulo $p$. Indeed, one can easily see that they are congruent modulo~$p$, so we expect a congruence between canonically normalized special values of the corresponding twists of an~$L$-function. In other words, if we let
\begin{align}
L(\chi_0,\rho,s)
\end{align}
be the complex~$L$-function of~$\chi_0$ twisted by~$\rho$, then the special values of
\begin{align}
L(\chi_0,\sigma_r,s),\,\,\textrm{and}\,\,L(\chi_0,\rho_r,s)
\end{align}
at critical points are expected to be congruent modulo~$p$, after normalizing them in a canonical way. From Artin's formalism, we have
\begin{align}
L(\chi_0,\rho_r) = L(\chi_r,\eta_r,s)
\end{align}
and
\begin{align}
L(\chi_0,\sigma_r) = L(\chi_r,s)
\end{align}
so we use our~$p$-adic~$L$-functions~$\mcal L_r(\chi_r)$ to deduce the desired congruences.
\begin{definition}
Define
\begin{align}
\mcal L(\chi_0,\rho_r,n) = \int_{Z_r}\eta_r \widehat\kappa^n d\mcal L_r(\chi_0)
\end{align}
and
\begin{align}
\mcal L(\chi_0,\sigma_r,n) = \int_{Z_r} \widehat\kappa^n d\mcal L_r(\chi_0)
\end{align}
where~$\kappa$ is the norm character defined in \eqref{eq:3.7.10}.
\end{definition}
\par
For~$n=1,2,\cdots,k-1$, we see from the interpolation formula \eqref{eq:3.8.7} that~$\mcal L(\chi_0,\rho_r,n)$ and~$\mcal L(\chi_0,\sigma_r,n)$ are the normalized special values of~$L(\chi_0,\rho_r,n)$ and~$L(\chi_0,\sigma_0,n)$, respectively.
\begin{theorem}
With the notation as above, we have
\begin{align}\label{eq:5.3.13}
\mcal L(\chi_0,\rho_r,n) \equiv \mcal L(\chi_0,\sigma_r,n) \hspace{5mm}\textrm{(mod~$p$)}
\end{align}
for each~$n=1,2,\cdots,k-1$.
\end{theorem}
\begin{proof}
Observe first that~$\mcal L(\chi_0,\rho_r,n)$ belongs to~$\Z_p$, although the character~$\eta_r$ takes values in~$\Z_p(\mu_{p^r})$. It is essentially because of the fact that if~$\eta_r$ is of the form~$g\circ \eta'_r$ for some~$g \in \Gal(\overline \Q_p/ \Q_p)$ and a~$\overline \Q_p$-valued character~$\eta'_r$ of~$G_r$, the~$\rho_r$ is the induced representation of either~$\eta_r$ or~$\eta_r'$. Therefore, from the interpolation formula, we see that
\begin{align}
\int_{Z_r} \eta_r \widehat\kappa^n d\mcal L_r(\chi_0) = \int_{Z_r}\eta'_r \widehat\kappa^n d\mcal L_r(\chi_0)
\end{align}
for such~$\eta_r$ and~$\eta_r'$. Since~$\widehat \kappa$ and~$\widehat {\chi_0}$ take values in~$\Z_p^\times$, and~$\mcal L_r$ is defined over~$\Z_p$, it follows that~$\mcal L(\chi_0,\rho_r,n)$ is contained in~$\Z_p$. Since~$\eta_r$ is congruent to the trivial character modulo the maximal ideal of~$\Z_p(\mu_{p^r})$, the congruence \eqref{eq:5.3.13} follows.
\end{proof}
\begin{remark}
For any modular form with weight at least~$2$, the congruence similar to \eqref{eq:5.3.13} is expected to hold. Indeed, it was verified numerically in \cite{Dokchitser} for an elliptic curve with~$(p,r)=(5,1)$, and in \cite{CDLSS} for modular forms of weight at least~$4$ with~$(p,r)=(3,1)$. Here we provide a theoretical proof when the relevant modular form is of CM type. The interpolation formula of Katz~$p$-adic~$L$-function given here and the interpolation formula of \cite{CDLSS} are compared in \cite{CMFT} for~$r=1$ case, but the comparison for~$r\ge 1$ is similar. The problem with our previous approach in \cite{CMFT} was that the interpolation formula of Katz~$p$-adic~$L$-function was not the same as the prediction from the non-commutative Iwasawa theory. It is why the author had to introduce Hypothesis~4 in \cite{CMFT}. On the other hand, the numerical values obtained in \cite{CDLSS} and \cite{Dokchitser} suggest that the prediction by  the non-commutative Iwasawa theory is the correct one. By dividing out~$U_{\mathrm{alg}}$ when we define the~$p$-adic~$L$-function in Definition~\ref{def:3.8.1}, we obtain the~$p$-adic~$L$-function that is compatible with the prediction of the non-commutative Iwasawa theory.
\end{remark}
\par
Now we discuss the second family of congruences which arises from the transfer congruence of Theorem~\ref{thm:main}. In order to do so, we first study the effect of composing the transfer homomorphism to the character~$\eta_r$.
\begin{lemma}\label{lemma:5.3.17}
Let~$r\ge 2$, and let~$\eta_r$ be any character as in Definition~\ref{def:5.3.3}. Let~$\eta_{r-1}'$ be the composition of~$\eta_r\circ \ver_r$, where~$\ver_r$ is the transfer homomorphism
\begin{align}
\ver_r \colon G_{r-1}^{\mathrm{ab}} \longrightarrow G_r^\mathrm{ab}.
\end{align}
Then~$\eta_{r-1}$ is of order exactly~$p^{r-1}$, and satisfies the condition of Definition~\ref{def:5.3.3}. Namely,~$\eta_{r-1}'$ is a lift of a character
\begin{align}
\Gal(\mscr M_{r-1}(m^{\frac{1}{p^{r-1}}})/\mscr M_{r-1}) \longrightarrow \overline \Q_p^\times
\end{align}
whose order is exactly~$p^{r-1}$.
\end{lemma}
\begin{proof}
It can be verified easily using the power residue symbol. Let~$\mfrak q$ be any prime of~$\mscr M_{r}$ which is prime to~$p\cdot m$. Let~$q$ be the cardinality of the residue field of~$\mscr M_{r}$ at~$\mfrak q$. Consider the~$p^r$-th power residue symbol
\begin{align}
\left(\frac{m}{\mfrak q}\right)_{p^r} \in \mu_{p^{r}}\subset \mscr M_r
\end{align}
which is characterized by the property that 
\begin{align}
\left(\frac{a}{\mfrak q}\right)_{p^r} \equiv a^{\frac{q-1}{p^r}} \hspace{5mm}(\mathrm{mod}\,\,\mfrak q)
\end{align}
for every $a \in \mscr O_{\mscr M_r}$ relatively prime to $p$ and the residue characteristic of $\mfrak q$. If we interpret the power residue symbol as a character of~$Z_r$ using ideal theoretic Artin reciprocity map, then it is of order exactly~$p^r$, giving rise to~$\eta_r$. Now a proof of the lemma is reduced to the assertion that
\begin{align}\label{eq:5.3.22}
\prod_{\mfrak q | \mfrak l} 
\left(\frac{m}{\mfrak q}\right)_{p^r}
=
\left(\frac{m}{\mfrak l}\right)_{p^{r-1}}
\end{align}
holds for each prime~$\mfrak l$ of~$\mscr M_{r-1}$ prime to~$p\cdot m$.
\par
Let~$l$ be the cardinality of the residue field of~$\mscr M_{r-1}$ at~$\mfrak l$. We prove the above equality by dividing it into two cases. If~$\mfrak l$ splits in~$\mscr M_r$, then the residue field of any~$\mfrak q$ dividing~$\mfrak l$ will be the same as~$l$. Therefore, the left hand side of \eqref{eq:5.3.22} is equal to
\begin{align}
\prod_{\frak q | \mfrak l}m^{\frac{l-1}{p^r}} =m^{\frac{l-1}{p^{r-1}}} 
\end{align}
modulo~$\mfrak q$, so we obtain \eqref{eq:5.3.22}. If~$\mfrak l$ is inert in~$\mscr M_r$, then~$q = l^p$, and the left hand side of \eqref{eq:5.3.22} is
\begin{align}
m^{\frac{l^p-1}{p^r}} =\left(m^{\frac{l-1}{p^{r-1}}}\right)^{b}
\end{align}
where~$b$ is given by
\begin{align}
b = \frac{1+l+\cdots + l^{p-1}}{p}.
\end{align}
We claim that~$b$ is congruent to~$1$ modulo~$p^{r-1}$, which will imply \eqref{eq:5.3.22}. The key fact is that~$l$ is congruent to~$1$ modulo~$p^{r-1}$, because the residue field at~$\mfrak l$ contains all~$p^{r-1}$-th roots of unity. Once we observe this, we can write~$l=1 + a p^{r-1}$ with an integer~$a$. Applying the binomial expansion to each powers of $l$ and rearranging the sum, we obtain the expression
\begin{align}
1+l+\cdots + l^{p-1} = p + (1+2+\cdots + p-1) \cdot ap^{r-1} + a' \cdot p^{2r-2}
\end{align}
for some integer $a'$. Since~$r\ge 2$, we have~$ 2r-3 \ge r-1$, and since~$p$ is odd~$p$ divides~$1+2+\cdots+p-1$. Therefore~$b$ is congruent to~$1$ modulo~$p^{r-1}$. 
\end{proof}
As a consequence of the previous lemma, we establish the second family of congruences.
\begin{theorem}\label{thm:5.3.27}
For each~$r\ge 2$, we have the congruence
\begin{align}\label{eq:5.3.28}
\mcal L(\chi_0, \rho_r,n) \equiv \mcal L(\chi_0,\rho_{r-1},n) \hspace{5mm}(\mathrm{mod}\,\, p)
\end{align}
for any integer~$n$.
\end{theorem}
\begin{proof}
Since 
\begin{align}
\mcal L(\chi_0,\rho_r,n) \equiv \mcal L(\chi_0,\rho_{r},n') \hspace{5mm}(\mathrm{mod}\,\, p).
\end{align}
for any pair of integers~$n$ and~$n'$, the congruence is insensitive to a choice of~$n$. So we lose no generality by taking~$n=0$. The assertion of the theorem will follow from the observation that for any~$\lambda \in T_r$, where~$T_r$ is the trace ideal of~$\Lambda(G_r^{\mathrm {ab}})$, we have
\begin{align}\label{eq:5.3.30}
\int_{G_r^{\mathrm {ab}}} \eta_r  d\lambda \in p\Z_p
\end{align}
provided that we know a priori the fact that the value on the left hand side belongs to~$\Z_p$. If we take
\begin{align}
\lambda = \mcal L_r(\chi_r) - \ver_r\left( \mcal L_{r-1}(\chi_{r-1})\right)
\end{align}
then the assertion of the theorem will follow from combining \eqref{eq:5.3.30} with Lemma~\ref{lemma:5.3.17}.
\par
We now prove \eqref{eq:5.3.30}. Observe first that for any~$g\in \Gal(\mscr M_r/ \mscr M_{r-1})$,~$g$ acts on both~$G_r^{\mathrm{ab}}$ and the values of the character~$\eta_r$, and the action is equivariant in the following sense. For any such~$g$ and~$z \in G_r^{\mathrm{ab}}$, we have
\begin{align}
\eta_r(g\cdot z ) = \eta_r(z ) ^g.
\end{align}
From the characterization of~$T_r$ given in Proposition~\ref{prop:4.4.24}, it follows that \eqref{eq:5.3.30} holds.
\end{proof}
\par
\begin{remark}
We note that the congruence \eqref{eq:5.3.28} shows a genuinely non-abelian characteristic, being a congruence between~$L$-functions of different degree.
\end{remark}
\begin{remark}
The author tried to compute the~$L$-values to numerically verify the congruence \eqref{eq:5.3.28}, without any success. If we allow~$k$ to be at least~$3$, then the Euler product converges for at least one critical point, so one is able to obtain a series converging to~$L(\chi_0,\rho_r,k)$ from Euler product. However, for~$r=2$ case, which is the smallest among the relevant, the conductor of the~$L$-function will be at least~$N\cdot p^{p(2p-3)}$, where~$N$ is the conductor of~$\chi_0$. The conductors at primes dividing~$m$, and the conductor arising from ramification in~$\Q_p(m^{1/p^r})/\Q_p$ need to be multiplied to this to obtain the actual conductor, which is bound to be very large. If we look at the table of CM elliptic curves, the smallest one for which~$3$ is good ordinary is the unique CM elliptic curve of conductor~$121$. It follows that the conductor of~$L(\chi_0,\rho_2,s)$ is at least~$121\cdot 3^9$, which is greater than two million. It seems to the author that the numerical verification of \eqref{eq:5.3.28} for modular forms, either of CM type or not, is a challenge from the computational point of view.
\end{remark}
\subsection{The interpolation formula for $\mcal L(G)$}\label{subsection:5.4}
In this subsection, we discuss the interpolation formula for $\mcal L(G)$ and compare it with the prediction of non-commutative Iwasawa theory. First we recall the meaning of the value of an element $\Theta \in K_1(\Lambda(G)_S)$ at a representation of $G$ on a finite dimensional $\Q_p$-vector space. Throughout, $\rho$ will denote a representation of $G$ of the form
\begin{align}
\rho = \rho_r \otimes \xi
\end{align}
where the $\rho_r$ is the irreducible representation introduced in \eqref{eq:5.3.2}, and $\xi$ is a finite order character of $G^{\mathrm{ab}}$. Recall that $\kappa$ was defined to be the adelic norm character, whose $p$-adic avatar was denoted by $\widehat \kappa$. In the paragraphs following Lemma~3.3 of \cite{CFKSV}, it is defined how to define, for each integer $n$, the value of $\Theta$ at the representation $\rho\widehat\kappa^n$. We denote the value as
\begin{align}
\int_G\rho\widehat\kappa^n d\Theta
\end{align}
which is meant to give hints to the interpretation of $\Theta$ as a measure in the commutative case.
\begin{remark}
In the literature, one often writes a representation of $G$ as $\rho \phi^n$, where $\phi$ is the restriction of $\widehat\kappa$ to the cyclotomic $\Z_p$-extension of $\Q$. In fact, as one can see easily from our description of irreducible representation of $G$ given after Definition~\ref{def:5.3.3}, these make no difference when $r\ge 1$.
\end{remark}
\par
As one can expect from our construction of $\mcal L(G)$, the interpolation formula of $\mcal L(G)$ is determined by the interpolation formula of $\mcal L_r$ given in \eqref{eq:3.8.7}, and our intention is to provide an interpolation formula in terms of intrinsic invariants of $\rho$ and $n$, which do not involve the intermediates fields. Unfortunately, we will have to leave some factors involve the intermediate fields, and we will contend ourselves by comparing the interpolating factors with those predicted by the non-commutative Iwasawa theory. To make the computation simpler, we will assume:
\begin{enumerate}
\labitem{Inert}{Inert} Every prime divisor of $m$ is inert in $\mscr M_0$.
\end{enumerate}
\par
Recall that we write $\rho$ in the form $\rho = \rho_r \otimes \xi$. We fix a character $\eta_r$ of $G_r^{\mathrm{ab}}$ whose induced representation of $\rho_r$. Also, we denote by $\xi_r$ the restriction of $\xi$ to $G_r^{\mathrm{ab}}$. We will identify the adelic characters with Galois characters using the Artin reciprocity maps in this subsection.
\par
Let us consider the periods $\Omega_\infty$ appearing in the interpolation formula \eqref{eq:3.8.7}. For each $r$, let $\Sigma^r$ be the CM type of $\mscr M_r$. Then $\Omega_\infty$ is a tuple $(\Omega_{\infty,\sigma})_{\sigma \in \Sigma^r}$ with $\Omega_{\infty,\sigma} \in \C^\times$ for each $\sigma$. To emphasize the dependence on $\mscr M_r$, we often write $\Omega_\infty(\mscr M_r)$ instead of $\Omega_\infty$. Define $\Omega \in \C^\times$ to be
\begin{align}\label{eq:5.4.1}
\Omega  = \delta_0 \cdot \Omega_\infty(\mscr M_0).
\end{align}
The key observation is that we have
\begin{align}\label{eq:5.4.2}
\sigma(\delta_r) \cdot \Omega_{\infty,\sigma}(\mscr M_r) = \Omega
\end{align}
for every $r \ge 0$ and any $\sigma \in \Sigma^r$. Precisely speaking, we can choose the periods so that \eqref{eq:5.4.2} hods, as it was proved in Proposition~5.1 of \cite{CMFT}. For the characters of the form
\begin{align}\label{eq:5.4.7}
\chi = \chi_r \kappa_r^n\eta_r\xi_r
\end{align}
where $\chi_r$ is constructed by composing the norm map to $\chi_0$, the values of $n$ for which $\chi$ belongs the interpolation range are given by
\begin{align}
n=0,1,2,\cdots,k-1,
\end{align}
and the infinity type of $\chi$ is given by
\begin{align}
(k-2n)\Sigma + (1-c)n\Sigma
\end{align}
respectively. Thus we can rewrite the power of $\Omega_\infty$ appearing in the denominator in \eqref{eq:3.8.7} as
\begin{align}
\Omega_\infty^{(k-2n)\Sigma + 2n\Sigma} = \Omega^{k\Sigma} \delta_r^{-k\Sigma}.
\end{align}

\par
Next we consider the epsilon factors. For the moment, let us view $\rho$ as a representation on a complex vector space, using the embeddings $i_p$ and $i_\infty$. Then, for the unique prime $\mfrak p$ in the $p$-adic CM type of $\mscr M_r$, we have
\begin{align}
\epsilon_{\mfrak p}(0,\eta_r\xi_r\chi_r\kappa_r^n,\psi_{\mfrak p}) = \kappa^n\chi_0(C_{p_w}(\rho))\epsilon_{\mfrak p}(0,\eta_r\xi_r, \psi_{\mfrak p})
\end{align}
where $C_{p_w}(\rho)$ is the $p_w$-part of the conductor of $\rho$. Define $\alpha$ to be
\begin{align}
\alpha = (\chi_r)_\mfrak p (\varpi_\mfrak p)
\end{align}
where $\mfrak p$ is as above and $\varpi_\mfrak p$ is the uniformizer at $\mfrak p$. If we let $e_p(\rho)$ be the $p$-adic valuation of $C_{p_w}(\rho)$, then we have
\begin{align}
\kappa^n\chi_0(C_{p_w}(\rho)) = \left( \frac{\alpha}{p^n}\right)^{e_p(\rho)}.
\end{align}
\par
Next we consider the ratio
\begin{align}
\frac{L(0,\chi_w)}{L(1, \chi_w^{-1})} 
\end{align}
appearing in \eqref{eq:3.8.7}. Let $p_w$ be the unique place in the $p$-adic CM type of $\mscr M_0$, and define the polynomial $P_{p_w}(\rho,X)$ to be the polynomial in $X$ such that
\begin{align}
P_{p_w}\left(\rho,p^{-s}\right)^{-1}
\end{align}
is the Euler factor of the Artin $L$-function attached to $\rho$. If $\rho$ is an unramified linear character, then $P_{p_w}(\rho,X) = 1- \rho(\phi_{p_w})X$, where $\phi_{p_w}$ is the (geometric) Frobenius element in the Galois group. By the elementary argument given in the proof of Lemma~4.2 of \cite{CMFT}, we have
\begin{align}
L(0,\chi_w) = P_{p_w}\left(\rho,\frac{\alpha}{p^n}\right)^{-1}.
\end{align}
Similarly, letting $\rho^\wedge$ be the contragredient of $\rho$, we have
\begin{align}
L(1,\chi_w^{-1}) = P_{p_w}\left(\rho^\wedge,\frac{p^{n-1}}{\alpha}\right)^{-1}.
\end{align}
\par
Collecting our formulae, for $\chi$ of the form \eqref{eq:5.4.7}, the right hand side of \eqref{eq:3.8.7} becomes
\begin{align}
L^{p\mfrak C}(\chi,0)\cdot\mcal C_\infty(\chi)
\cdot
\frac{\chi_\mfrak p(2\delta)\delta^{k\Sigma}}{ (\mathrm{Im}(\delta))^{n\Sigma}\epsilon_{\mfrak p }(0,\eta_r\xi_r,\psi_\mfrak p)\sqrt{|D_\mscr F|_\R}}
\cdot
\frac
{\pi^{n\Sigma}\Gamma_\Sigma((k-n)\Sigma)}
{ \Omega^{k\Sigma }}
\cdot \left( \frac{p^n}{\alpha}\right)^{e_p(\rho)}
\cdot
\frac{P_{p_w}(\rho^\wedge,\frac{ p^{n-1}}{\alpha})}{P_{p_w}(\rho,\frac{\alpha}{p^{n}})}
\end{align}
because $\mscr M_r/\mscr F_r$ is unramified everywhere and $\mfrak C_+=1$ by \eqref{Inert}. It becomes
\begin{align}
L^{p\mfrak C}(\chi,0)\cdot\mcal C_\infty(\chi)
\cdot
\frac{\chi_\mfrak p(2\delta)\delta^{k\Sigma}p^{e_p(\rho)}}{ (\mathrm{Im}(\delta))^{n\Sigma}\epsilon_{\mfrak p }(0,\eta_r\xi_r,\psi_\mfrak p)\sqrt{|D_\mscr F|_\R}}
\cdot
\frac
{\pi^{n\Sigma}\Gamma_\Sigma((k-n)\Sigma)}
{ \Omega^{k\Sigma }}
\cdot \left( \frac{p^{n-1}}{\alpha}\right)^{e_p(\rho)}
\cdot
\frac{P_{p_w}(\rho^\wedge,\frac{ p^{n-1}}{\alpha})}{P_{p_w}(\rho,\frac{\alpha}{p^{n}})}
\end{align}
after a slight modification in the second and fourth factor. The $L$-function can be written, using Artin formalism, as
\begin{align}
L^{p\mfrak C}(\chi,0) = L^{pC}(\chi_0,\rho,k-n),
\end{align}
where $L(\chi_0,\rho,s)$ is the $L$-function of $\chi_0$ over $\mscr M_0$ twisted by $\rho$, and $C$ is the norm of $\mfrak C$ to $\mscr M_0$, and the superscript $pC$ means that the Euler factors at primes dividing $pC$ are removed. The first factor $\mcal C_\infty(\chi)$ is simply $1-\alpha p^{2k}$, which is in fact independent of $\rho$. 
\par
The resulting interpolation formula still involves intermediate fields, but the author is not able to give other form of interpolation formula, which looks substantially nicer than this, and would like to give some remarks on the interpolation factors. The last two factors are what appear in the conjectural interpolation formula of non-commutative $p$-adic $L$-function given in \cite{CFKSV}. The first factor, which is a $p$-adic unit, does not appear therein, but appears in the interpolation formula of Hida's three variable Rankin-Selberg $p$-adic $L$-function, which is given in Theorem~1, $\mathsection$~10.4 of \cite{Hida Elementary}. This factor is natural since we worked with all $k\ge 1$, which gives rise to a Hida family. For the second factor, It would be nice to obtain a simpler formula it, but the author does not know a better one. The third factor consists of the expected gamma factor and the the archimedean period, and $\Omega$ can be identified with the Neron period of an elliptic curve up to a $p$-adic unit. Indeed, one can see it from \eqref{eq:5.4.1}, since $\Omega_\infty(\mscr M_0)$ is by definition a Neron period of an elliptic curve, and $\delta_0$ is a $p$-adic unit.

\section*{Acknowledgement}
The author would like to thank his advisor John Coates for giving him warm encouragement throughtout the project. He also would like to thank T. Bouganis, R. Greenberg, M. Hsieh, and M. Kakde for helpful conversations. This work was partially supported by NRF(National Research Foundation of Korea) Grant funded by the Korean Government(NRF 2012 Fostering Core Leaders of the Future Basic Science Program). This work was partially supported by Priority Research Centers Program through the National Research Foundation of Korea(NRF) funded by the Ministry of Education, Science and Technology(2013053914).






\end{document}